\documentclass[12pt]{amsart}
\usepackage[letterpaper,margin=1.1in]{geometry}
\usepackage{amsmath}
\usepackage{amssymb}
\usepackage{curves}
\usepackage[french,english]{babel}
\usepackage{epic}
\usepackage{rotating}
\usepackage{epsf}
\usepackage[scanall]{psfrag}
\usepackage{graphicx}
\usepackage{esint} 
\usepackage{enumerate}
\usepackage[nobysame,alphabetic]{amsrefs}

\usepackage{hyperref} 

\numberwithin{equation}{section}

\newtheorem{theorem}{Theorem}[section]
\newtheorem{proposition}[theorem]{Proposition}
\newtheorem{corollary}{Corollary}[section]
\newtheorem{lemma}{Lemma}[section]
\newtheorem{Claim}{Claim}
\newtheorem{defn}{Definition}[section]

\theoremstyle{remark}
\newtheorem{remarks}{Remarks}[section]
\newtheorem{remark}{Remark}[section]

\renewcommand{\d}{\partial} 
 
\newcommand{\1}{{\bf 1}} 
\newcommand{\R}{\mathbb{R}}
\newcommand{\NN}{\mathbb{N}}
\newcommand{\ZZ}{\mathbb{Z}}

\renewcommand{\SS}{\mathbb {S}}

\newcommand{\s}{\sigma}
\renewcommand{\H}{\mathcal H}
\newcommand\res{\hbox{ {\vrule height.22cm}{\leaders\hrule\hskip.2cm} } }

\newcommand{\bs}{\backslash}

\renewcommand{\O}{\Omega}
\newcommand{\p}{\partial}
\newcommand{\loc}{\mathrm{loc}}

\newcommand{\dist}{\mathrm{dist}\,}
\newcommand{\diam}{\mathrm{diam}\,}
\newcommand{\vecn}{\overrightarrow{n}\,}
\newcommand{\Div}{\mathrm{div}}

\renewcommand{\div}{\,\mathrm{div}\,}

\newcommand{\sm}{\setminus}

\renewcommand{\i}{\subset}
\newcommand{\wt}{\widetilde}

\newcommand{\ms}{\medskip}

\usepackage{xcolor}

\begin{document}

\title{Free boundary regularity for almost-minimizers}
\author{Guy David}
\author{Max Engelstein}
\author{Tatiana Toro}
  \thanks{G.\ David was partially supported by the Institut Universitaire de France, the ANR, programme blanc GEOMETRYA, ANR-12-BS01-0014 and the Simons Collaborations in MPS Grant 601941 GD. M.\ Engelstein was partially supported by an NSF Graduate Research Fellowship, NSF DGE 1144082, the University of Chicago RTG grant DMS 1246999, a NSF postdoctoral fellowship, NSF DMS 1703306 and by David Jerison?s grant NSF DMS 1500771. T.\ Toro was partially supported by a Guggenheim fellowship, NSF grant DMS-1361823, by the Robert R. \& Elaine F. Phelps Professorship in Mathematics and by the Craig McKibben \& Sarah Merner Professorship in Mathematics.   This material is based upon work supported by the National Science Foundation under Grant No. DMS-1440140 while the authors were in residence at the Mathematical Sciences Research Institute in Berkeley, California, during the Spring 2017 semester.}
\address{Equipe d'Analyse Harmonique\\
Universit\'e Paris-Sud\\
Batiment 425 \\
91405 Orsay Cedex, France}
\email{Guy.David@math.u-psud.fr}
\address{Department of Mathematics\\Massachusetts Institute of Technology\\Cambridge, MA, 02139-4307 }
\email{maxe@mit.edu}
\address{Department of Mathematics\\ University of Washington\\ Box 354350\\ Seattle, WA 98195-4350}
\email{toro@uw.edu}
\date{\today}

\maketitle

\vspace{-.7cm}
\begin{abstract}
In this paper we study the free boundary regularity for almost-minimizers of the functional 
\begin{equation*}
J(u)=\int_\O |\nabla u(x)|^2
+q^2_+(x)\chi_{\{u>0\}}(x) +q^2_-(x)\chi_{\{u<0\}}(x)\ dx
\end{equation*}
where $q_\pm \in L^\infty(\O)$.
Almost-minimizers satisfy a variational inequality but not a PDE or a monotonicity formula 
the way minimizers do (see \cite{AC}, \cite{ACF}, \cite{CJK}, \cite{W}). 
Nevertheless, using a novel argument which brings together
tools from potential theory and geometric measure theory, we succeed in proving that, 
under a non-degeneracy assumption on $q_\pm$, the free boundary is uniformly rectifiable.  Furthermore, when $q_-\equiv 0$, and $q_+$ is H\"older continuous we show that the free boundary is almost-everywhere given as the graph of a $C^{1,\alpha}$ function (thus extending the results of \cite{AC} to almost-minimizers). 
\end{abstract}

\selectlanguage{french}
\begin{abstract}
On \'etudie la r\'egularit\'e des fronti\`ere libres des
presque-minimiseurs de la fonctionnelle
\begin{equation*}
J(u)=\int_\O |\nabla u(x)|^2
+q^2_+(x)\chi_{\{u>0\}}(x) +q^2_-(x)\chi_{\{u<0\}}(x)\ dx,
\end{equation*}
o\`u $q_\pm \in L^\infty(\O)$.
Les presque-minimiseurs v\'erifient une in\'egalit\'e variationnelle, mais pas une EDP
ni une formule de monotonie comme le font les minimiseurs
(voir \cite{AC}, \cite{ACF}, \cite{CJK}, \cite{W}).
N\'eanmoins, gr\^ace \`a un argument nouveau qui utilise des outils de th\'eorie du potentiel et de th\'eorie g\'eom\'etrique de la mesure, on arrive \`a d\'emontrer que, sous une hypoth\`ese de non
d\'eg\'en\'erescence sur $q_\pm$,
leur fronti\`ere libre est uniform\'ement rectifiable. De plus, quand
$q_-\equiv 0$ et $q_+$ est
H\"old\'erienne, on montre que la fronti\`ere libre coincide dans un
voisinage de presque tout point
avec un graphe de fonction $C^{1,\alpha}$, ce qui \'etend les r\'esultats
de \cite{AC} aux presque-minimiseurs.
\end{abstract}
\selectlanguage{english}

\tableofcontents
\section{Introduction}\label{S1}

In \cite{DT} the first and third authors studied almost-minimizers with free boundary. 
They proved that almost-minimizers for the type of functionals considered by Alt and Caffarelli \cite{AC} and
Alt, Caffarelli and Friedman \cite{ACF} are Lipschitz. 
The almost-minimizing property can be used to describe minimizers of variants of the functionals
above, which include additional terms or perturbations that have a smaller contribution at small scales.
We think either of perturbations whose explicit form is not so important, or perturbations coming from noise.
The flexibility of the set up allows one to deal with a broader spectrum of questions, or incorporate small errors and randomness. 

The methods used in \cite{DT} do not provide any information about either the size or the structure of the free boundaries for almost-minimizers.
We address this question in this paper, and, in particular, we show that the free boundary is uniformly rectifiable. This requires a novel argument which brings together
tools from potential theory and geometric measure theory. It provides a new approach to estimating the size and proving the rectifiability of a free boundary.
In the one phase case, that is when $q_- \equiv 0$, $q_+$ is H\"older continuous and the almost-minimizer is non-negative, we also prove that, at most points, the free boundary is given by the graph of a $C^1$ function.
Almost-minimizers were first considered in a geometric context, when Almgren \cite{Alm} studied almost-area minimizing surfaces. More recently, almost-minimizers for the functionals we consider here were introduced in \cite{DT} and further studied by de Queiroz and Tavares \cite{dQT}(who focused on the regularity of almost-minimizers for semi-linear and variable coefficient analogues of the Alt-Caffarelli and Alt-Caffarelli-Friedman functionals).

 The theory of almost-minimal surfaces has found applications to the existence and regularity of isoperimetric partitions \cite{AlmPartition}. The idea of looking at small perturbations of 
 minimizers is inherent in the study of stability questions in shape optimization and quantitative inequalities (see \cite{CMM}, \cite{KM} and \cite{CFM} for example for some of the most recent developments in this area).
 
 It was observed in \cite{ACS} that the functionals studied in this paper can be used to prove regularity for minimization problems involving the Dirichlet energy and a volume constraint. Therefore, in analogy with the almost-area minimizers, almost-minimizers to the functional in \eqref{eqn1.1} (and related functionals) have appeared in the study shape-optimization for functions of the Dirichlet eigenvalues of the Laplacian (see, e.g. \cite{MTV}), eigenvalue partition problems (see, e.g. \cite{STV}) and the stability for the Faber-Krahn inequality (see, e.g. \cite{FKstable}).  
 
 Let us also point out that the Alt-Caffarelli-type functionals considered here are the prototypical example of a free boundary problem in which the energy is non-convex. One interesting aspect of studying almost-minimizers is that they allow us to disentangle the behavior of minimizers from that of weak solutions (which can be thought of as critical points defined by the Euler-Lagrange equation). For functionals with convex energies (e.g. obstacle type problems), every critical point is a minimizer. By considering a non-convex functional this distinction becomes more salient (and interesting).

\ms
We consider a bounded domain $\O\subset\R^n$, $n \geq 2$, 
and study the the functional 
\begin{equation}\label{eqn1.1}
J(u)=\int_\O |\nabla u(x)|^2
+q^2_+(x)\chi_{\{u>0\}}(x) +q^2_-(x)\chi_{\{u<0\}}(x)\ dx,
\end{equation}
where $q_\pm \in L^\infty(\O)$ are two bounded real valued functions. 
We are especially interested in the  properties of the two sets
\begin{equation}\label{eqn1.2}
\Gamma^\pm(u)= \Omega \cap \p\{x\in \Omega \, ; \, \pm u(x)>0\},
\end{equation}
when $u$ is an almost-minimizer for $J$.

In \cite{AC}, Alt and Caffarelli proved free boundary regularity results
for minimizers in the following context. 
Let $\O \i \R^n$ be a bounded  Lipschitz domain and $q_+ \in L^\infty(\O)$ 
be given, set 
\begin{equation} \label{eqn1.3}
K_+(\O) =\left\{u\in L^1_{\loc}(\O) \, ; \, u(x) \geq 0 \mbox{ almost everywhere on } \O
\mbox{ and } \nabla u\in L^2(\O) \right\}
\end{equation} 
and 
\begin{equation} \label{eqn1.4}
J^+(u)=\int_\O|\nabla u|^2+q^2_+(x)\chi_{\{u>0\}}\ dx
\end{equation}
for $u\in K_+(\O)$, and let $u_{0} \in K_{+}(\O)$ be given, 
with $J^+(u_0)<\infty$.
They proved the existence of a function
$u\in K_+(\O)$ that minimizes $J^+$ among functions of $K_+(\O)$ such that 
\begin{equation} \label{eqn1.5}
u = u_{0} \mbox{ on } \p \O.
\end{equation}

Alt and Caffarelli also showed that the minimizers are Lipschitz-continuous up to the free boundary 
$\Gamma^+(u)$, and that if $q_+$ is H\"older-continuous and bounded away from zero, then 
\begin{equation} \label{eqn1.6}
\Gamma^+(u) =\p_\ast\{u>0\}\cup E, 
\end{equation}
where $\mathcal{H}^{n-1}(E)=0$ and $\p_\ast\{u>0\}$ 
is the reduced boundary of $\{x\in \Omega \, ; \, u(x)>0\}$ in $\Omega$.
They proved that $\p_\ast\{u>0\}$ locally coincides with a $C^{1,\alpha}$ submanifold of dimension $n-1$.

Later on, Alt, Caffarelli, and Friedman \cite{ACF} showed  
that if $\O$ is a bounded Lipschitz domain, $q_\pm\in L^\infty(\O)$, 
\begin{equation} \label{eqn1.7A}
K(\O) = \left\{ u\in L^1_\loc(\O) \, ; \, \nabla u\in L^2(\O) \right\}
\end{equation}
and $u_0\in K(\O)$, then there exists $u\in K(\O)$ that minimizes $J(u)$ under the constraint (\ref{eqn1.5}).
(See the proof of Theorem 1.1 in \cite{ACF}).
In fact, in \cite{ACF} they consider a slightly different functional,
for which they show that the minimizers are Lipschitz. 
They also prove optimal regularity results for the free boundary when $n=2$, 
and make important strides towards the higher dimensional cases.  
Later papers by \cite{CJK}, \cite{DeJ} and \cite{W}
present a more complete picture of the structure of the free boundary in higher dimensions.

\medskip
In this paper we study the regularity properties of the free boundary of
almost-minimizers for $J^+$ and $J$. We consider a domain $\O \i \R^n$, with $n\ge 2$,
and two functions $q_\pm \in L^\infty(\O)$. 
In the case of $J^+$ we assume that $q_-$ is identically equal to zero.
Set
\begin{equation} \label{eqn1.7}
K_{\loc}(\O) = \left\{u\in L^1_{\loc}(\O) \, ; \nabla u\in L^2(B(x,r))
\mbox{ for every open ball } B(x,r) \i \O \right\},
\end{equation}
\begin{equation} \label{eqn1.8}
K_{\loc}^+(\O) = \left\{u\in K_{\loc}(\O) \, ; u(x) \geq 0 
\mbox{ almost everywhere on } \O\right\}, 
\end{equation}
and let constants $\kappa \in (0,+\infty)$ and $\alpha \in (0,1]$
be given.

We say that $u$ is an {\it almost-minimizer} for $J^+$ in $\O$ 
(with constant $\kappa$ and exponent $\alpha$) if
$u\in K_{\loc}^+(\O)$ and
\begin{equation}\label{eqn1.9}
J^+_{x,r}(u)\le (1+\kappa r^\alpha)J^+_{x,r}(v)
\end{equation}
for every ball $B(x,r)$ such that $\overline B(x,r) \i \O$ and every  
$v\in L^1(B(x,r))$ such that  $\nabla v\in L^2(B(x,r))$ 
and $v=u$ on $\p B(x,r)$, where
\begin{equation}\label{eqn1.10}
J^+_{x,r}(v)=\int_{B(x,r)}|\nabla v|^2+q^2_+ \, \chi_{\{v>0\}}.
\end{equation}
Here, when we say that $v=u$ on $\p B(x,r)$, we mean
that they have the same trace on $\p B(x,r)$. 
Notice that if we set $v^+ = \max(v,0)$, then $v^+=u$ on $\p B(x,r)$
and $J^+(v^+) \leq J^+(v)$, so we can restrict ourselves to competitors $v\in K^+_{\loc}(\O)$. 
In this case we only care about
\begin{equation}\label{eqn1.10A}
\Gamma^+(u)= \Omega \cap \partial\{u>0\}. 
\end{equation}

Similarly, we say that $u$ is an almost-minimizer for $J$ in $\O$ 
if $u\in K_{\loc}(\O)$ and
\begin{equation}\label{eqn1.11}
J_{x,r}(u)\le (1+\kappa r^\alpha)J_{x,r}(v)
\end{equation}
for every ball $B(x,r)$ with $\overline B(x,r) \i \O$ and every  $v\in L^1(B(x,r))$ such that  
$\nabla v\in L^2(B(x,r))$ and $v=u$ on $\p B(x,r)$, where
\begin{equation}\label{eqn1.12}
J_{x,r}(v)=\int_{B(x,r)}|\nabla v|^2
+q^2_+ \, \chi_{\{v>0\}}
+ q^2_- \, \chi_{\{v<0\}} .
\end{equation}
 In this case we are interested in both sets $\Gamma^\pm(u)$ of \eqref{eqn1.2}.

\ms
In both cases we restrict our attention to 
$U = \big\{ x\in \Omega \, ; \, u(x) > 0 \big\}$ and $\Gamma^+(u) = \d U \cap \Omega$.
We assume that $q_+$ and $q_-$ are bounded and continuous on $\Omega$, that $q_+ \geq c_0 > 0$
on $\Omega$, and that either $q_- \geq c_0 > 0$ or $0 \leq q_- \leq q_+$ on $\Omega$, 
and we prove that $U$ is locally NTA (Non-Tangentially Accessible) in $\Omega$ (see Definition \ref{d2.3}
and Theorem \ref{t2.3}), and $\Gamma^+(u)$ is locally Ahlfors-regular 
and uniformly rectifiable; see Theorems \ref{t4.2} and \ref{t4.3}. 
The most challenging part of the argument is the construction of an Ahlfors-regular measure supported on $\Gamma^+(u)$.
It should be mentioned that, a priori, it was not even clear
that $\Gamma^+(u)$ should be $(n-1)$-dimensional.

\ms

For almost-minimizers of $J^+$, we can continue the study a little bit further, 
and generalize regularity results from \cite{AC}. 
We identify an open set ${\mathcal R} \subset \Gamma^+ (u)$ of regular points (see Definition \ref{defn:flatpoints}).
$\mathcal R$ has full measure in $\Gamma^+(u)$, and it is locally a $C^{1+\beta}$ sub-manifold provided $q_+ > c_0$ is H\"older-continuous (see Theorem \ref{main-reg-theo}). 

\begin{remark}\label{r:DeSilvaSavin}
While this paper was being reviewed, D. De Silva and O. Savin reprove in \cite{DeS} many of the results in \cite{DT} and in this paper using different methods. More precisely, the paper \cite{DeS} is the continuation of a program, began in \cite{quasi}, in which a viscosity approach is applied to almost-minimizers of several variational problems. The idea is that while almost-minimizers may not satisfy any pointwise equation, they exhibit what De Silva-Savin call ``two scale behavior". This allows them to prove a Harnack inequality and apply viscosity methods to prove Lipschitz continuity of almost-minimizers (as in \cite{DT}) and $\mathcal H^{n-1}$-almost everywhere $C^{1,\alpha}$-regularity of the free boundary (as in this paper). This approach is very interesting and we hope to investigate it further in the future. 
\end{remark}

\ms
We briefly outline the structure of the paper. In Section \ref{global} we prove that the positivity set $U$ of an almost-minimizer, $u$, is a locally NTA domain (see Theorem \ref{t2.3}). This is done via a compactness argument. Along the way we use the Alt-Caffarelli-Friedman monotonicity formula to show that the set where a Lipschitz global minimizer 
is positive is a connected set. We note that the recent preprints \cite{CSY} and \cite{MTV} prove that the positivity set of a \emph{minimizer} to the functional, \eqref{eqn1.4}, is an NTA domain (both papers cover the vectorial case) (see also \cite{KL}). Let us remark that these results, published while this paper was in preparation, are proven by different methods and neither imply nor are implied by our 
Theorem~\ref{t2.3}. 

\ms

In Section \ref{harmonic-functions}, we construct local subharmonic competitors, $h_{x_0, r}$. They will be the main tool in the subsequent arguments. Essentially, at every point $x_0 \in \Gamma^+(u)$ and every scale $r > 0$, we construct a function, $h_{x_0, r}$, which is subharmonic in $B(x_0,r)$, satisfies $h_{x_0,r}=0$ when $u=0$,
is harmonic in $B(x_0,r)\cap \{u >0\}$ and has the same trace as $u$ on $\partial (B(x_0,r)\cap \{u > 0\})$. In particular, we use the NTA properties of $\{u > 0\}$ to show that $h_{x_0,r}$ and $u$ are comparable up to $\Gamma^+(u)$ (Theorem \ref{t3.1}) with an error which is a power of $r$. This allows us to use $h_{x_0,r}$ to study the free boundary $\Gamma^+(u)$. 

\ms

In Section \ref{unif-rect} we use $h_{x_0,r}$ to show that the harmonic measure on $\Gamma^+(u)$ is Ahlfors-regular (Theorem \ref{t4.1}). A consequence of this is that $\Gamma^+(u)$ is uniformly rectifiable 
(and even,  
contains big pieces of Lipschitz graphs at every point and every scale), see Theorem \ref{t4.2}. In Section \ref{WeissMF}, we study a monotonicity formula due to Weiss \cite{W} and show that it is ``almost-monotone" for almost-minimizers (Theorem \ref{thm:monotonicityone}).

In Section \ref{Consequences}, we list several consequences of the monotonicity formula. Most significantly, we are able to measure how ``close" an almost-minimizer is to a half-plane solution by the value of the monotone quantity at small scales (Proposition \ref{flatpointsregularpoints}). We end the section by showing that at most points in the free boundary, $\Gamma^+(u)$, 
there is a well defined notion of normal derivative (and full gradient) for both the almost-minimizer, $u$, and the competitors, $h_{x_0, r}$. Finally, at small enough scales, these derivatives are comparable to one another, with an error that gets small with the scale (see Corollary \ref{cor6.1A}). 

In Section \ref{bdryreg} we finish the argument, 
modulo some computations on harmonic functions that we leave for Section \ref{appendix}.
We show that if $u$ is close to a half-plane solution (see Definition \ref{defn:flatpoint} for what ``close" means) in a ball, then an appropriately chosen $h_{x_0, r}$ is also close (Lemma \ref{uflat-hflat}). A quantified version of the ``improved flatness" argument of Alt-Caffarelli \cite{AC}, tells us that $h_{x_0, r}$ is even closer to a half-plane solution on a slightly smaller ball (see Corollary \ref{main-cor} and the rest of Section \ref{appendix}  for this quantified ``improved flatness" argument). We are then able to transfer the improved closeness of $h_{x_0,r}$ to $u$ on this smaller ball and iterate to conclude regularity of the free boundary, Theorem \ref{main-reg-theo}

In Section \ref{dimensionofSS}, we use the results of Section \ref{bdryreg} to prove bounds on the 
Hausdorff dimension of the singular set $\Gamma^+ \sm \mathcal R$ of the free boundary, for 
almost-minimizers to the one-phase problem. See Theorem \ref{thm:dimsingularset}.

\ms
\noindent {\bf Acknowledgements:} All the authors would like an anonymous referee whose careful reading and comments greatly improved this manuscript. The first two authors would like to express their gratitude to the Mathematics Department at the University of Washington where part of this work was carried forward.
This project was finished while the authors were visiting MSRI in Spring 2017, the authors would like to thank MSRI for its hospitality. The third author would also like to thank the Mathematics Department at UC Berkeley.

\section{Global minimizers and quantified connectedness}\label{global}

 As in \cite{CJK}, one of the key steps for the regularity of the free boundary for almost-minimizers is to get some control on global minimizers. 
 
In this section we show that if $u$ is a global minimizer, then $\{ u > 0 \}$ is connected, 
and use this to prove quantitative connectedness properties for almost-minimizers. While the methods are different the results concerning the connectivity of $\{ u > 0 \}$
are similar to those obtained 
in \cite{ACS} and \cite{DeJ1}.

 We first define global minimizers.
 Let $\lambda_\pm$ be constants such that 
 $0 \leq \lambda_- \leq \lambda_+<\infty$. 
 We think about the functionals $J$ and $J^+$ as defined by
\begin{equation}\label{eqn2.1}
J(v)=\int |\nabla v|^2
+\lambda_+^2 \, \chi_{\{v>0\}}
+ \lambda_- ^2\, \chi_{\{v<0\}}
\end{equation}
and (for nonnegative functions $v$)
\begin{equation}\label{eqn2.2a} 
J^+(v)=\int |\nabla v|^2
+\lambda_+^2 \, \chi_{\{v>0\}},
\end{equation}
but since both integrals on $\R^n$ are probably infinite, we only define the local
versions
 \begin{equation}\label{eqn2.3}
J_{x,r}(v)=\int_{B(x,r)} |\nabla v|^2
+\lambda_+^2 \, \chi_{\{v>0\}}
+ \lambda_-^2 \, \chi_{\{v<0\}}
\end{equation}
and 
\begin{equation}\label{eqn2.4} 
J_{x,r}^+(v)=\int_{B(x,r)} |\nabla v|^2
+\lambda_+^2 \, \chi_{\{v>0\}},
\end{equation}
where $B(x,r)$ is a ball in $\R^n$ and $v$ is any function of $L^1(B(x,r))$
such that $\nabla v\in L^2(B(x,r))$.
For $J^+$, we may also restrict our attention to nonnegative functions $v$, but this will not matter.

\begin{defn}\label{d2.1}
 We say that $u\in K_{\loc}(\R^n)$ is a global minimizer for $J$ if
\begin{equation}\label{eqn2.2}
J_{x,r}(u)\le J_{x,r}(v)
\end{equation}
for every ball $B(x,r)$ and every  $v\in L^1(B(x,r))$ such that
$\nabla v\in L^2(B(x,r))$ and $v=u$ on $\p B(x,r)$.
\end{defn}

\begin{defn}\label{d2.2}
 We say that $u\in K^+_{\loc}(\R^n)$ is a global minimizer for  $J^+$ if 
\begin{equation}\label{eqn2.5}
J^+_{x,r}(u)\le J^+_{x,r}(v)
\end{equation}
for every ball $B(x,r)$ and every nonnegative function $v\in L^1(B(x,r))$ such that  
$\nabla v\in L^2(B(x,r))$ and $v=u$ on $\p B(x,r)$.
\end{defn}

If we did not restrict to nonnegative $v\in L^1(B(x,r))$, we would get the same definition,
because the positive part $v^+$ of $v$ has the same trace as $u$ and is at least 
as good as $v$. The main result of this section is the following.

 \begin{theorem}\label{t2.1}
Let $v$ be a Lipschitz global minimizer for $J$ or $J^+$.
Then the sets $\{x\in \R^n \, ; \, v(x)>0\}$ and $\{x\in \R^n \, ; \, v(x)<0\}$) are 
(empty or) connected. 
\end{theorem}

In general, global minimizers are merely locally Lipschitz (see \cite{AC} and  \cite{ACF}) 
thus the hypothesis that $v$ is Lipschitz in all of $\R^n$ is not redundant. However, the uniform limit of almost-minimizers,
which are the objects to which we will apply this result, are global minimizers which are Lipschitz in all of $\R^n$ (see Theorem 9.1 in \cite{DT}).  General uniform limits of almost-minimizers, as opposed to blowups, are not necessarily one-homogenous, which complicates the proof. However, by the maximum principle, each non-empty component of $\{\pm v > 0\}$ is unbounded. This combined with control at infinity given to us by the monotonicity formula of \cite{ACF} will allow us to rule out multiple components.

\begin{proof} 
Denote by $M$ the Lipschitz constant for $v$.
Suppose for instance that $\{x\in \R^n \, ; \, v>0\}$ is not connected, and let $U$ and $V$ be different connected components of this set. We consider the functions $f = \1_U v$ and $g = \1_V v$, 
which are both nonnegative and $M$-Lipschitz (because $v=0$ on $\d U$ and $\d V$). 
The product $fg$ is identically $0$, and $\Delta f, \Delta g \geq 0$ because $\Delta v = 0$ on
$\{ v>0\}$. This is enough to apply the monotonicity theorem of \cite{ACF}  that says that
$F(R) = \phi_f(R) \phi_g(R)$ is a nondecreasing function of $R$, where
\begin{equation}\label{a7}
\phi_f(R) = R^{-2} \int_{B(0,R)} {|\nabla f(x)|^2 \over |x|^{n-2}} dx
\ \text{ and } \ 
\phi_g(R) = R^{-2} \int_{B(0,R)} {|\nabla g(x)|^2 \over |x|^{n-2}} dx.
\end{equation}
Since $|\nabla f(x)| \leq M$, it is easy to see that $\phi_f(R) \leq C M$, and similarly
$\phi_g(R) \leq C M$; set
\begin{equation}\label{a8}
\ell = \lim_{R \to +\infty} F(R) = \lim_{R \to +\infty} \phi_f(R) \phi_g(R).
\end{equation}
Thus $\ell < +\infty$; let us check that $\ell > 0$, or equivalently that $F(R) > 0$ for some $R > 0$.
Pick $x \in U$ and $y\in V$; then $f(x) > 0$, $g(y) > 0$, and $f(y)=g(x)=0$. Thus $\nabla f \neq 0$
somewhere on $[x,y]$, and similarly for $\nabla g$. If $R$ is so large that $[x,y] \subset B(0,R)$,
then $F(R) > 0$. Thus $0 < \ell < +\infty$.

Next we will consider any blow-down limit of $v$, and at the same time $f$ and $g$.
For any $\lambda > 0$, define new functions $v_\lambda$, $f_\lambda$, and $g_\lambda$ by
\begin{equation}\label{a9}
v_\lambda(x)=\frac{v(\lambda x)}{\lambda}, f_\lambda(x)=\frac{f(\lambda x)}{\lambda},
g_\lambda(x)=\frac{g(\lambda x)}{\lambda},
\end{equation}
and notice that all these functions are $M$-Lipschitz too. By Arzela-Ascoli, we can find 
sequences $\{\lambda_i\}_i$ such that $\lim_{i \to +\infty} \lambda_i = +\infty$,
and the three sequences, $\{ v_{\lambda_i} \}$, $\{ f_{\lambda_i} \}$, and 
$\{ g_{\lambda_i} \}$ converge, uniformly on compact sets, to limits that we denote by
$v_\infty$, $f_\infty$, and $g_\infty$. We shall need to know that
\begin{equation} \label{a10}
v_\infty(x_0) = f_\infty(x_0) \ \ \text{ when } f_\infty(x_0) > 0. 
\end{equation}
And indeed, $f_\infty(x_0)$ is the limit of $f_{\lambda_i}(x_0)$, so
$f_{\lambda_i}(x_0) > 0$ for $i$ large, which means $f_{\lambda_i}(x_0)= v_{\lambda_i}(x_0)$
by definition, and, therefore, $v_\infty(x_0) = \lim_{i \to +\infty} v_{\lambda_i}(x_0) = f_\infty(x_0)$.

Next we want to check that for $R > 0$,
\begin{equation}\label{a11}
\lim_{i \to +\infty} \phi_f(\lambda_i R) 
= R^{-2} \int_{B(0,R)} {|\nabla f_\infty(x)|^2 \over |x|^{n-2}} dx
= \phi_{f_\infty}(R)
\end{equation}
(with the notation of \eqref{a7} for $f_\infty$).

To prove \eqref{a11} it suffices to show that $|\nabla f_{\lambda_i}|^2 \stackrel{*}{\rightharpoonup} |\nabla f_\infty|^2$ in $L^\infty$. This requires an elementary argument using integration by parts and the uniform boundedness of the $f_{\lambda_i}$ in $W^{1,\infty}$. However, this convergence actually happens strongly in $W^{1,2}_{\mathrm{loc}}$. See Remark \ref{r:thm2.1obs} for more details. 

The proof of \eqref{a11} also shows that
$\lim_{i \to +\infty} \phi_g(\lambda_i R) = \phi_{g_\infty}(R)$. 
We take the product and get that for $R>0$,
\begin{equation} \label{a20}
\phi_{f_\infty}(R) \phi_{g_\infty}(R) 
= \lim_{i \to +\infty} \phi_f(\lambda_i R)\phi_g(\lambda_i R)
= \ell,
\end{equation}
by \eqref{a8}. That is, the analogue of $F$ for the functions $f_\infty$ and $g_\infty$ is constant.
Notice that $f_\infty$ and $g_\infty$ satisfy the assumptions of the monotonicity formula
in \cite{ACF}, because they are Lipschitz and $f_\infty g_\infty=0$. A careful study of the equality case, done in \cite{BKP}, then shows that $f$ and $g$ have the very special form below.
Alternatively, this is also done with some detail (but roughly the same ideas) in 
Lemma 19.3 of \cite{DFJM}
(a paper that was started after this one, but was finished faster). The special form is the following.
There is a unit vector $e \in \R^n$, two positive constants $\alpha$ and $\beta$, and a constant 
$c \in \R$, such that
\begin{equation} \label{a21}
f_\infty (x) = \alpha [\langle x,e \rangle -c]_+ 
\ \text{ and } \ 
g_\infty (x) = \beta [\langle x,e \rangle -c]_-.
\end{equation}
The product $\alpha\beta$ is positive, because it is simply related to $\ell$ and $\ell >0$.
Then $(f_\infty+g_\infty)(x) > 0$ for all $x\in \R^n$ such that $\langle x,e \rangle \neq c$.
We know from \eqref{a10} that $f_\infty(x) = v_\infty(x)$ when $f_\infty(x) > 0$. Similarly, 
$g_\infty(x) = v_\infty(x)$ when $g_\infty(x) > 0$. We are left with $v_\infty(x) = (f_\infty+g_\infty)(x)$
almost everywhere, \eqref{a21} determines $v_\infty$, and it is easy to see
that $v_\infty$ is not a global minimizer. This contradicts Theorem 9.1 in \cite{DT}
(because $v_\infty$ is the limit of the minimisers $v_{\lambda_i}$).

So $\big\{ v(x) > 0 \big\}$ is connected; the fact that $\big\{ v(x) < 0 \big\}$ is connected too 
(when we work with $J$) is proved the same way. 
\end{proof}

\begin{remarks}\label{r:thm2.1obs}
We want to show that the convergence $f_{\lambda_i} \rightarrow f_{\infty}$ happens in the strong $W^{1,2}_{\mathrm{loc}}$ sense. In the first version of this manuscript we had a long argument for this fact, we would like to thank an anonymous referee for the considerably simplified version presented below. Note, perhaps surprisingly, that we do not need $\lambda_+$ or $\lambda_-$ to be positive here (and thus do not need such a condition anywhere in the proof of Theorem \ref{t2.1}). 

Recall that by Arzela-Ascoli we have that $f_{\lambda_i}$ converges uniformly (up to a subsequence which we relabel) to $f_\infty$ on compact sets and by integration by parts that $\nabla f_{\lambda_i}$ converges weak-star in $L^\infty_{\mathrm{loc}}$ to $\nabla f_\infty$. Again by weak star compactness we have that (a further relabeled subsequence of) $|\nabla f_{\lambda_i}|^2$ converges weak star in $L^\infty$ to something. Integrating by parts for a $\varphi$ supported compactly in $\{v_\infty > 0\}$ we get that  $$\int \varphi |\nabla f_{\lambda_i}|^2 = -\int f_{\lambda_i} \nabla \varphi \cdot \nabla f_{\lambda_i} = \frac{1}{2} \int f_{\lambda_i}^2 \Delta \varphi \stackrel{i\rightarrow \infty}{\rightarrow} \frac{1}{2}\int f_\infty^2 \Delta \varphi = \int |\nabla f_{\infty}|^2 \varphi,$$ where we have used that $\Delta f_{\lambda_i} = 0$ on the support of $\varphi$ for $i$ large enough. 

This shows that the weak star limit of $|\nabla f_{\lambda_i}|^2$ is almost everywhere equal to $|\nabla f_{\infty}|^2$ (as $\partial \{v_\infty > 0\}$ is a set of measure 0) which is enough for the convergence of the ACF monotonicity above. Then since $|x|^{2-n}$ is in $L^1_{\mathrm{loc}}$ we can write $$\begin{aligned} \|\nabla f_{\lambda_i} - \nabla f_\infty\|_{L^2(B(0,R))}^2 \leq& R^{n-2}\int_{B(0,R)}\frac{|\nabla f_{\lambda_i}- \nabla f_\infty|^2}{|x|^{n-2}}dx\\
=& R^{n}\left(\phi_{f}(\lambda_i R) + \phi_{f_\infty}(R) - \frac{2}{R^2} \int_{B(0,R)} \frac{\nabla f_{\lambda_i} \cdot \nabla f_{\lambda_\infty}}{|x|^{n-2}}dx\right).
\end{aligned}$$
Letting $i\rightarrow \infty$ and using the $L^\infty$ weak star convergence of $\nabla f_{\lambda_i}$, the local integrability of $\nabla f_{\lambda_\infty}|x|^{2-n}$ and the fact that $\phi_{f}(\lambda_i R) \rightarrow \phi_{f_\infty}(R)$ we get that the $\nabla f_{\lambda_i} \rightarrow \nabla f_{\infty}$ in $L^2_{\mathrm{loc}}$.
\end{remarks}

We now use Theorem \ref{t2.1} to find paths inside $\{v>0\}$ that connect two given 
points and don't get too close to the free boundary. In order to simplify notation, we will use $\ell(\gamma)$ to signify the length of a curve, $\gamma: [0,1]\rightarrow \R^n$. Eventually the existence of these paths will allow us to establish NTA conditions, but we start with a simpler result.

\begin{theorem}\label{t2.2}
Let $\Omega \subset \R^n$ be bounded 
and $q_\pm\in L^\infty(\Omega)\cap C(\overline\Omega)$ be given, and assume that $q_+\ge c_0>0$. 
Then, given $M>0$, and $\theta\in(0,1)$, if $u$ is an almost-minimizer for $J$ or $J^+$ in $\Omega$ with $\|\nabla u\|_{L^\infty(\Omega)}\le M$, there exists $C_0=C_0(M, \theta) > 0$ and $r_0 = r_0(M, \theta) > 0$ such that for $r\in (0, r_0)$ and $x,\ y \ \in \{u>0\}\cap \Omega$ with 
\begin{eqnarray}\label{eqn2.22}
\min\{\dist(x,\Omega^c),\dist(y,\Omega^c)\}&\ge& C_0 r\nonumber\\
|x-y|&\le& r\\
\min\{\delta(x),\delta(y)\}&\ge& \theta r,\nonumber
\end{eqnarray}
(where $\delta(\cdot)=\dist(\cdot, \Gamma^+(u))$ and $\Gamma^+(u)$ is as in
\eqref{eqn1.2}), there exists a curve, $\gamma:[0,1]\to \{u>0\}\cap \Omega$, satisfying
\begin{eqnarray}\label{eqn2.23}
\gamma(0)=x& \text{ and } & \gamma(1)=y\nonumber\\ 
\dist(\gamma([0,1]), \Gamma^+(u))&\ge& C_0^{-1} r \\ 
\ell(\gamma) &\le& C_0r\nonumber.
\end{eqnarray}
\end{theorem} 

\begin{remarks}\label{rmk2.1}
\begin{enumerate}
\item The reader is possibly surprised that we require a full Lipschitz control of $u$ on $\Omega$
(and maybe to a lesser extent, that $q_\pm$ is continuous on the whole $\d\Omega$), 
but this is no more than a way to assert that we do not look for a control near $\d\Omega$.
Indeed, if $\|\nabla u\|_{L^2(\Omega)}< \infty$, for instance, then $u$ is Lipschitz on every compact subset
of $\Omega$ (see Theorems~5.1 and 8.1 in \cite{DT}); so we can apply Theorem \ref{t2.2} to any relatively
compact subdomain of $\Omega$.

\item The statement is more difficult to prove (and hence we expect a larger $C_0$) when 
$\theta \in (0,1)$ is small. Also, if $x,\ y$ are as in Theorem \ref{t2.2} and 
$\min\{\delta(x),\delta(y)\}\ge 2r$, then since $|x-y|\le r$ the segment 
joining $x$ to $y$ satisfies \eqref{eqn2.23}.
Thus in the proof of Theorem \ref{t2.2} we will assume $\min\{\delta(x),\delta(y)\}\le 2r$.

\item If $\gamma$ is as in Theorem \ref{t2.2}, since $\ell(\gamma)\le C_0 r$ 
then $\diam\gamma\le C_0 r$ and we get that for $z\in\gamma([0,1])$,
\begin{equation}\label{eqn2.24}
\delta(z)\ge\frac{r}{C_0}\ge \frac{\diam\gamma}{C_0^2}\ge  \frac{|z-x|}{C_0^2}
\end{equation}
\item In our statement $C_0$ and $r_0$ depend on our choice of $\Omega$, $q_+$, and $q_-$,
but what really matters is to have the lower bound, $c_0$, on $q_+$ and 
a (uniform) modulus of continuity for $q_+$, and $q_-$ on $\Omega$; 
the proof would be almost be the same as below, except that we would also let $\Omega$, 
$q_+$, and $q_-$, vary along our contradiction sequence. 
We will not need this remark, and in fact we only need Theorem \ref{t2.3} below.
\end{enumerate}
\end{remarks}

\begin{proof}
We proceed by contradiction, 
using a limiting argument as well as the information we have about global minimizers. 
Let $\Omega$, $q_+$, $q_-$, $M$, and $\theta\in (0,1)$ be given, and 
suppose that for all $k\in\NN$ there exist an almost-minimizer, $u_k$, 
 for $J$ (resp. $J^+$) in $\Omega$ such that $\|\nabla u_k\|_{L^\infty(\Omega)}\le M$, 
 a sequence $\{r_k\}$ with $\lim_{k\to\infty} r_k=0$, 
 and points $x_k,\ y_k\in\{u_k>0\}\cap \Omega$ such that 
\begin{eqnarray}\label{eqn2.27}
\min\{\dist(x_k,\Omega^c),\dist(y_k,\Omega^c)\}&\ge& 2^k r_k\nonumber\\
|x_k-y_k|&\le& r_k\\
\theta r_k&\le& \min\{\delta(x_k),\delta(y_k)\} \nonumber 
\end{eqnarray} 
and for any curve $\gamma_k:[0,1]\to \{u_k>0\}$ with $\gamma_k(0)=x_k$ and $\gamma_k(1)=y_k$, either 
\begin{equation}\label{eqn2.28}
\dist(\gamma_k(t),\Gamma^+(u_k))<\frac{r_k}{2^k}\quad\hbox{   for some   }t\in[0,1]
\end{equation}
 or
 \begin{equation}\label{eqn2.29}
 \ell(\gamma_k)>2^kr_k.
 \end{equation}
We may assume that $\delta(x_k)\le \delta(y_k)$. 
Pick $\overline{x}_k\in\Gamma^+(u_k)$ such that $|x_k-\overline{x}_k|=\delta(x_k)$. 
As mentioned in the Remark \ref{rmk2.1}, $|x_k-\overline{x}_k| \leq 2 r_k$, because otherwise the
segment $[x_k,y_k]$ would yield a curve $\gamma_k$ for which \eqref{eqn2.28} and \eqref{eqn2.29}
fail. Thus $B(\overline{x}_k, 2^{k-1}r_k)\subset\Omega$ for $k \geq 2$, by \eqref{eqn2.27}.
Let us restrict to $k\geq 2$ and set 
\begin{equation}\label{eqn2.30}
v_k(x)=\frac{u(r_k x+\overline{x}_k)}{r_k} \quad\hbox{ for } x\in B(0,2^{k-1}).
\end{equation}
By assumption $||\nabla u_k||_\infty \leq M$, so $v_k$ is $M$-Lipschitz on $B(0,2^{k-1})$. 
Also, $u_k$ vanishes at $\overline x_k \in \Gamma^+(u_k)$, hence $v_k(0)=0$.
Modulo passing to a subsequence (which we immediately relabel) we may assume that 
$\{ v_k \}$ converges, uniformly on compact subsets of $\R^n$, to an $M$-Lipschitz function $v_\infty$.

Since $\Omega$ is bounded, we may also assume that 
$\lim_{k\to\infty}\overline{x}_k =\overline x_\infty\in\overline\Omega$. 
Set $q_\pm^k(x):= q_\pm(r_kx +\overline x_k)$; we have the same $L^\infty$ bounds on 
the $q_\pm^k$ as on $q_\pm$, and since $q_\pm$ is continuous on $\overline\Omega$,
$\{ q_\pm^k \}$ converges to the constant $q_\pm(\overline x_\infty)$, uniformly on compact
subsets of $\R^n$. This is where, if we wanted to prove that $C_0$ does not depend on $\Omega$
or the $q_{\pm}$, we would use a uniform modulus of continuity and get that $\{ q_\pm^k \}$ 
converges to a constant.

Each $v_k$ is an almost-minimizer for $J_k$ (resp. for $J^+_k$) in $B(0,2^{k-1})$,
corresponding to the functions $q_\pm^k$ (and the constant $r_k^{\alpha}\kappa$). 
Theorem 9.1 and (the proof of) Theorem 9.2 in \cite{DT} 
ensure that $v_\infty$ is a global minimizer of $J_\infty$ (resp. $J^+_\infty$) in $\R^n$,
associated to the constants $\lambda_\pm = q_\pm(\overline x_\infty)$, as in 
Definition~\ref{d2.1} or \ref{d2.2}. It is also $M$-Lipschitz, so we may apply Theorem~\ref{t2.1}
to it. We get that $\{ v_\infty > 0 \}$ is connected.

We now compare $\{ v_\infty > 0 \}$ to the sets $\{ u_k > 0 \}$.
This is the place in the argument where we will use our assumption that $q_+ \geq c_0 >0$,
through the non-degeneracy of $u_k$ and $v_\infty$.
We may assume, at the price of an additional almost-minimizers extraction, that
the sets
\begin{equation}\label{eqn2.32}
\Lambda_k=\overline{\{v_k>0\}}=\overline{\frac{1}{r_k}\big(\{u_k>0\} -\overline{x_k}\big)}
\end{equation}
converge, in the Hausdorff distance on every compact subset of $\R^n$, to some
(closed) set $\Lambda_\infty$. Let us check that 
\begin{equation}\label{a2.33}
{\rm{int}}\, \Lambda_\infty=\{v_\infty>0\}.
\end{equation}
If $p\in \rm{int}\, \Lambda_\infty$ there is $s\in(0,1)$ such that $B(p,s)\subset \Lambda_\infty$. 
Thus for $k$ large enough 
$B(p,s/2)\subset \Lambda_k=\frac{1}{r_k}\left(\{u_k>0\} -\overline{x_k}\right)$.
That is, $B_k := B(r_kp+\overline{x_k}, sr_k/2)\subset \{u_k>0\}$.
Recall that $B(\overline{x}_k, 2^{k-1}r_k)\subset\Omega$; thus for $k$ large, $B_k$ 
lies well inside $\Omega$, where we also know that $u_k$ is $M$-Lipschitz; 
then Theorem 10.2 in \cite{DT} ensures that there is $\eta > 0$ such that for $k$ large,
$u_k(r_kp+\overline{x_k})\ge \eta sr_k/2$. 
Thus $v_k(p)\ge \eta s/2$ for all $k$ large, which implies that $v_\infty(p)\ge \eta s/2$
and $p \in \{v_\infty>0\}$.

Conversely, let $p\in  \{v_\infty>0\}$ be given. Then for $k$ large enough $v_k(p)\ge v_\infty(p)/2$.
Set $B = B(p, v_\infty(p)/(4M))$; since $v_k$ is $M$-Lipschitz, we also get that 
$v_k(q)\ge v_\infty(p)/4$ for $q\in B$. 
That is, $u_k(r_kq+\overline{x_k})\ge v_\infty(p)r_k/4$. 
Hence $r_kq+\overline{x_k}\subset \{u_k>0\}$ and 
$q\in \Lambda_k=\frac{1}{r_k}\big(\{u_k>0\} -\overline{x_k}\big)$.  
Thus $B \subset \Lambda_k$ for $k$ large, and it follows that $p \in {\rm{int}}\, \Lambda_\infty$;
 \eqref{a2.33} follows.
 
 \ms
 
Next consider the points $x'_k = r_k^{-1}(x_k-\overline{x_k})$ and $y'_k = r_k^{-1}(y_k-\overline x_k)$.
Notice that $|x'_k-y'_k| \leq 1$ by \eqref{eqn2.27}, and
$|x'_k| = r_k^{-1} |x_k-\overline x_k| \leq 2$ (see below \eqref{eqn2.29}). 
Thus we can assume, modulo extracting a new subsequence, that $\{ x'_k \}$ converges to some point
$x'\in \overline B(0,2)$ and $\{ y'_k \}$ converges to $y'\in \overline B(0,3)$.
Moreover, by \eqref{eqn2.27}
\begin{equation}\label{a2.34}
\theta r_k \leq \delta(x_k) = \dist(x_k, \Gamma^+(u_k))
= \dist(x_k, \{ u_k \leq 0 \})
= r_k \dist(x'_k, \R^n \sm \Lambda_k)
\end{equation}
because $\dist(x_k,\R^n \sm \Omega) \geq 2^k r_k$ is much larger than $\delta(x_k)$,
and by \eqref{eqn2.32}. Thus for $z\in B(x',\theta/2)$, we get that for $k$ large
$$
\dist(z, \R^n \sm \Lambda_k) \geq \dist(x'_k, \R^n \sm \Lambda_k) - |z-x'|
\geq \dist(x'_k, \R^n \sm \Lambda_k) - {2 \theta \over 3} \geq {\theta \over 3},$$
hence by \eqref{a2.33} $B(x',\theta/2) \subset {\rm{int}}\, \Lambda_\infty=\{v_\infty>0\}$.
By the same proof, $B(y',\theta/2) \subset \{v_\infty>0\}$.

By Theorem \ref{t2.1}, $\{v_\infty>0\}$ is connected, hence there is a path
$\wt\gamma :[{1 \over 3},{2\over 3}]\to \{v_\infty>0\}$, with $\wt\gamma(0)=x'$, and $\wt\gamma(1)=y'$. 
We may even assume (since $\{v_\infty>0\}$ is open) that $\wt\gamma$ is smooth, and in particular
it is $L$-Lipschitz for some $L > 0$.
Also, set $\tau = \dist(\wt\gamma([{1 \over 3},{2\over 3}]), \{v_\infty \leq 0\})$; then $\tau > 0$ because
$\wt\gamma([0,1])$ is compact and $\{v_\infty \leq 0\}$ is closed.

For $k$ large, we can complete $\wt\gamma$ by adding a small segment from $x'_k$ to $x'$ at one end,
and another one from $y'$  to $y'_k$ at the other end; we get a new path 
$\wt \gamma_k : [0,1] \to \{v_\infty>0\}$, 
whose length is $\ell(\wt\gamma_k) \leq L+1$ (for $k$ large),
and such that $\dist(\wt\gamma_k([0,1]), \{v_\infty \leq 0\}) \geq \tau/2$.
Finally set $\gamma_k(t) = \overline x_k + r_k \wt\gamma_k(t)$ for $t \in [0,1]$; we want to show that,
for $k$ large, the existence of $\gamma_k$ violates our initial definitions.

First of all, $\gamma_k(0) = \overline x_k + r_k x'_k = x_k$, and 
$\gamma_k(1) = \overline x_k + r_k y'_k = y_k$. Next let us check that for $k$ large,
\begin{equation}\label{a2.35}
\dist(\gamma_k(t), \Gamma^+(u_k)) \geq \tau/4 \ \text{ for } t\in [0,1],
\end{equation}
and hence \eqref{eqn2.28} fails. Notice that
\begin{eqnarray}\label{a2.36}
\dist(\gamma_k(t), \{ u_k \geq 0 \}) 
&=& r_k \dist(\wt\gamma_k(t), \{ v_k \geq 0 \})
= r_k \dist(\wt\gamma_k(t), \R^n \sm \Lambda_k)
\nonumber \\
&\geq& r_k [\dist(\wt\gamma_k(t), \R^n \sm \Lambda_\infty)-\tau/4]
\end{eqnarray}
by \eqref{eqn2.30} and \eqref{eqn2.32}, and because $\Lambda_\infty$ is the limit of the $\Lambda_k$.
Now $\R^n \sm \Lambda_\infty \subset \{v_\infty \leq 0\}$ by \eqref{a2.33}, so
$\dist(\wt\gamma_k(t), \R^n \sm \Lambda_\infty) \geq 
\dist(\wt\gamma_k(t), \{v_\infty \leq 0\}) \geq \dist(\wt\gamma_k([0,1]), \{v_\infty \leq 0\})
\geq \tau/2$ and so $\dist(\gamma_k(t), \{ u_k \geq 0 \}) \geq \tau/4$. So it is enough to check 
that $\dist(\gamma_k(t), \Gamma^+(u_k)) = \dist(\gamma_k(t), \{ u_k \geq 0 \})$,
or equivalently that  $\dist(\gamma_k(t), \R^n \sm \Omega) > \dist(\gamma_k(t), \{ u_k \geq 0 \})$
(recall the definition \eqref{eqn1.2} and that $u_k(\gamma_k(t)) > 0$).
But $|\wt\gamma_k(t)-x'_k| \leq \ell(\wt\gamma_k)\leq L+1$ because $x'_k= \wt\gamma_k(0)$,
hence $|\gamma_k(t)- x_k| \leq (L+1) r_k$ (because $x_k = \overline x_k + r_k x'_k$), while on the 
other hand $\dist(x_k, \R^n \sm \Omega) \geq 2^k r_k$ by \eqref{eqn2.27}; this proves \eqref{a2.35}
and the failure of \eqref{eqn2.28}.

But \eqref{eqn2.29} also fails for $k$ large, because 
$\ell(\gamma_k) = r_k \ell(\wt\gamma_k) \leq (L+1) r_k$; this contradiction completes our proof 
of Theorem \ref{t2.2}. 
\end{proof}  

We now use Theorem \ref{t2.2} to  prove that, under suitable  assumptions, the open set,
$\{ u > 0 \}$, is a {\bf locally-NTA open set} in $\Omega$. We need some definitions, which are just local versions
of the standard definitions for the Non-Tangentially Accessible (NTA) domains of \cite{JK}. Here $U$ will be a bounded open set,
and since we are thinking of $U = \Omega \cap \{ u > 0 \}$ for some almost-minimizer $u$,
let us not require $U$ to be connected.

Let us first define {\bf corkscrew points} for $U$. Let $z\in \d U$ and $r > 0$. We say that $x$ is a corkscrew point for
$B(z,r)$ (relative to $U$), with constant $C_1 \geq 1$, when $x\in U \cap B(z,r/2)$ and
$\dist(x,\d U) \geq C_1^{-1} r$. We say that $y$ is a corkscrew point for
$B(z,r)$, relative to $\R^n \sm U$ and with constant $C_1 \geq 1$, when $y\in B(z,r/2) \sm U$ and
$\dist(y, \partial U) \geq C_1^{-1} r$.

Finally, given $x,y \in U$, a {\bf Harnack chain} from $x$ to  $y$, of length $N \geq 1$ and constant $C_2 > 1$,
is a collection, $B_1, \ldots, B_N$, of balls, such that $x \in B_1$, $y\in B_N$, 
$B_{j+1} \cap B_j \neq \emptyset$ for $1 \leq j \leq N-1$, and 
\begin{equation}\label{a237}
C_2^{-1} \diam(B_j) \leq \dist(B_j,\d U) \leq C_2 \diam B_j
\ \text{ for } 1 \leq j \leq N.
\end{equation}

\begin{defn}\label{d2.3}
Let $\Omega \subset \R^n$ and $U \subset \Omega$ be open sets.
We say that $U$ is locally NTA in $\Omega$ when for each compact set $K \subset \Omega$,
we can find $r_1 > 0$, and $C_1, C_2$, and $C_3 \geq 1$, such that 
\begin{enumerate}
\item For $x\in K \cap \d U$ and $0 < r \leq r_1$, there is a corkscrew point for $B(x,r)$, 
relative to $U$ and with constant $C_1$;
\item For $x\in K \cap \d U$ and $0 < r \leq r_1$, there is a corkscrew point for
$B(x,r)$, relative to $\R^n \sm U$ and with constant $C_1$;
\item For $x,y\in K \cap U$, with $|x-y| \leq r_1$, and $\ell \in \NN$ such that
$\min( \dist(x,\d U), \dist(y, \d U) \geq 2^{-\ell} |x-y|$, there is a Harnack chain from $x$ to  $y$, 
of length $N \leq C_3 \ell + 1$ and with constant $C_2$.
\end{enumerate}
\end{defn}

Notice that nothing prevents $U$ from having more than one connected component, but if this happens, 
the components must be distance greater than $r_1$ from each other inside any compact set, $K$. We are ready to state the local NTA property
of $U = \{ x > 0 \}$ for almost-minimizers for $J$ and $J_+$.

\begin{theorem}\label{t2.3}
Let $q_\pm\in L^\infty(\Omega)\cap C(\Omega)$ with $q_+\ge c_0>0$, and  let 
$u$ be an almost-minimizer for $J$ or $J^+$ in $\Omega$.
If $u$ is an almost-minimizer for $J$, assume in addition that $0 \leq q_- \leq q_+$ on $\Omega$,
or that $q_- \geq c_0 > 0$ on $\Omega$.
Then $U = \{ x\in \Omega \, ; \, u(x) > 0 \}$ is locally NTA in $\Omega$.
\end{theorem}

The main ingredient will be Theorem \ref{t2.2}, together with non-degeneracy estimates for $u$ and some geometry.
We could ask for more precise estimates, in particular concerning the way that $r_1$ and the NTA constants for $K$
depend on $c_0$, $\dist(K, \d\Omega)$, a bound for $\int_{\Omega} |\nabla u|^2$, and a modulus of continuity
for $q_\pm$ near $K$. Nevertheless since Theorem \ref{t2.2} was obtained via a compactness argument these bounds will not be explicit.

The trickiest part of the proof is to make sure that we do not get too close to $\partial \Omega$ in our constructions. Without worrying about this (important) detail, the argument works roughly as follows: interior/exterior corkscrew points are given by the non-degeneracy of almost-minimizers. To construct a Harnack chain between points $x,y \in U$, we first use the existence of corkscrew points to create a sequence of intermediate points between $x$ and $y$. Then we use Theorem \ref{t2.2} to connect these intermediate points by curves which are not too long or too close to $\Gamma^+(u)$. A collection of balls centered around points on these curves will then satisfy the Harnack chain condition. With this outline in mind, we now present the details. 

\begin{proof}
Let $\Omega$, $q_\pm$, $u$, be given as in the statement, and (for the verification of Definition \ref{d2.3}),
let a compact set $K \subset \Omega$ be given. We need a little room for our construction.
Pick a relatively compact open set $\Omega_1$ in $\Omega$, such that $K \subset \Omega_1 \subset\subset \Omega$.

By Theorems~5.1 and 8.1 in \cite{DT}, $u$ is locally Lipschitz, so we can find $M \geq 0$ such that 
$|\nabla u| \leq M$ on $\Omega_1$. Since $q_+$ and $q_-$ are continuous on $\d\Omega_1$, we can
apply Theorem \ref{t2.2} to the restriction of $u$ to $\Omega_1$, with a constant $\theta$ that will be chosen
soon; this gives constants $C_0(\theta)$ and $r_0(\theta)$ so that the conclusion of the theorem holds.

We start our verification with corkscrew points. 
Set $r_2 = 10^{-1} \dist(K, \d\Omega_1)$ and 
$K_1 = \big\{ z\in \Omega_1 \, ; \, \dist(z,K) \leq r_2\big\}$. 
We even want to find corkscrew points for balls centered on $K_1 \cap \d U$.

For $U$ itself, we get them from Theorem 10.2 in \cite{DT} (i.e. the non-degeneracy of almost-minimizers), and we do not need our extra assumption on $q_-$.
For $\R^n \sm U$, we get the corkscrew points from Proposition 10.3 in \cite{DT}, and our extra assumption
that $q_- \leq q_+$ or $q_- \geq c_0$ is used there, to get one of the sufficient conditions of (10.52) or (10.53) of 
Lemma 10.5 there. This is actually the only place in the proof where we need these extra assumptions, 
so without them we still have local interior NTA properties. More precisely, we get a radius $r_3 > 0$
and a constant $C_1$, that depend on $u$ and $K$ (through $c_0$, $M$ and $\dist(K,\d\Omega_1)$),
such that for $x\in K_1 \cap \d U$ and $0 < r \leq r_3$, there is a corkscrew point $A_+(x,r)$ for $U$, and a 
corkscrew points $A_-(x,r)$ for $\R^n \sm U$, both with the constant $C_1$. Of course we can
take $r_3 < {1 \over 3} \dist(K,\d\Omega_1)$, so $A_\pm(x,r)$ still lies well inside $\Omega_1$.

We are left with the existence of Harnack chains in $U$. Let $x, y \in U$ be given, and set $d = |x-y|$.
Thus we assume that $d \leq r_1$, and we will choose $r_1 < \dist(K,\d\Omega_1)/10$, so 
$d < \dist(K,\d\Omega_1)/10$.
Set $\delta(z) = \dist(x,\d U)$ for $z\in U$. If $\delta(x) \geq 2d$, the single ball $B(x,3d/2)$ makes a perfect
Harnack chain from $x$ to $y$, so we may assume that $\delta(x) \leq 2d\leq\dist(K,\d\Omega_1)/5$, 
and similarly $\delta(y) \leq 2d \leq \dist(K,\d\Omega_1)/5$.

Let us first find a nice chain of points from $x$ to $y$. 
Let $\overline x \in \d\Omega$ be such that $|x-\overline x| = \delta(x)$. We will choose
$r_1 < r_2/10$, so $|x-\overline x| = \delta(x) \leq 2d \leq 2r_1 \leq r_2/5$ and $\overline x \in K_1$
since $x\in K$. Set $x_k = A_+(\overline x, 2^{-k} d)$ for $k \geq 0$; those are well defined (if $r_1 \leq r_3/10$).
We stop the construction as soon as $2^{-k} d < \delta(x)/2$, say, because after this we get too close to $\d U$ for our purpose.

Similarly pick $\overline y \in \d\Omega$ such that $|y-\overline y| = \delta(y)$, and define
$y_k = A_+(\overline y, 2^{-k} d)$ for $k \geq 0$ such that $2^{-k} d \geq \delta(y)/2$. Notice that 
in both case, we keep at least one point ($x_0$ or $y_0$). 

Our string of points is the collection of points $x_k$ and $y_k$. We now say how to define a curve that 
connects all these points, and later use that curve to find a Harnack chain. First consider two consecutive points
$x_k$ and $x_{k+1}$ in our chain that goes to $x$. We want to use Theorem \ref{t2.2} to find a curve
$\gamma_k$ in $U$, that goes from $x_k$ to $x_{k+1}$. 
Set $r = 2^{-k+1} d$ and observe that $|x_k-x_{k+1}| \leq 2^{-k+1} d$
because $x_k \in B(\overline x, 2^{-k} d)$ and similarly for $x_{k+1}$; thus the 
middle constraint in \eqref{eqn2.22} is satisfied. Also, $r \leq 2d \leq 2r_1 \leq r_0(\theta)$
if $r_1$ is small enough. We add that $\dist(x_k, \d\Omega) \geq r_2$ because $x_k \in K_1$,
and similarly for $x_{k+1}$, so the fact that $r \leq 2r_1$ takes care of the first condition in 
\eqref{eqn2.22} if $2C_0(\theta) r_1 \leq r_2$. So we just need to make sure to choose $r_1$ after
$\theta$. 

For our final constraint of \eqref{eqn2.22}, notice that by definition of a corkscrew point, 
$\delta(x_k) \geq C_1^{-1} 2^{-k} d = C_1^{-1} r/2$ and 
$\delta(x_{k+1}) \geq C_1^{-1} 2^{-k-1} d = C_1^{-1} r/4$.
So taking $\theta \leq C_1^{-1} r/4$ is enough to get \eqref{eqn2.22} here.
We apply Theorem \ref{t2.2} and find a path $\gamma_k$ in $U$, from $x_k$ to $x_{k+1}$,
with length at most $C_0(\theta) r = C_0(\theta) 2^{-k+1} d$ and that stays at distance at least
$C_0(\theta)^{-1} r$ from $\Gamma^+(u)$ (or equivalently from $\d U$, because $\d \Omega$
is much further from $\gamma_k$ than $\Gamma^+(u)$ is).

We also find a path $\wt\gamma_k$ from $y_k$ to $y_{k+1}$, when 
$2^{-k-1} d \geq \delta(y)/2$, with similar properties. And three additional paths,
a path $\gamma_{00}$ from $x_0$ to $y_0$, a path $\gamma_f$ from $x$ to the last $x_k$, 
and a path $\wt \gamma_f$ from $y$ to the last $y_k$.
The constraints are similar, but the reader will be happy that we don't check the details, and 
if we pick $\theta$ small enough compared to $C_1^{-1}$
(which depends on $u$ and $K_1$, but not on $r_1$), and then $r_1$ small, 
we can construct all these curves. Let us put all these curves together, to get a long curve $\Gamma$ 
from $x$ to $y$.

It is easy to see that each of the curves above can be covered by a Harnack chain of length at most $C$
that connects its endpoints, and with a same constant $C_2 = 100 C_1$, say. 
If $\ell \in \NN$ is, as in Definition \ref{d2.3}, such that 
$\min(\delta(x),\delta(y)) \geq 2^{-\ell} |x-y| = 2^{-\ell} d$, we see that 
we needed at most $2 \ell + 10$ curves in our construction. Thus we get a Harnack chain from $x$ to $y$,
with length less than $C \ell +1$, as needed.

This completes our verification of the Harnack chain condition: Theorem \ref{t2.3} follows.
\end{proof}
\vspace{-.25 cm}
For the convenience of the reader, we mention an obvious corollary of Theorem \ref{t2.3} in the 
two-phase case. Note, by convention, if one coefficient dominates the other we always relabel them so that 
$q_+ \geq q_-$; thus in Corollary \ref{t2.4} below, we merely assume that both $q_{\pm}$ are 
non-degenerate. 

\begin{corollary}\label{t2.4}
Let $q_\pm\in L^\infty(\Omega)\cap C(\Omega)$ be such that $\min(q_-(x),q_+(x)) \geq c_0 > 0$ on $\Omega$, and 
let $u$ be an almost-minimizer for $J$ in $\Omega$.
Then $U_{\pm} = \{ x\in \Omega \, ; \, \pm u(x) > 0 \}$ is locally NTA in $\Omega$.
\end{corollary} 

\section{Harmonic functions and almost-minimizers}\label{harmonic-functions}

In this section we prove that, under the same non-degeneracy assumption as in Theorem~\ref{t2.3},
if $u$ is an almost-minimizer for $J$ or $J^+$ in $\Omega$, then non-negative harmonic functions 
on $\Omega\cap\{u>0\}$ which vanish continuously on $\Gamma^+(u)$ inherit the behavior of $u$ at the free boundary. Thus, in particular, they vanish linearly at the free boundary.
This will be helpful later, as harmonic functions are very useful as competitors.

The assumptions on $\Omega$, the $q_\pm$, and $u$ will be the same for all this section, so
we state them now. These are also the assumptions of Theorem~\ref{t2.3}, which will be quite helpful.

Let $\Omega\subset \R^n$ be an open, connected, and bounded open set, and let 
$q_-$ and $q_+$ be bounded continuous functions on $\Omega$.
We assume that for some $c_0 > 0$,
\begin{equation}\label{a3.1}
q_+(x) \geq c_0 \ \text{ for } x\in \Omega
\end{equation}
and (for the later results)
\begin{equation}\label{a3.2}
0 \leq q_- \leq q_+ \text{ on } \Omega,
\ \text{ or } \  q_- \geq c_0 > 0 \text{ on } \Omega.
\end{equation}
Of course, if these assumptions are not satisfied on the whole $\Omega$, we can always try to localize,
since the restriction to $\Omega_1 \subseteq \Omega$ of an almost-minimizer in $\Omega$ is an almost-minimizer
in $\Omega_1$. Finally we give ourselves a function $u$ on $\Omega$, and assume that
\begin{equation}\label{a3.3}
u \text{ is an almost-minimizer for $J$ or $J^+$ in $\Omega$.}
\end{equation}

Set $U = \big\{ x\in \Omega \, ; \, u(x) > 0 \}$; thus Theorem~\ref{t2.3} says that $U$ is 
locally NTA. Also set $\Gamma^+ = \Gamma^+(u) = \Omega \cap \d U$, as in \eqref{eqn1.2};
for $x_0$ and $0 < r < \dist(x_0,\d\Omega)$, we define a function $h_{x_0,r}$ by the facts that
$h_{x_0,r} \in W^{1,2}_{loc}(\Omega)$, 
\begin{equation}\label{a3.4}
 h_{x_0,r} =u  \ \text{ on }\ \Omega\sm [B(x_0,r)\cap U],
\end{equation}
and $\int_{B(x_0,r)} |\nabla h_{x_0,r}|^2$ is minimal under these constraints.
Here \eqref{a3.4} is our fairly clean way to state the Dirichlet condition $h_{x_0,r} = u$
on $[U \cap \d B(x_0,r)] \cup [\d U \cap B(x_0,r)]$.
The existence is fairly easy, by convexity and because $u$ itself is a candidate, 
and it follows from the definitions that $h_{x_0,r}$ lies in the 
class $K(\Omega)$ of acceptable competitors. Finally, since $h_{x_0,r}$ minimizes
$\int_{B(x_0,r)} |\nabla h_{x_0,r}|^2$ locally in $B(x_0,r)\cap U$,
\begin{equation}\label{a3.5}
\Delta h_{x_0,r}=0  \text{ in } \ B(x_0,r)\cap U.
\end{equation}
We are interested in the properties of $h_{x_0,r}$ near $\d U$, which we shall obtain
by comparing with $u$ and using the local NTA property of $U$.
We keep the notation
\begin{equation}\label{a3.6}
\delta(z) = \dist(z,\Gamma^+) = \dist(z,\Omega \cap \d U) \ \text{ for } z\in U.
\end{equation}

Recall that we want to get information on $\d U$; for this a good control on harmonic
functions like the $h_{x_0,r}$ will be useful, but for the moment we control $u$ better, because
of its almost-minimizing property; thus we want to compare the two. 
We start with an estimate where we show that $u-h_{x_0,r}$ is small in the part of $U \cap B(x_0, r)$
which does not lie to close to $\d U$. For this we will just need to know that $u$ almost-minimizes
the functional, and $h_{x_0,r}$ minimizes a similar energy; in particular, we will not use 
\eqref{a3.2} or the NTA property yet.

\begin{lemma}\label{l3.1}
Let $\Omega$, $q_\pm$, and $u$ be as above.
For each $r_0 >0$ we can find $\rho_0 \in (0,r_0)$
such that if $\Omega$, $q_\pm$, $u$, and $U$ are as above, and if
$x_0\in \Gamma^+(u)$ is such that $B(x_0, 2r_0)\subset \Omega$, 
then for $r\in (0,\rho_0]$ the harmonic competitor, $h_{x_0,r}$, defined above, satisfies
\begin{equation}\label{eqn3.2}
(1-r^{\alpha/8n})u(x)\le h_{x_0,r}(x)\le (1+r^{\alpha/8n})u(x),
\end{equation}
for all $x\in U\cap B(x_0,r)$ with $\delta(x) \geq r^{1+\alpha/8n}$. 
\end{lemma}

\begin{remark}\label{rmk3.1}
\begin{enumerate}
\item
The reader should not be surprised by the various powers of $r$ that show up in this section.
Using powers of $r$ is just our way of grading the size of errors in a simple way; in particular we don't claim that the powers are optimal.
\item
We could easily improve our control on $\rho_0$. The way we stated things, it would seem that $\rho_0$
depends also on $\Omega$, $q_\pm$, and even $u$. In fact $\rho_0$ depends only on 
$n$, $c_0$ (from \eqref{a3.1}), $\|q_\pm\|_{L^\infty}$, $\kappa$, $\alpha$, $r_0$, 
and a bound on $\int_\Omega |\nabla u|^2$. 
\end{enumerate}
\end{remark}
 
\begin{proof}
Let $x_0$ and $r$ be as in the statement, and set $B_r=B(x_0,r)$ and $h_r = h_{x_0,r}$ for convenience.
We first use the minimizing property of $u$ and the definition of $h_r$ to prove that 
\begin{equation}\label{a3.8} 
\int_{B_r}|\nabla u-\nabla h_r|^2 \leq C r^{n+\alpha},
\end{equation}
with a constant $C$ that depends on $n$, $\|q_\pm\|_{L^\infty}$, $\kappa$, 
$\alpha$, $r_0$, and a bound on $\int_\Omega |\nabla u|^2$ (we don't need $c_0$
but this does not matter).

Notice that for $t \in \R$, the function $w_t =h_r + t(u-h_r)$ also lies in $W^{1,2}_{loc}(\Omega)$ and
satisfies the constraint \eqref{a3.4}, so the minimizing property of $h_r$ implies that 
$\int_{B_r\cap U} |\nabla h_r|^2 \leq \int_{B_r\cap U} |\nabla w_t|^2$ for all $t$ and hence 
\begin{equation}\label{a3.9}
\int_{B_r\cap U} \ \langle\nabla (u-h_r) , \nabla h_r\rangle = 0.
\end{equation}
Then
\begin{eqnarray}\label{a3.10} 
\int_{B_r}|\nabla u-\nabla h_r|^2  
&=&\int_{B_r\cap U }|\nabla u|^2 + \int_{B_r\cap U} |\nabla h_r|^2
-2 \int_{B_r\cap U} \langle\nabla u, \nabla h_r\rangle
\nonumber \\
&=&\int_{B_r\cap U} |\nabla u|^2 + \int_{B_r\cap U }|\nabla h_r|^2 
 -2\int_{B_r\cap U} |\nabla h_r|^2
\nonumber\\
&=& \int_{B_r\cap U} |\nabla u|^2 - \int_{B_r\cap U} |\nabla h_r|^2 
\\
&=& \int_{B_r} |\nabla u|^2 - \int_{B_r} |\nabla h_r|^2.
\nonumber
\end{eqnarray}
But $u$ is an almost-minimizer for $J$ or $J^+$, so by \eqref{eqn1.9} or \eqref{eqn1.11}, 
$$
\int_{B_r} |\nabla u|^2 + \int_{B_r} q_+^2\chi_{\{ u > 0 \}} +q_-^2\chi_{\{u<0\}}
\leq (1+\kappa r^\alpha) \Big\{ \int_{B_r} |\nabla h_r|^2 
+ \int_{B_r} q_+^2\chi_{\{h_r>0\}} +q_-^2\chi_{\{h_r<0\}} \Big\}.
$$
A maximum principle argument with $h_r$ in $U \cap B_r$
shows that $h_r \geq 0$ in $U \cap B_r$, hence 
$$
\int_{B_r} q_+^2\chi_{\{h_r>0\}} +q_-^2\chi_{\{h_r<0\}} 
\leq \int_{B_r} q_+^2\chi_{\{u>0\}} +q_-^2\chi_{\{u<0\}}
$$
and we are left with
\begin{eqnarray}\label{a3.11}
\int_{B_r} |\nabla u|^2 - \int_{B_r} |\nabla h_r|^2
&\leq& \kappa r^\alpha \int_{B_r} |\nabla h_r|^2 
+ \kappa r^\alpha \int_{B_r} q_+^2\chi_{\{u>0\}} +q_-^2\chi_{\{u<0\}}
\nonumber\\
&\leq& \kappa r^\alpha \int_{B_r} |\nabla u|^2 + C \kappa r^{n+\alpha},
\end{eqnarray}
where the constant $C$ depends on the $||q_\pm||_\infty$.
We now use a bound on $\int_\Omega |\nabla u|^2$
(and actually a bound on $\int_{B(x_0,2r_0)} |\nabla u|^2$ would have been enough)
to get a Lipschitz bound on the restriction of $u$ to $B(x_0,r_0)$,
from which we deduce that $\int_{B_r} |\nabla u|^2 \leq C r^n$, with a constant $C$ that
depends on the various quantities mentioned in the statement of Lemma \ref{l3.1}, 
(including $r_0$, but not $c_0$). Now \eqref{a3.8} follows from \eqref{a3.10} and \eqref{a3.11}.  

It follows from \eqref{a3.8} and Poincar\'e's inequality that
\begin{equation}\label{eqn3.4}
\fint_{B_r}|u-h_r|^2\le C r^2\fint_{B_r}|\nabla u-\nabla h_r|^2\le C r^{2+\alpha}.
\end{equation}

Next we want to use \eqref{eqn3.4} to control $u-h_r$ and prove \eqref{eqn3.2}.
But let us first check that
\begin{equation}\label{a3.13}
\delta(x) = \dist(x,\d U) \leq r \ \text{ for } x\in U \cap B_r.
\end{equation}
(That is to say, that $\partial \Omega$ is further from $x$ than $\Gamma^+(u)$).
Recall that $\delta(x) = \dist(x,\Omega \cap \d U)$. If $x\in B_r$, then 
$\delta(x) \leq r$ because $x\in B_r = B(x_0,r)$ and $x_0 \in \Gamma^+(u)$. 
But $\dist(x,\d \Omega) \geq r_0$ because $B_r = B(x_0,r) \subset B(x_0, r_0)$ 
and $B(x_0,2r_0) \subset \Omega$; hence $\delta(x) = \dist(x,\d U)$, as needed for \eqref{a3.13}.

We shall use the fact that 
\begin{equation}\label{a3.14}
C^{-1} \delta(x) \leq u(x) \leq C \delta(x) \ \text{ for } x\in U \cap B_r,
\end{equation}
with a constant $C$ that depends on the various quantities mentioned in Remark \ref{rmk3.1}.1. 
The upper bound comes from our Lipschitz bounds on $u$
(Theorems~5.1 and 8.1 in \cite{DT}), and the lower bound, which also uses the fact that
$q_+ \geq c_0 > 0$ on $\Omega$, comes from Theorem 10.2 in \cite{DT}.
Set
\begin{equation}\label{a3.15}
Z = \left\{x\in U \cap B_r \, ; \,  \delta(x) > r^{1+\alpha/8n}   \right\};
\end{equation}
this is the set where we want to show that \eqref{eqn3.2} holds. Also set
\begin{equation}\label{eqn3.5}
A_r=\left\{x\in \Omega \, ; \,  |u(x)-h_r(x) | > r^{1+\alpha/4}   \right\}.
\end{equation}
Notice that $A_r \subset U \cap B_r$ by \eqref{a3.4}. If 
$x\in Z \sm A_r$, then by \eqref{a3.14}
\begin{equation}\label{a3.17}
|u(x)-h_r(x) | \leq  r^{1+\alpha/4} \leq C r^{1+\alpha/4} \delta(x)^{-1} u(x)
\leq C r^{\alpha/4} r^{-\alpha/8n} u(x) < r^{\alpha/8n} u(x)
\end{equation}
(if $r$ is small enough and because $n>1$), and so \eqref{eqn3.2} is satisfied. 
So we just need to show \eqref{eqn3.2} for $x\in Z \cap A_r$.

Let $x\in Z \cap A_r$ be given. By \eqref{a3.13}, $B(x,\delta(x)) \subset U$ and,
since $h_r$ is nonnegative and harmonic in $B_r\cap U$,
\begin{equation}\label{eqn3.7}
\sup_{B(x, {\delta(x) \over 2})}|\nabla h_r| 
\leq C \delta(x)^{-1} h_r(x).
\end{equation}
Besides, Chebyshev's inequality, combined with \eqref{eqn3.4}, yields
\begin{equation}\label{eqn3.6}
\mathcal H^n(A_r)\le Cr^{\alpha/2}\mathcal H^n(B_r),
\end{equation}
so $A_r$ does not contain any ball of radius larger than $C r^{1+\alpha/2n}$, and 
we can find $y\in \R^n \sm A_r$ such that 
\begin{equation}\label{a3.20}
|x-y|\le Cr^{1+\alpha/2n} < {r^{1+\alpha/8n} \over 3} < {\delta(x) \over 2}
\end{equation}
if $r$ is small enough and because $x\in Z$. Thus $y\in B(x,\delta(x)/2)$, 
we may apply \eqref{eqn3.7}, and we get that
\begin{equation}\label{a3.21}
\begin{aligned}
|h_r(x) - h_r(y)| &\leq |x-y| \sup_{B(x, {\delta(x) \over 2})} |\nabla h_r| 
\leq C r^{1+\alpha/2n} {h_r(x) \over \delta(x)}
\\
&\leq  C r^{1+\alpha/2n} {h_r(x) \over r^{1+\alpha/8n}}
\leq C r^{\alpha/4n} h_r(x)
\end{aligned}
\end{equation}
because $x\in Z$. We also know that $u$ is Lipschitz near $B_r$, so
$|u(x)-u(y)| \leq C |x-y| \leq Cr^{1+\alpha/2n}$. Finally, $y\in \R^n \sm A_r$, so
$|h_r(y)-u(y)| \leq r^{1+\alpha/4}$ (even if $y\notin U \cap B_r$).
Altogether,
\begin{eqnarray}\label{eqn3.8}
|h_r(x)-u(x)|&\le & |h_r(x)-h_r(y)| + |h_r(y)-u(y)| +|u(y)-u(x)|
\nonumber \\ 
&\leq& C r^{\alpha/4n} h_r(x) + r^{1+\alpha/4} + Cr^{1+\alpha/2n}
\leq C r^{\alpha/4n} h_r(x) + Cr^{1+\alpha/2n}.
\end{eqnarray} 
Recall from \eqref{a3.14} and \eqref{a3.15} that 
$u(x) \geq C^{-1} \delta(x) \geq C^{-1} r^{1+\alpha/8n}$, so \eqref{eqn3.8} implies that
\begin{equation}\label{a3.23}
h_r(x) \geq u(x) - C r^{\alpha/4n} h_r(x) - Cr^{1+\alpha/2n}
\geq C^{-1} r^{1+\alpha/8n}  - C r^{\alpha/4n} h_r(x)
\end{equation}
hence also $h_r(x) \geq C^{-1} r^{1+\alpha/8n}$. Therefore, the second term on the right hand side in \eqref{eqn3.8} satisfies
$$
Cr^{1+\alpha/2n} \leq C \, {r^{1+\alpha/2n} \over r^{1+\alpha/8n}} \, h_r(x)
\leq C r^{\alpha/4n} h_r(x)
$$
and \eqref{eqn3.8} implies that 
\begin{equation}\label{a3.24}
|h_r(x)-u(x)| \leq C r^{\alpha/4n} h_r(x)
\end{equation}
or equivalently ${h_r(x) \over u(x)} \in (1-Cr^{\alpha/4n}, 1+Cr^{\alpha/4n})$.
Of course \eqref{eqn3.2} follows, and this completes our proof of Lemma \ref{l3.1}.

\end{proof}

\ms
Our next task is to control the ratio $u/h_{x_0,r}$ on a larger set that gets closer to $x_0$,
and for this we shall use non-tangential cones and the local NTA property of $U$.
For $x_0\in \Gamma^+(u)$ and $A > 1$, define a non-tangential cone, $\Gamma_A(x_0)$, by
\begin{equation}\label{eqn3.13A}
\Gamma_A(x_0)=\big\{x\in U:\ |x-x_0|\le A\delta(x)\big\},
\end{equation} 
where we still denote $U = \{ x > 0 \}$ and $\delta(x)=\dist(x,\Omega\cap \d U)$.
We claim that we can find $A > 1$, and a radius $\rho_1 \in(0,\rho_0)$ 
such that if $B(x_0,2r_0) \subset \Omega$,
there is a curve $\gamma_{x_0}$ such that
\begin{equation}\label{a3.26}
\text{$\gamma_{x_0} \subset\Gamma_A(x_0)$ starts from $x_0$
and ends on $\Gamma_A(x_0) \cap \d B(x_0,\rho_1)$.}
\end{equation}
This is a fairly standard fact that follows from the fact that $U$ is locally NTA in $\Omega$, 
but let us say a few words about the proof. 
First observe that we can restrict our attention to the compact set
$K = \big\{ x \in \Omega \, ; \, \dist(x,\d\Omega) \geq r_0/2\big\}$,
because we assume that $B(x_0,2r_0) \subset \Omega$. Then we can apply Theorem \ref{t2.2}
and the proof of Theorem \ref{t2.3} to find corkscrew points for $U$ and curves that connect them.
We proceed roughly as in the final step of Theorem \ref{t2.3}.
For $k\in \ZZ$ such that $2^k \leq C_1 \rho_1$, we select a corkscrew point $z_k$ for $U$ in
$B(x_0,2^k)$. Such a points exist by Theorem \ref{t2.3} if, say, $\rho_1 \leq C_1^{-1}\rho_0$. 
In addition, we can connect each $z_k$ to $z_{k-1}$ with a nice curve
$\gamma_k \subset U$, as in Theorem \ref{t2.2}. We take for $\gamma(x_0)$ the concatenation of all the 
$\gamma_k$, all the way up to the first point when we reach $\d B(x_0,\rho_0)$ for the first time.
The verification that $\gamma \subset \Gamma_A(x_0)$ for $A$ large is easy:
the points of $\gamma_k$ all lie within $C 2^k$ from $z_k$ (hence, also from $x_0$), 
and at the same time at distance larger than $C^{-1} 2^k$ from $\d U$.

\begin{remark}\label{rmk3.2}
Let us again comment on the constants. 
Here we found $\rho_1$ and $A$ that depend on $\Omega$, $q_\pm$, $u$, and of course $r_0$.
But in fact, we claim that we can choose $A$ and $\rho_1$ depending only on 
$n$, $c_0$, $\|q_\pm\|_{L^\infty}$, $\kappa$, $\alpha$, $r_0$, 
a bound on $\int_\Omega |\nabla u|^2$, and also a modulus of continuity for
$q_+$ and $q_-$ on $B(x_0,9r_0/10)$.
Compared to our similar statement in Remark \ref{rmk3.1}.2, we also added the module of continuity
of the $q_\pm$, because it may play a role in the local NTA constant for $\d U$ at the scale $r_0$,
as mentioned in Remark \ref{rmk2.1}.4.
This observation will apply to most of the results below, and we shall refer to the list of quantities above as 
``the usual constants of Remark \ref{rmk3.2}''.
\end{remark}

We shall naturally restrict to constants, $A$, for which the curves $\gamma_{x_0}$ of \eqref{a3.26}
exist, and as usual taking $A$ even larger will only make other constants larger.
We shall estimate $|h_{x_0,r} - u|$ near $x_0$ by comparing $h_{x_0,r}$ to $h_{x_0,s}$,
for judiciously chosen numbers $s\in (0,r)$, and for this we intend to use the local NTA property of $U$. 
We claim that there exist constants $\eta \in (0,1)$ and $C_3 > 1$, that depends only on the usual constants 
of Lemma \ref{l3.1} (through the local NTA constants for $U$), such that if $0 < s < r < \rho_1$, then
\begin{equation}\label{eqn3.16}
\left|\frac{h_{x_0,r}(x)}{h_{x_0,s}(x)}-\frac{h_{x_0,r}(y)}{h_{x_0,s}(y)}\right|
\le C_3\frac{h_{x_0,r}(x)}{h_{x_0,s}(x)}\left(\frac{|x-y|}{s}\right)^{\eta}
\end{equation}
for $x,y\in U \cap B(x_0,s/2)$. The point is that both $h_{x_0,r}$ and $h_{x_0,s}$ 
are nonnegative harmonic functions on $U \cap B(x_0,s)$ (by \eqref{a3.5}) that vanish on 
$\d U \cap B(x_0,s)$; 
then \eqref{eqn3.16} follows from the results in \cite{JK}
(with a simple adaptation to locally NTA domains), which use a boundary Harnack inequality to prove 
the H\"older regularity of $\frac{h_{x_0,r}}{h_{x_0,s}}$ up to the boundary.

We shall now improve on the previous lemma, and approximate $u$ by $h_{x_0,r}$
in the non-tangential cone, $\Gamma_A(x_0)$.

\begin{lemma}\label{l3.2}
Let  $\Omega\subset \R^n$, $q_\pm\in L^\infty(\Omega)\cap C(\Omega)$, and
$u$ (an almost-minimizer for $J$ or $J^+$) satisfy the assumptions \eqref{a3.1}-\eqref{a3.3}
of the beginning of this section. For each choice of $r_0 > 0$ and $A > 1$, there exist constants
$\rho_2 \in (0,r_0)$ and $\beta \in (0,\alpha/16n)$, 
with the following properties.
Let $x_0\in \Gamma^+(u)$ be such that $B( x_0, 2r_0)\subset \Omega$; then for 
$0 < r \leq \rho_2$, the function $h_{x_0,r}$, defined near \eqref{a3.4}, satisfies
 \begin{equation}\label{eqn3.11}
(1-r^\beta)u(x)\le h_{x_0,r}(x)\le (1+r^\beta)u(x)
\end{equation}
for $x\in B( x_0, 10 r^{1+\alpha/17n})\cap\Gamma_A(x_0)$.
\end{lemma}

\ms
In fact, the proof below and Remark \ref{rmk3.2} will show that
$\rho_2$ and $\beta$ depend only on $A$ and the usual constants of Remark \ref{rmk3.2}.

\begin{proof}
Let $A > 1$ be given. We can safely assume that $A$ is large enough for \eqref{a3.26}.
Then, let $x_0\in \Gamma^+(u)$ and $r_0>0$ be such that $B( x_0, 2r_0)\subset \Omega$.

For $0 < s \leq \rho_1$, we can use the path $\gamma_{x_0}$ of \eqref{a3.26} to find
a point $z(s) \in \gamma_{x_0} \cap \d B(x_0, s)$; thus
\begin{equation}\label{a3.29}
|z(s)-x_0| = s \ \text{ and } \delta(z(s)) \geq s/A
\end{equation}
because $z(s) \in \Gamma_A(x_0)$.
We may now apply Lemma \ref{l3.1}; if $0 < r < \rho_0$, we get that \eqref{eqn3.2}
holds for $x\in U\cap B_r$ such that  $\delta(x) \geq r^{1+\alpha/8n}$. In particular,
taking $x =  z(s)$, we see that 
\begin{equation}\label{a3.30}
\frac{h_{x_0,r}(z(s))}{u(z(s))} \in [1-r^{\alpha/8n}, 1+r^{\alpha/8n}]
\end{equation}
for 
\begin{equation}\label{a3.31}
A r^{1+\alpha/8n} \leq s < r.
\end{equation}
Let $\gamma > 1$ be such that $\gamma^2 < 1+\alpha/8n$. For instance, $\gamma$
just a bit larger than $1 + \alpha/18n$ will do.
Set $s(r) = r^\gamma$, and then define $s_k(r)$ by induction, by $s_0(r) = r$ and $s_{k+1}(r) = s(s_k(r))$ 
for $k \geq 0$. 
Notice that $s(r) < r$ (if $r < 1$), and 
$s_2(r) = r^{\gamma^2} > A r^{1+\alpha/8n}$ for $0 < r \leq \rho_2$, 
by definition of $\gamma$ and if $\rho_2$ is chosen is small enough. 
Thus \eqref{a3.30} holds for $s_2(r) \leq s \leq s(r)$.

Fix $r \leq \rho_2$ (with $\rho_2$ small enough), and now set $r_k = s_k(r)$, 
$x_k = z(s_k(r))$ and $h_k = h_{x_0,s_k(r)}$. We just observed that
\begin{equation}\label{a3.32}
\frac{h_{0}(x_\ell)}{u(x_\ell)} \in [1-r^{\alpha/8n}, 1+r^{\alpha/8n}] \ \text{ for } \ell = 1, 2
\end{equation}
but we may  also apply this  to the radius $r_k = s_k(r)$ 
and the corresponding function $h_k$, and we get that
\begin{equation}\label{a3.33}
\frac{h_{k}(x_{k+\ell})}{u(x_{k+\ell})} \in [1-r_k^{\alpha/8n}, 1+r_k^{\alpha/8n}] 
\ \text{ for } \ell = 1, 2.
\end{equation}
We apply this with $k$ and $\ell = 2$, then $k+1$ and $\ell = 1$, then divide and and get that
\begin{equation}\label{a3.34}
\frac{h_{k}(x_{k+2})}{h_{k+1}(x_{k+2})} = 
\frac{h_{k}(x_{k+2})}{u(x_{k+2})}\frac{u(x_{k+2})}{h_{k+1}(x_{k+2})}
\in [1-3r_k^{\alpha/8n}, 1+3r_k^{\alpha/8n}].
\end{equation}
Next we use the fact that for $j > k+2$, 
$|x_j - x_{k+2}| \leq |x_j - x_0| + |x_0 - x_{k+2}| 
\leq s_j(r) + s_{k+2}(r) = r_j + r_{k+2} \leq 2 r_{k+2}$ (because $s(\rho) < \rho$);
then by \eqref{eqn3.16}
\begin{equation}\label{a3.35}
\left|\frac{h_{k}(x_{k+2})}{h_{k+1}(x_{k+2})}-\frac{h_{k}(x_{j})}{h_{k+1}(x_{j})}\right|
\le C_3\frac{h_{k}(x_{k+2})}{h_{k+1}(x_{k+2})}\left(\frac{2 r_{k+2}}{r_{k+1}}\right)^{\eta}
\leq 4C_3 r_{k+1}^{\eta (\gamma-1)}
\end{equation}
because $r_{k+2} = s(r_{k+1}) = r_{k+1}^\gamma$. Thus for $j \geq k+2$,
\begin{equation}\label{a3.36}
\frac{h_{k}(x_{j})}{h_{k+1}(x_{j})} 
\in [1-3r_k^{\alpha/8n}-4C_3 r_{k+1}^{\eta (\gamma-1)}, 1+3r_k^{\alpha/8n}
+4C_3 r_{k+1}^{\eta (\gamma-1)}].
\end{equation}
We take logarithms, notice that $C_3 \geq 1$ and $\gamma -1 \leq \alpha/8n$, restrict to $r$ small, and get that
 \begin{equation}\label{a3.37}
\big|\ln(h_{k}(x_{j})) - \ln(h_{k+1}(x_{j}))\big| \leq 8C_3 r_{k}^{\eta (\gamma-1)}.
\end{equation}
Then we fix $j$, sum over $k \leq j-2$, and get that
\begin{equation}\label{a3.38}
\big|\ln(h_{0}(x_{j})) - \ln(h_{j-1}(x_{j}))\big| \leq 8C_3 \sum_k r_{k}^{\eta (\gamma-1)}
\leq C(\eta) r^{\eta (\gamma-1)}.
\end{equation}
We add a last term that comes from \eqref{a3.33} (with $k = j-1$ and $\ell = 1$), and get that 
\begin{equation}\label{a3.39}
\big|\ln(h_{0}(x_{j})) - \ln(u(x_{j}))\big| \leq C(\eta) r^{\eta (\gamma-1)}.
\end{equation}
This looks a lot like \eqref{eqn3.11}, but along the specific points $\{x_j\}$, whereas we need an inequality at
 generic points in $\Gamma_A(x_0)$. Yet we are ready to prove \eqref{eqn3.11}, with
$\beta = \eta\alpha/18n > 0$. 

Let $x \in \Gamma_A(x_0) \cap B(x_0, r_1)$ be given. 
Let $k$ be such that $r_{k+2} \leq |x-x_0| < r_{k+1}$, and notice that $k \geq 0$.
Also observe that $|x-x_{k+1}| \leq |x-x_0| + |x_0 - x_{k+1}| < 2 r_{k+1}$.

Let us copy the proof of \eqref{a3.32}.
Since $r_{k+2} \leq |x-x_0|$ and $x \in \Gamma_A(x_0)$, the definition yields 
$\delta(x) \geq A^{-1}|x-x_0| \geq A^{-1} r_{k+2} = A^{-1} r_k^{\gamma^2} > r_k^{1+\alpha/8n}$
because $\gamma^2 < 1+\alpha/8n$ and $r_k \leq \rho_2$ is small. Then \eqref{eqn3.2} holds for 
$h_k = h_{x_0,r_k}$, i.e., 
\begin{equation}\label{a3.40}
\frac{h_{k}(x)}{u(x)} \in [1-r_k^{\alpha/8n}, 1+r_k^{\alpha/8n}]. 
\end{equation}
Then we apply \eqref{eqn3.16} to $h_0$ and $h_k = h_{x_0,r_k}$, and the points
$x_{k+1}$ and $x$. We get that
\begin{equation}\label{a3.41}
\left|\frac{h_{0}(x_{k+1})}{h_{k}(x_{k+1})}-\frac{h_{0}(x)}{h_{k}(x)}\right| =
\left|\frac{h_{x_0,r}(x_{k+1})}{h_{x_0,r_k}(x_{k+1})}-\frac{h_{x_0,r}(x)}{h_{x_0,r_k}(x)}\right|
\le C_3\frac{h_{x_0,r}(x_{k+1})}{h_{x_0,r_k}(x_{k+1})}\left(\frac{|x_{k+1}-x|}{r_k}\right)^{\eta}.
\end{equation}
But we said that $|x_{k+1}-x| \leq 2r_{k+1} = 2 r_k^\gamma$, and 
$h_{x_0,r}(x_{k+1}) \leq 2 h_{x_0,r_k}(x_{k+1})$ by \eqref{a3.38}, so
\begin{equation}\label{a3.42}
\left|\frac{h_{0}(x_{k+1})}{h_{k}(x_{k+1})}-\frac{h_{0}(x)}{h_{k}(x)}\right| 
\leq 2 C_3 \left(\frac{|x_{k+1}-x|}{r_k}\right)^{\eta}
\leq 4 C_3 r_k^{\eta(\gamma-1)}.
\end{equation}
So we know that $h_k(x)/u(x)$ is close to $1$ by \eqref{a3.40},
that $h_0(x)/h_k(x)$ is close to $h_0(x_k)/h_k(x_k)$ by \eqref{a3.42}
and that $h_0(x_k)/h_k(x_k)$ is close to $1$ by \eqref{a3.38}; this proves that 
$h_0(x)/u(x)$ is close to $1$, and more precisely that
\begin{equation}\label{a3.43}
\left|\frac{h_{0}(x)}{u(x)}-1\right| 
\leq 2 r_k^{\alpha/8n} + 2C(\eta) r^{\eta(\gamma-1)}+8C_3 r_k^{\eta(\gamma-1)}
\leq C'(\eta) r^{\eta(\gamma-1)}.
\end{equation}
We picked $\gamma$ just a bit larger than $1 + \alpha/18n$, as announced, and this way 
$\gamma-1 > \alpha/18n$, and \eqref{eqn3.11}, with $\beta = \eta\alpha/18n > 0$, follows from \eqref{a3.43}.
Finally, our proof holds for all $x\in B(x_0, r_1)\cap \Gamma_A(x_0) \equiv B(x_0, r^\gamma)\cap\Gamma_A(x_0) \subset B(x_0, 10r^{1+\alpha/17n})\cap \Gamma_A(x_0)$, as long as we make sure to take $\gamma < 1+\alpha/17n$ (because then $10r^{1+\alpha/17n} < r^\gamma$). This completes our proof of Lemma \ref{l3.2}.

\end{proof}

\ms
Finally we show that under the assumptions of the two previous lemmas, $h_{x_0,r}(x)$ approximates
$u(x)$ well near $x_0$, even for $x$ outside of the cone $\Gamma_A(x_0)$ (this will not be too difficult as $x$ will lie
in some other non-tangential cone, depending on $x$). 
We can also replace functions $h_{x_0,r}(x)$ with $h_{z,r}$, for $z \in \Gamma^+(u)$ near $x_0$.

\begin{theorem} \label{t3.1}
Let  $\Omega\subset \R^n$, $q_\pm\in L^\infty(\Omega)\cap C(\Omega)$, and
$u$ (an almost-minimizer for $J$ or $J^+$) satisfy the assumptions \eqref{a3.1}-\eqref{a3.3}
of the beginning of this section. For each choice of $r_0 > 0$, there exist constants
$\rho_3 \in (0,r_0)$, and $\beta \in (0,\alpha/16n)$,
with the following properties;
given $x_0\in \Gamma^+(u)$, such that $B( x_0, 4r_0)\subset \Omega$,
\begin{equation}\label{a3.44}
(1-5r^{\beta})u(x)\le h_{z,r}(x)\le (1+5r^{\beta})u(x)
\end{equation}
for all $0 < r \leq \rho_3$, $z \in \Gamma^+(u)$, and $x\in U$ such that
$|z-x_0|+ |x-x_0| < 5 r^{1+\alpha/17n}$.
Here the function $h_{z,r}$ is defined just as $h_{x_0,r}$ near \eqref{a3.4}, 
but with $x_0$ replaced by $z$.
\end{theorem} 

\ms
As in the previous remarks, the proof will show that $\rho_3$ and $\beta$
depend only on the usual constants of Remark \ref{rmk3.2}.

\begin{proof} 
As we shall see, most of the information comes from Lemma \ref{l3.1}.
Let $x_0$ and $r$ be as in the statement, and set $\rho = r^{1+\alpha/17n}$
to simplify the notation.
We start with any $z \in \Gamma^+(u) \cap B( x_0, 10 \rho)$. 
Thus $B(z,2r_0) \subset \Omega$, and we can apply Lemma~\ref{l3.1}
(with a large constant $A$ that will be chosen soon), both to $x_0$ and to $z$. 
So, if we make sure to take $\rho_3 \leq \rho_2$, \eqref{eqn3.11} says that
\begin{equation}\label{a3.45}
(1-r^\beta) \le \frac{h_{x_0,r}(x)}{u(x)}\le (1+r^\beta)
\end{equation}
for $x\in \Gamma_A(x_0) \cap B(x_0, 10 \rho)$ but also
 \begin{equation}\label{a3.46}
(1-r^\beta) \le \frac{h_{z,r}(x)}{u(x)}\le (1+r^\beta)
\end{equation}
for $x\in \Gamma_A(z) \cap B(x_0, 10 \rho)$.
We compare the two (i.e., multiply and divide by $u(x)$) and get that 
for $x\in \Gamma_A(x_0) \cap \Gamma_A(z) \cap B( x_0, 10 \rho)\cap B(z, 10 \rho)$,
\begin{equation}\label{a3.47}
(1-r^\beta)^2 \leq \frac{h_{z,r}(x)}{h_{x_0,r}(x)} \leq (1+r^\beta)^2.
\end{equation}
Notice that this last set is not empty: if $\rho_3$ is small enough, the local NTA property
of $U$ gives us a corkscrew point, $\xi$, for $U$ in $B( x_0, \rho)$, as in
Part 1 of Definition \ref{d2.3}. That is, $\delta(\xi) \geq  C_1^{-1}\rho$.
This point $\xi$ lies in the intersection above if $A > 10C_1$, and hence satisfies \eqref{a3.46}.

Now $h_{x_0,r}$ and $h_{z,r}$ are two non-negative harmonic function on $U \cap B( x_0, r/2)$
that vanish on $\d U$, hence by the NTA property and as in \eqref{eqn3.16}
\begin{equation}\label{a3.48}
\left|\frac{h_{z,r}(\xi)}{h_{x_0,r}(\xi)}-\frac{h_{z,r}(y)}{h_{x_0,r}(y)}\right|
\le C_3\frac{h_{z,r}(\xi)}{h_{x_0,r}(\xi)}\left(\frac{|x-y|}{r}\right)^{\eta} \leq 20C_3 r^{\eta \alpha/17n}
\end{equation}
for $y\in U \cap B(x_0, 10\rho)$ and $\xi = A_+(x_0, \rho)$ as above. Recall that we took $\beta = \eta \alpha /18n < \eta \alpha/17n$;
then if we take $\rho_3$ small enough, we get that $20C_3r^{\eta \alpha/17n} < r^\beta/2$.
We compare \eqref{a3.48} with \eqref{a3.47} for $\xi$ and get that
\begin{equation}\label{a3.49}
(1-3r^\beta) \leq \frac{h_{z,r}(y)}{h_{x_0,r}(y)} \leq (1+3r^\beta).
\end{equation}
This holds for $y\in U \cap B(x_0, 10\rho)$, but if in addition $y\in \Gamma_A(z) \cap B(z, 10\rho)$, 
we can apply \eqref{a3.46} to $y$ and we get that
\begin{equation}\label{a3.50}
(1-5r^\beta) \leq \frac{h_{x_0,r}(y)}{u(y)} \leq (1+5r^\beta).
\end{equation}
Let us check that in fact \eqref{a3.50} holds for every $y\in U \cap B(x_0,5\rho)$.
Let $z\in \Gamma^+(u)$ minimize the distance to $y$; then $|z-y| \leq |x_0 - y| < 5\rho$ and 
obviously $y \in U \cap B(x_0, 10\rho) \cap U \cap B(z, 10\rho)$; in addition, 
$\delta(y) = |z-y|$ so $y\in \Gamma_A(z)$. Then \eqref{a3.50} holds, as announced.

In fact our proof of \eqref{a3.50} works just the same if we assume that $B(x_0,3r_0) \subset \Omega$
(instead of $B(x_0,4r_0) \subset \Omega$), so if $z\in \Gamma^+ \cap B(x_0, 10\rho)$, we also get that
\begin{equation}\label{a3.51}
(1-5r^\beta) \leq \frac{h_{z,r}(y)}{u(y)} \leq (1+5r^\beta)
\ \text{ for } y \in U \cap B(z,5\rho)
\end{equation}
with exactly the same proof. In particular, this holds when $|z-x_0|+ |y-x_0| < 5 \rho$,
as announced in the statement. This completes the proof of Theorem \ref{t3.1}.

\end{proof}

\ms
Let us record some simple consequences of Theorem \ref{t3.1}. First observe that when
$u$, $r_0$, $x_0$, $0 < r \leq \rho_3$ are as in the statement, then
\begin{equation}\label{a3.52}
(1-11r^\beta) \leq \frac{h_{z,r}(x)}{h_{w,r}(x)} \leq (1+11r^\beta)
\end{equation}
for $z, w\in \Gamma^+(u)$ and $x \in U$ such that $|x-x_0| + \max(|z-x_0|,|y-x_0|) < 5 r^{1+\alpha/17n}$.
Indeed, \eqref{a3.44} also holds with $z$ replaced by $w$, and then we compare.

\ms
We can also compare $h_{x_0,r}$ with $h_{x_0,s}$. Let $u$, $r_0$, and $x_0$ be as in the statement;
then for $0 < s < r \leq \rho_3$, the ratio $\frac{h_{x_0,s}}{h_{x_0,r}}$ of positive harmonic functions 
on $U \cap B(x_0,s/2)$ is continuous up to the boundary, so we can define
\begin{equation}\label{eqn3.34}
\ell_{s,\rho}(x_0)=\lim_{x \to x_0 \, ; \, x \in U} \ \frac{h_{x_0,s}(x)}{h_{x_0,\rho}(x)}.
\end{equation}
Then, for $x\in U \cap B(x_0,5 s^{1+\alpha/17n})$, we have \eqref{a3.44} for $h_{x_0,r}(x)$,
but also $h_{x_0,s}(x)$; we take the ratio, then take the limit when $x$ tends to $x_0$, and get that
\begin{equation}\label{eqn3.35}
1-11\rho^\beta\le \ell_{s,\rho}(x_0)\le 1+11\rho^\beta.
\end{equation}

\ms
Here is a simple consequence of Theorem \ref{t3.1}, that will be enough in some cases.

\begin{corollary}\label{cor3.3}
Let  $\Omega\subset \R^n$ be an open connected domain, 
and $q_\pm\in L^\infty(\Omega)\cap C(\Omega)$ with $q_+\ge c_0>0$ and \eqref{a3.2}. 
Let $u$ be an almost-minimizer for $J$ or $J^+$ in $\Omega$. 
Given $\epsilon>0$ and $r_0>0$, there exist $\rho_4>0$ and $\rho_5\in(0,\rho_4)$
such that if $x_0\in \Gamma^+(u)$
and $B( x_0, 4r_0)\subset \Omega$, then the harmonic competitor, $h_{x_0,\rho_4}$ (defined near
\eqref{a3.4}, with $r$ replaced by $\rho_4$), satisfies  
\begin{equation}\label{eqn3.11B}
(1-\epsilon)u(x)\le h_{ x_0,\rho_4}(x)\le (1+\epsilon)u(x)
\end{equation}
for $x\in U \cap B( x_0, \rho_5) = \{u>0\} \cap B( x_0, \rho_5)$. 
\end{corollary}

\begin{proof} This is a straightforward consequence of Theorem \ref{t3.1}.
Given $\epsilon>0$, choose $\rho_4<\rho_3$ (where $\rho_3$ is as in Theorem \ref{t3.1}),
so that in addition $5\rho_4^\beta<\epsilon$. Then choose $\rho_5=5\rho_4^{1+\alpha/17n}$,
and notice that \eqref{eqn3.11B} follows from \eqref{a3.44} with $z=x_0$.
As in the previous remarks, $\rho_3$ and $\beta$
depend only on the usual constants of Remark \ref{rmk3.2}.

\end{proof}

We end this section with a consequence of Theorem \ref{t3.1} and nondegeneracy estimates
from \cite{DT}.

\begin{corollary}\label{cor3.2}
Let  $\Omega$, $q_\pm$, a minimizer $u$ for $J$ or $J^+$, $r_0$, $\rho_3 \in (0,r_0)$, and
$x_0 \in \Gamma^+(u)$ such that $B(x_0,4r_0) \subset \Omega$, 
be as in the statement of Theorem \ref{t3.1}. Then 
\begin{equation}\label{eqn3.11BB}
c_{\mathrm{min}}\le 
\frac{1}{s}\fint_{\partial B(z,s) \cap U} h_{x_0,r}\, d\mathcal{H}^{n-1}
\le C_{\mathrm{max}}
\end{equation}
for $0 < s \leq r \leq \rho_3$ and $z \in \Gamma^+(u)$ such that 
$B(z,s)\subset B(x_0, 5r^{1+\alpha/17n})$.
Here the constants $0 < c_{\mathrm{min}} < C_{\mathrm{max}}$ depend only on 
$n$, $c_0$, $\|q_\pm\|_{L^\infty}$, $\kappa$, $\alpha$, $r_0$, 
and a bound on $\int_\Omega |\nabla u|^2$.
\end{corollary}

\begin{proof}
We have similar estimates for $u$, namely
\begin{equation}\label{eqn3.32A}
c'_{\mathrm{min}}\le 
\frac{1}{s}\fint_{\partial B(z,s)} u^+\, d\mathcal{H}^{n-1}
\le C'_{\mathrm{max}}.
\end{equation}
The upper bound holds because $u$ us locally Lipschitz (as in Theorems~5.1 and 8.1 in \cite{DT}),
and the lower bound is Lemma 10.3 in \cite{DT}.
Now we use Theorem \ref{t3.1} to show that (if $\rho_3$ is small enough)
$u/2 \leq h_{x_0,r} \leq 2u$ on $\partial B(z,s) \cap U$; the corollary follows.

\end{proof}

\section{Local uniform rectifiability of the free boundary}\label{unif-rect}

In this section we show that under the assumptions that 
$q_\pm\in L^\infty(\Omega)\cap C(\overline\Omega)$ and
$q_\pm\ge c_0>0$, the free boundary of almost-minimizers for $J$ or $J^+$ in $\Omega$
is locally Ahlfors-regular and uniformly rectifiable in $\Omega$.

In fact, given the local NTA property of $U = \big\{x\in \Omega \, ; \, u(x)>0 \big\}$ that was proved in 
Section~\ref{global} (see Theorem \ref{t2.3} or Corollary \ref{t2.4}), the hard part will be to prove 
the local Ahlfors regularity of $\Omega \cap \d U$. 
In the context of minimizers, as studied in \cite{AC}, \cite{ACF}, and others,
the distribution, $\Delta u$, which is a positive measure, plus maybe a controllable error term, is a good candidate for an Ahlfors regular measure supported on $\Omega \cap \d U$.
Here we cannot argue this way, because the almost-minimality of $u$ is not enough to control $\Delta u$, even
inside $U$. Instead we will show that the harmonic measure on $\d U$ is locally Ahlfors-regular, and for this
we will use the harmonic functions $h_{x_0,r}$ introduced in the previous section, plus the control on 
the ratio $\frac{h_{x_0,r}}{u}$ that we proved in that section. 

In the context of almost-minimizers, this result is new, 
and it opens the door to study higher regularity of the free boundary under additional assumptions on $q_\pm$.

\smallskip

Our assumptions for this section are the same as in Section \ref{harmonic-functions}.
We are given a bounded (connected) domain $\Omega \subset \R^n$, and two 
bounded functions $q_\pm$ that are continuous on $\Omega$.
We assume, as in \eqref{a3.1}, that $q_+\ge c_0>0$ on $\Omega$, and, as in \eqref{a3.2}, that 
$0 \leq q_- \leq q_+$ on $\Omega$ or $q_- \geq c_0 > 0$ on $\Omega$.

Under these assumptions, Theorem \ref{t2.3} says that $U$ is locally NTA in $\Omega$.
This implies that for every choice of $r_0 > 0$
we can find constants $C_1$, $C_2$, $C_3$, and also a radius $r_1 \in (0,r_0)$,
such if $x_0 \in \Gamma^+(u)$ is such that $B(x_0,r_0) \subset \Omega$, 
then for $0 < r \leq r_1$ we can find corkscrew points and Harnack chains as in Definition \ref{d2.3}.
In addition, we claim that $C_1$, $C_2$, $C_3$, and $r_1$ depend only on the usual constants of 
Remark \ref{rmk3.2}.

In particular we shall use the notation $A(x_0,r)$ for a corkscrew point for $U$, in $B(x_0,r)$; this means that
$A(x_0,r) \in B(x_0,r/2)$ and 
\begin{equation}\label{a4.1}
\dist(A(x_0,r), \d U) \geq C_1^{-1} r.
\end{equation}
By definition, such a point exists for $0 < r \leq r_1$. We then denote by 
$\omega^{A(x_0,r)}$ the harmonic measure on $\d U$, coming from the open set $U$
and the pole $A(x_0,r)$.

\begin{theorem}\label{t4.1} Let  $\Omega\subset \R^n$, $q_\pm \in L^\infty(\Omega)$,
and the almost-minimizer $u$ for $J$ or $J^+$ satisfy the assumptions above.
For each $r_0 > 0$, there exists $\rho_4 \in (0,r_0)$ and $C_5 \geq 0$ such that 
for any  $x_0\in \Gamma^+(u)$ with $B(x_0, 8r_0)\subset \Omega$,
\begin{equation}\label{eqn4.1}
C_5^{-1}r^{n-1}\le \omega^{A(x_0,\rho_4)}(B(z,r))\le C_5 r^{n-1}
\end{equation}
for all $z\in \Gamma^+(u) \cap B(x_0, r_0)$ and $0<r<\rho_4$.
\end{theorem}

\ms
In fact we shall choose $\rho_4 < r_1$, so $\omega^{A(x_0,\rho_4)}$ is well defined, and also so that
$\rho_4$ and $C_5$ depend only on the usual constants of Remark \ref{rmk3.2}, not on the specific
choices of $\Omega$, $q_\pm$, and $u$.

\begin{proof}
Let $r_0$ be as in the statement, and choose $\rho_4$ and $\rho_5 \in (0,\rho_4)$ as in Corollary \ref {cor3.3},
applied with $\varepsilon = 1/2$. The proof allows us to pick $\rho_4$ smaller if needed, at the expense of
taking $\rho_5$ even smaller. Thus we can assume that $\rho_4 < r_1$, for instance.
Since $B(x_0, 8r_0)\subset \Omega$, we can even apply the corollary to any 
\begin{equation}\label{a4.3}
z\in \Gamma^+(u) \cap B(x_0,4r_0).
\end{equation}
We get that for such $z$,
\begin{equation}\label{eqn4.2}
\frac{u(x)}{2} \le h_{ z,\rho_4}(x)\le \frac{3u(x)}{2}
\ \text{ for $x\in U \cap B(z, \rho_5)$.}
\end{equation}
Furthermore, under the hypothesis above, $u$ is locally Lipschitz and non-degenerate 
(see Theorems 5.1, 8.1 and 10.2 in \cite{DT}), so by \eqref{eqn4.2} there exists a constant $C>1$
such that 
\begin{equation}\label{eqn4.3}
C^{-1}\delta(x)\le h_{z,\rho_4}(x)\le C\delta(x)
\ \text{ for $x\in U \cap B(z,\rho_5)$,}
\end{equation}
where $\delta(x)=\dist(x,\Gamma^+(u))= \dist(x,\d U)$ (because the rest of $\d U$ is much further).
Here and below, $C$ is a constant that depends only on the usual constants of Remark \ref{rmk3.2}.
This allows $C$ to depend also on $r_0$, $\rho_4$, $\rho_5$, and the NTA constants for $U$
in $B(x_0,7r_0)$.

Set $A_0 = A(x_0,\rho_4)$ to simplify the notation, and denote by $G(A_0,\cdot)$
the Green function of $U\cap B(x_0, 4\rho_4)$ with pole $A_0$.
Also denote by $A(z,\rho_5)$ a corkscrew point for $U$ in $B(z,\rho_5)$.
Standard estimates for non-negative harmonic functions vanishing at the boundary of  NTA domains 
(see \cite{JK}, Lemma 4.10) 
ensure that there exists a constant $C>1$, depending only on $n$ and the local NTA constants 
for $U$, such that for $x\in B(z, \rho_5)\cap U$ 
\begin{equation}\label{eqn4.4}
C^{-1} \, \frac{G(A_0, A(z,\rho_5))}{h_{z, \rho_4}(A(z,\rho_5))}
\le \frac{G(A_0, x)}{h_{z, \rho_4}(x)}
\le C \, \frac{G(A_0, A(z,\rho_5))}{h_{z, \rho_4}(A(z,\rho_5))}.
\end{equation}
Notice that 
$\delta(A(z,\rho_5)) \geq C^{-1} \rho_5$, and $\delta(A_0) \geq C^{-1} \rho_4$,
so $C^{-1} \leq G(A_0, A(z,\rho_5)) \leq C$.
In addition, \eqref{eqn4.3} applies to $x=A(z,\rho_5)$, and yields 
$C^{-1} \leq h_{z, \rho_4}(A(z,\rho_5)) \leq C$. Thus by \eqref{eqn4.4} and \eqref{eqn4.3}
\begin{equation}\label{eqn4.5}
C^{-1}\le \frac{G(A_0, x)}{\delta(x)} \le C
\ \text{ for $x\in U \cap B(z,\rho_5)$.}
\end{equation}
A Caffarelli-Fabes-Mortola-Salsa estimate on NTA domains (see, e.g., \cite{JK}, Lemma 4.8)
ensures that for $z\in \Gamma^+(u) \cap B(x_0, \rho_4)$ (as in \eqref{a4.3}) and 
$0 < r \leq \rho_5$, 
\begin{equation}\label{eqn4.6}
C^{-1} \, \frac{G(A_0, A(z,r))}{r} 
\leq \frac{\omega^{A_0}(B(z,r))}{r^{n-1}} \leq C \, \frac{G(A_0, A(z,r))}{r}.
\end{equation}
We can apply \eqref{eqn4.5} with $x= A(z,r)$, because $x\in B(z,\rho_5)$. Since 
$C^{-1} r \leq \delta(x) \leq r$ by definition of a corkscrew point, \eqref{eqn4.5} and \eqref{eqn4.6}
yield
\begin{equation}\label{eqn4.7}
C^{-1}\le  \frac{\omega^{A_0}(B(z,r))}{r^{n-1}}\le C.
 \end{equation}
This is the same estimate as \eqref{eqn4.1}, but we only proved it for $0 < r \leq \rho_5$.
But for $\rho_5 \leq r \leq \rho_4$, 
\begin{equation}\label{a4.10}
\omega^{A_0}(B(z,\rho_5)) \leq \omega^{A_0}(B(z,r)) \leq \omega^{A_0}(B(z,\rho_4))
\leq C \omega^{A_0}(B(z,\rho_5)),
\end{equation}
where $C$ depends on $\rho_5$, $\rho_4$, and the local doubling constant for $\omega^{A_0}$,
which itself depends on the local NTA constants for $U$ (and finally the usual constants).
Since the factor $r^{n-1}$ does not vary too much either, the general case of \eqref{eqn4.1}
follows, and this yields Theorem \ref{t4.1}.  
\end{proof}

We are now ready to prove the local Ahlfors-regularity and the local uniform rectifiability of the 
free boundary, with big pieces of Lipschitz graphs. Let us first recall the notion of uniform rectifiability.

\begin{defn}\label{d:urandbplg}
Let $E \subset \mathbb R^n$ be a $d$-Ahlfors regular set. We say that $E$ is $d$-{\bf uniformly rectifiable} if there exists an $L > 0$ and a $\theta \in (0,1)$ such that for all $x \in E$ and $r > 0$ there is an $L$-Lipschitz function $f_{x,r}: B(0,r) \subset \mathbb R^d \rightarrow \mathbb R^n$ such that \begin{equation}\label{e:bpli} \mathcal H^d(f_{x,r}(B(0,r))\cap E\cap B(x,r)) \geq \theta r^d.\end{equation} 
\end{defn}

The condition of \eqref{e:bpli} is often referred to as ``big pieces of Lipschitz images." In Theorem \ref{t4.2} we would like to prove something stronger, namely, the existence of ``big pieces of Lipschitz graphs". We shall use the fact that for global unbounded Ahlfors regular sets, 
the additional property known as ``Condition B'' implies the existence of big pieces 
of Lipschitz graphs (also known as BPLG). Let us state this formally.

Let $E \subset \mathbb R^n$ be (unbounded) Ahlfors regular.
This means that $E$ is closed (nonempty), and there is a constant $C_0 \geq 1$ such that
\begin{equation}\label{a4.16}
C_0^{-1} t^{n-1} \leq \H^{n-1}(E \cap B(y,t)) \leq C_0 t^{n-1}
\ \text{ for $y\in E$ and $t >0$.}
\end{equation}
We say that $E$ satisfies Condition B if there is a constant $C_1 \geq 1$
such that, for $y\in E$ and $t>0$, we can find two points $y_1 = y_1(y,t)$ and $y_2 = y_2(y,t)$,
that lie in different connected components of $\R^n \sm E$, and such that
\begin{equation}\label{a4.17}
y_i \in B(y,t) \text{ and } \dist(y_i,E) \geq C_1^{-1} t
\ \text{ for } i=1,2.
\end{equation}

\begin{theorem}\label{t4.3} 
If $E$ is an unbounded Ahlfors regular set that satisfies Condition B, 
then there exist constants $C_7$ and $C_8$, that depend only on $n$, $C_0$, and $C_1$ above, 
such that for $y\in E$ and $t > 0$, we can find a $C_7$-Lipschitz graph $G = G(y,t)$ such that
\begin{equation}\label{eqn4.18}
\H^{n-1}(B(y,t) \cap E\cap G) \geq C_8^{-1} {t^{n-1}}.
\end{equation}
\end{theorem}

\ms
By $C_7$-Lipschitz graph, we mean a set of the form
$G =\big\{ x+A(x) \, ; \, x\in P \big\}$, where $P$ is a hyperplane in $\R^n$
and $A : P \to P^\perp$ is a $C_7$-Lipschitz function from $P$ to its
orthogonal complement $P^\perp$. 

Theorem \ref{t4.3} is proved in \cite{Morceaux}, but a simpler proof can be found in \cite{DaJ}.
Recall also that Condition~B was introduced by S. Semmes in \cite{Se}, who proved the uniform
rectifiability of $E$ under mild additional assumptions (but with estimates that do not
use these assumption).

Theorem \ref{t4.2} and its proof can be understood 
as a local version of \cite{DaJ}.
For the readers' convenience we present a self contained proof which only relies on Theorem \ref{t4.3} 
(see also \cite{Se} and \cite{DaJ}). 

\begin{theorem}\label{t4.2} 
Let  $\Omega\subset \R^n$, $q_\pm \in L^\infty(\Omega)$,
and $u$ satisfy the assumptions of Theorem~\ref{t4.1}. 
That is, $\Omega$ is open, bounded, and connected, $q_+$ and $q_-$ are bounded, continuous, 
and satisfy the nondegeneracy condition \eqref{a3.1} and \eqref{a3.2}, 
and $u$ is an almost-minimizer for $J$ or $J^+$ in $\Omega$.
For each $r_0 > 0$, we can find constants $C_6$, $C_7$, and $C_8$ such that 
for $x\in \Gamma^+(u)$ and $r$ such that
\begin{equation}\label{a4.11}
B(x, 11r_0)\subset \Omega \ \text{ and } \ 0 < r \leq r_0,
\end{equation}
we have 
\begin{equation}\label{a4.12}
C_6^{-1}r^{n-1} \leq \H^{n-1}(\Gamma^+(u) \cap B(x,r))\le C_6 r^{n-1}
\end{equation}
and there exists a $C_7$-Lipschitz graph $G = G(x,r)$ such that
\begin{equation}\label{eqn4.10}
\H^{n-1}(B(x,r) \cap \Gamma^+(u) \cap G) \geq C_8^{-1} {r^{n-1}}.
\end{equation}
In addition, $C_6$, $C_7$, and $C_8$ depend only on the usual constants of Remark \ref{rmk3.2}
\end{theorem}

\ms

We shall even prove that, when $x \in \Gamma^+(u)$ and $r>0$ are as in the statement,
there is a uniformly rectifiable set $E(x,r)$, with big pieces of Lipschitz graphs
(and with constants that depend only on the usual constants) such that 
\begin{equation}\label{a4.14}
\Gamma^+(u) \cap B(x,r) \subset E(x,r).
\end{equation}
This essentially amounts to the same thing, but seems a little more precise, and in particular it
allows us to use the classical results on uniformly rectifiable sets directly, without having to localize the proofs.
In particular, we get that $\Gamma^+(u)$ is rectifiable (but lose a lot of information when we say this).
See Remark \ref{rmk4.1} below.

\begin{proof}
We start with the local Ahlfors regularity property \eqref{a4.12}.
We decided to restate it in terms of the Hausdorff measure because it is more usual,
but it is a simple consequence of the local existence of some Ahlfors regular measure on 
$\Gamma^+(u)$, namely the harmonic measure of Theorem~\ref{t4.1}.
That is, if $u$ and $r_0$ are as in the statement, we found for each $x_0 \in \Gamma^+(u)$
such that $B(x_0,8r_0)$ a measure $\omega$ such that \eqref{eqn4.1} holds
for $z\in \Gamma^+(u) \cap B(x_0, r_0)$ and $0<r<\rho_4$. By a simple covering argument,
we can prove that $\omega$ is equivalent to $\H^{n-1}$ on $\Gamma^+(u) \cap B(x_0, r_0/2)$,
and more precisely that 
\begin{equation}\label{a4.15}
C^{-1} \omega(E) \leq \H^{n-1}(E) \leq C \omega(E)
\end{equation}
for every Borel set $E \subset \Gamma^+(u) \cap B(x_0, r_0/2)$, where $C$ depends only on $n$ and $C_5$.
See for instance Lemma 18.11 and its proof in Exercise 18.25 (on page 112) of \cite{MSbook}, 
but there was no claim for novelty there.

From \eqref{a4.15} and \eqref{eqn4.1} we now deduce that $\eqref{a4.12}$ holds for 
$x\in \Gamma^+(u)$ such that $B(x,8r_0)$ and $0 < r < \rho_4$. For the remaining radii,
$r \in (\rho_4,r_0)$, and at the price of making $C_6$ outrageously larger, we just say that 
$\H^{n-1}(\Gamma^+(u) \cap B(x,r)) \geq \H^{n-1}(\Gamma^+(u) \cap B(x,\rho_4))$
to get a (rather bad) lower bound, and (if now $B(x,9r_0) \subset \Omega$)
we cover $\Gamma^+(u) \cap B(x,r_0)$ by less than $C$ balls $B(z,\rho_4)$, 
$z\in \Gamma^+(u) \cap B(x,r_0)$, to get an upper bound. So $\eqref{a4.12}$ holds.

To prove \eqref{eqn4.10} we apply Theorem~\ref{t4.3} using a short localization argument, which is rather straightforward in co-dimension $1$ 
(our setting here).
We want to apply Theorem~\ref{t4.3} to some auxiliary Ahlfors-regular set, $E$.
Let $u$ and $r_0$ be as in the statement of Theorem~\ref{t4.2}, and let
$x\in \Gamma^+(u)$ and $r > 0$ be such that \eqref{a4.11} holds. 
Set $B = B(x,2r)$, choose a hyperplane, $P$, such that 
$\dist(x,P) = 10 r$, and take
\begin{equation}\label{eqn4.12}
E = [B \cap \Gamma^+(u)] \cup \partial B \cup P.
\end{equation}
We added $P$ to get an unbounded set $E$, but we easily see that it could not disturb in
the proofs or conclusions. We want to show that $E$ is Ahlfors-regular and satisfies Condition $B$.

Set $\Gamma = \Gamma^+(u)$ to simplify the notation. Notice that 
\begin{equation}\label{a4.20}
\H^{n-1}(\Gamma \cap B) \leq C r^{n-1},
\end{equation}
even if $2r > r_0$, because in this case we can cover $\Gamma \cap B$ by less than $C$ balls
$B(z,r_0)$, with $z\in B$, and \eqref{a4.12} also holds for $z \in B$, because $B(z,9r_0) \subset \Omega$
since $B(x,11r_0) \subset \Omega$.
Next we claim that 
\begin{equation}\label{a4.21}
\H^{n-1}(\Gamma \cap B \cap B(y,t)) \leq C t^{n-1}
\end{equation}
for $y\in \R^n$ and $t > 0$. When $t \geq r/2$, this follows from \eqref{a4.20}. Otherwise, 
even if $y$ does not lie in $\Gamma$, we get \eqref{a4.21} because if 
$\Gamma \cap B\cap B(y,t) \neq \emptyset$, we can find $z \in \Gamma \cap B\cap B(y,t)$, 
then $B(y,t) \subset B(z,2t)$, we can apply \eqref{a4.12} to $z$, and we get\eqref{a4.21}.

Now the upper bound in \eqref{a4.16} follows, because 
$\H^{n-1}(\partial B \cap B(y,t)) + \H^{n-1}(P \cap B(y,t)) \leq C t^{n-1}$ trivially.

For the lower Ahlfors regularity bound, we distinguish between cases.
When $y\in P$ or $\dist(y,P) \leq t/2$, we just need to observe that
$\H^{n-1}(E \cap B(y,t)) \geq \H^{n-1}(P \cap B(y,t)) \geq C^{-1} t^{n-1}$.
Thus we may assume that $y\in [B \cap \Gamma] \cup \d B$ and  $t \leq 20r$.

When $y\in \d B$, or even $\dist(y, \d B) \leq t/2$, we just observe that 
$\H^{n-1}(E \cap B(y,t)) \geq \H^{n-1}(\d B \cap B(y,t)) \geq C^{-1} t^{n-1}$.
So we are left with $y\in \Gamma \cap B$ such that $\dist(y, \d B) > t/2$. 
But then $\H^{n-1}(E \cap B(y,t)) \geq \H^{n-1}(\Gamma \cap B(y,t/20)) \geq C^{-1} t^{n-1}$.
directly by \eqref{a4.12}. So $E$ is Ahlfors regular.

Now we check Condition $B$. Let $y\in E$ and $t > 0$ be given; we want to find
points $y_1$ and $y_2$ as in \eqref{a4.17}. We start with the most interesting
case when $y\in \Gamma \cap B$ and $\dist(y,\d B) \geq t$. In this scenario, we need not consider $\d B$ and $P$, we simply use the local NTA property of $\Gamma$, which is given
by Theorem \ref{t2.3}; we proceed as in the beginning of this section, apply the theorem
with $K = \big\{ x\in \Omega \, ; \, \dist(x,\d \Omega) \geq r_0 \big\}$, and get a radius $r_1 > 0$
such that for $y \in \Gamma \cap K$ and $0 < r \leq r_1$, we can find corkscrew points for
$U$ and for $\big\{ x\in \Omega \, ; \, u(x) \leq 0 \big\}$, inside $B(y,r)$ (see Definition \ref{d2.3}).
If $t \leq r_1$, we simply take for $y_1$ and $y_2$ these two corkscrew points. Notice
that $\Gamma$ separates $y_1$ from $y_2$ in $\Omega$ (by the intermediate value theorem),
hence also in $B$. Thus $y_1$ and $y_2$ lie in different components of 
$\R^n \sm [(\Gamma \cap B) \cup \d B]$, as needed.

The next interesting case is when $y\in \d B$ and $0 \leq t \leq r$. We easily find 
$y_2 \in B(y,t)$ such that $\dist(y_2, P\cup B) \geq 10^{-1} t$, so it is enough to find
$y_1 \in B(y,t) \cap B$, such that $\dist(y_1, \d B) \geq 10^{-1} t$ but also
$\dist(y_1,\Gamma) \geq C^{-1} t$, because $E \supset \d B$ will automatically separate
$y_1$ from $y_2$. Let $\tau>0$ be small, to be chosen soon; we can easily find $C^{-1} \tau^{-n}$ points
$w_i \in \big\{ w\in B \cap B(y,t/2) \, ; \, \dist(y_1, \d B) \geq 10^{-1} t \big\}$,
that lie at distances larger than $4\tau t$ from each other. Suppose all the $B(w_i,\tau t)$
meet $\Gamma$; then $\H^{n-1}(\Gamma \cap B(w_i,2\tau t)) \geq C_6^{-1} (\tau t)^{n-1}$
by the lower bound in \eqref{a4.12}, and since all these balls are disjoint and contained in $B(y,t)$, 
we get that $\H^{n-1}(\Gamma \cap B(y,t)) \geq C^{-1} \tau^{-n} (\tau t)^{n-1}$. On the other hand,
the upper bound \eqref{a4.12} yields $\H^{n-1}(\Gamma \cap B(y,t)) \leq C t^{n-1}$, and if $\tau$
is chosen small enough we get a contradiction. Thus we can find $w_i$ such that 
$\dist(w_i,\Gamma) \geq \tau t$, and use this $w_i$ as $y_1$. 
This settles our second case when $y\in \d B$ and $0 \leq t \leq r$.

When $y \in \d B$ and $r \leq t \leq 20 r$, we can still use the points 
$y_i$ that work for $t=r$, and we get \eqref{a4.17} with a constant $20$ times larger.
When $y\in P$ but $t \leq 20r$, we simply select two points $y_i \in B(y,t/10)$,
that lie on different sides of $P$ and at distance at least $t/100$ from $P$.
They also lie far from the rest of $E$, because $\dist(y,E\sm P) \geq 8r$.
Similarly, when $y\in E$ but $t \geq 20r$, we pick a point $z\in P \cap B(y,t/2)$
such that $\dist(z,B) \geq t/4 \geq 5r$, and then select two points $y_i \in B(z,t/10)$
that lie on different sides of $P$, but at distance $t/100$ from $P$. Again they
are also far from $E \sm P$. 

We are only left with the case when $y \in \Gamma \cap B$ and $t \leq 20r$.
We already treated the case when $t \leq \dist(y,\d B)$. When $\dist(y,\d B) < t \leq 10\dist(y,\d B)$,
we may just use the two points $y_i$ that work for $t=\dist(y, \d B)$ (and get a larger constant).
Finally, when $10\dist(y,\d B) \leq t \leq 20r$, we select a point $z\in \d B$ such that
$|z-y| = \dist(y,\d B)$, and use the points $y_1$ and $y_2$ that correspond to the pair
$(z, t/2)$. This completes our verification of Condition B for $E$.

We apply Theorem \ref{t4.3} and get that $E$ contains big pieces of Lipschitz graphs, as in 
\eqref{eqn4.18}. The constants $C_7$ and $C_8$ depend on $n$, and $C_0$ and $C_1$ for $E$,
which themselves depend only on the usual constants of Remark \ref{rmk3.2}.

This already proves our claim relative to \eqref{a4.14}, but if we apply the conclusion of 
Theorem~\ref{t4.3} to $E$ and the ball $B(x,r)$, we get a Lipschitz graph $G$
that satisfies \eqref{eqn4.10}, just because $E \cap B(x,r) = \Gamma \cap B(x,r)$.
This completes our proof of Theorem \ref{t4.2}.
\end{proof}

\begin{remark}\label{rmk4.1} 
As we said near \eqref{a4.14}, it may be easier to use the existence of $E = E(x,r)$,
to derive information on $\Gamma^+(u)$ from similar information on the uniformly rectifiable
set $E$. Also, we said that $\Gamma^+(u)$ is rectifiable, and this is true, for instance, because
all our sets $E$ are rectifiable. Indeed, call $E_r$ and $E_u$ the rectifiable and unrectifiable parts of $E$
(known modulo a set of vanishing $\H^{n-1}$-measure). If $\H^{n-1}(E_u) > 0$,
then by a standard density argument (see for instance \cite{Ma}) we can find $y \in E_u$
such that $\lim_{t \to 0} t^{1-n} \H^{n-1}(E_r \cap B(y,t)) = 0$. This is impossible,
because almost every point of $G(y,t) \cap B(y,t) \cap E$ lies in $E_r$.
\end{remark}

\section{A Weiss Monotonicity formula}
\label{WeissMF}

The first result of this section is an extension of a monotonicity formula due to Weiss \cite{W}, 
who showed that the functional below is monotone when $u$ is a local minimizer of $J$ or $J^+$ in the sense of \cite{AC} or \cite{ACF}. Recalling that almost-minimizers are locally Lipschitz, the proof in \cite{W} works essentially unchanged for almost-minimizers. We quickly summarize the necessary changes below.

\begin{theorem}\label{thm:monotonicityone}[c.f. Theorem 1.2 in \cite{W}].
Let $u$ be an almost-minimizer for $J$ in the open set $\Omega \subset \R^n$, 
with constant $\kappa$ and exponent $\alpha$. Also let $x_0 \in \Omega$ and $R > 0$
be such that $u(x_0) = 0$ and $\overline B(x_0,R) \subset \Omega$.
Further assume that $q_+$ and $q_-$ are H\"older continuous on $B(x_0,R)$, with exponent $\alpha$.
Define, for $\rho \leq R$, 
\begin{eqnarray}\label{eqn:mono1} 
\widetilde{W}(u, x_0, \rho) 
&=& \frac{1}{\rho^n}\int_{B(x_0,\rho)} |\nabla u|^2 
+ q^2_+(x_0)\1_{\{u > 0\}} + q^2_-(x_0) \1_{\{u < 0\}} 
\nonumber\\
&\,& \hskip3cm
- \frac{1}{\rho}\int_0^\rho \frac{1}{r^{n-1}} \int_{\partial B(x_0,r)} (\nabla u \cdot \nu)^2 d\mathcal H^{n-1}dr.
\end{eqnarray}
Then, there exists $C > 0$, which depends only on $\alpha$, $n$, $\kappa$, the norms 
$\|q_{\pm}^2\|_{L^\infty(B(x_0,R))}$ and $\|q_{\pm}^2\|_{C^{0,\alpha}(B(x_0,R))}$, 
and the Lipschitz norm of $u$ in $B(x_0,R)$, such that for $0 < s < \rho < R$
\begin{eqnarray}\label{a5.2} 
0 &\leq& \int_s^\rho t^{-3}\int_{\partial B(0,1)} \left[t\int_0^t (\nabla u(x_0 + r\xi) \cdot \xi)^2dr 
- \left(\int_0^t \nabla u(x_0 + r\xi)\cdot \xi dr \right)^2\right] d\H^{n-1}(\xi)dt
\nonumber\\
&\,& \hskip3cm
\leq \widetilde{W}(u, x_0, \rho) - \widetilde{W}(u, x_0, s) + C \rho^\alpha.
\end{eqnarray}
\end{theorem}

For simplicity, we assumed here that the $q_\pm$ are H\"older continuous, and even 
with the same exponent $\alpha$ as in the definition of almost-minimizers; 
otherwise we could take the smallest exponent or modify slightly the estimates.
Also, if instead we only assumed that the $q_\pm$ are continuous on $\Omega$, we would get a
similar result, except that we should add an extra term like
$C \sup_{y\in B(x_0,\rho)} \big(|q_+(y)-q_+(x_0)| + |q_-(y)-q_-(x_0)| \big)$, where $C$ depends also
on the Lipschitz norm of $u$. 

Recall that the first inequality comes directly from Cauchy-Schwarz; the main information is the second one.

Finally, we decided to use the Hausdorff measure $d\H^{n-1}$ in the statement, but we shall also write this 
measure $d\sigma$, at least when we work on a sphere. This will be our definition of surface measure.

\begin{proof}
Without loss of generality we can let $x_0 = 0$.
In the proof of Theorem 1.2 in \cite{W}, Weiss defines $u_t$, for $t \in (0,R]$, by
$$
u_t(x) = \frac{|x|}{t}u\left(t\frac{x}{|x|}\right) \ \text{ for } x \in B_t = B(0,t))
$$ 
and $u_t(x) = u(x)$ outside of $B_t$. 
Taking the derivative we can see that the Lipschitz continuity of $u$ implies the Lipschitz continuity of $u_t$. 
Hence $u_t$ is a competitor for $u$. 
Since $u$ is an almost-minimizer, and with an implicit summation in $\pm$ to shorten the expressions,
\begin{eqnarray}\label{eqn:comparetout}
0 &\leq& (1+\kappa t^\alpha)\int_{B_t}|\nabla u_t|^2 
+ \1_{\{\pm u_t > 0\}}\, q^2_\pm(x) 
\ dx 
- \int_{B_t}|\nabla u|^2 + \1_{\{\pm u > 0\}} \, q^2_\pm(x) 
\nonumber\\
&\leq& \int_{B_t}|\nabla u_t|^2 + \1_{\{\pm u_t > 0\}} \, q^2_\pm(0) \ dx 
- \int_{B_t}|\nabla u|^2 + \1_{\{\pm u > 0\}} \, q^2_\pm(0) \ dx + A t^{n+\alpha},
\end{eqnarray}
with $A = \kappa t^{-n}\int_{B_t} |\nabla u_t|^2 + C \kappa \|q_{\pm}^2\|_{L^\infty(B_R)}
+ C \sup_{B_t} |q_\pm(x)-q_\pm(0)|$. 
It is easy to see that
$A \leq C \kappa ||u||^2_{\mathrm{Lip}(B_R)} + C \kappa \|q_{\pm}^2\|_{L^\infty(B_R)} 
+ C \|q_{\pm}^2\|_{C^{0,\alpha}(B_R)}$.
We compute the integrals of $\nabla u_t$ and $\1_{\{\pm u_t > 0\}}$ as in \cite{W}, and
deduce from \eqref{eqn:comparetout} that for almost every $t \in (0,1)$, 
\begin{eqnarray}\label{eqn:afterutcomputations} 
0 &\leq& \frac{t}{n} \int_{\partial B_t} |\nabla u|^2  + \1_{\{\pm u > 0\}} \, q_\pm^2(0) \ d\sigma 
- \int_{B_t} |\nabla u|^2 + \1_{\{\pm u > 0\}} \, q^2_\pm(0) \ dx
\nonumber\\
 &\,& \hskip2cm
 +A t^{n+\alpha}  + \frac{1}{nt}\int_{\partial B_t} u^2\ d\sigma 
 -\frac{t}{n}\int_{\partial B_t}(\nabla u \cdot \nu)^2\ d\sigma,
\end{eqnarray}
where $\nu$ denotes the unit normal to $\d B_t$.
The proof then proceeds exactly as in \cite{W} to produce the desired result. 
\end{proof}

Here we gave the result for an almost-minimizer for $J$, but the same result holds, with the same proof, when 
$u$ is an almost-minimizer of $J^+$ (and we set $q_- = 0$). We call this the associated monotonicity formula 
for $\widetilde{W}^+$. 

As it is difficult to control the integral of the normal derivative of $u$ on $\d B_t$, 
$\widetilde{W}$ is not well suited to our purposes. 
However, $\widetilde{W}$ is related to a similar, and easier to work with, monotonicity formula. Set
\begin{equation}\label{eqn:mono2}
W(u, x_0, r) \equiv \frac{1}{r^n} \int_{B(x_0,r)} |\nabla u|^2 
+ q^2_+(x_0)\1_{\{u > 0\}} + q^2_-(x_0)\1_{\{u < 0\}} 
- \frac{1}{r^{n+1}} \int_{\partial B(x_0,r)} u^2 d\sigma,
\end{equation}
where we just take $q_- =0$ or remove $q^2_-(x_0)\1_{\{u < 0\}}$ when we work with $J^+$.
This formula appears in \cite{W}, where it is shown to be monotone increasing for local minimizers of $J$ 
in the sense of \cite{ACF}. The proof there uses that the minimizers of $J$ satisfy an equation, something which is not true for almost-minimizers. Instead, our proof will relate $\widetilde{W}$ and $W$,
and then use the almost-monotonicity of $\widetilde{W}$ to prove the almost-monotonicity of $W$. 

\begin{proposition}\label{prop:monotonicitytwo}
Let $u$ be an almost-minimizer for $J$ or $J^+$ in $\O$, 
with constant $\kappa$ and exponent $\alpha$. 
Suppose that the $q^{\pm}$ are bounded and H\"older continuous on $B(x_0,R)$, with exponent $\alpha$.
Furthermore let $x_0 \in \Omega$ and $R > 0$ be such that $u(x_0) = 0$ and $\overline B(x_0,R) \subset \Omega$.
Then for $0 < s < \rho < R$,
\begin{equation}\label{a5.6}
W(u, x_0,\rho) - W(u, x_0, s) \geq -C \rho^\alpha + 
\int_s^\rho \frac{1}{t^{n+2}} \int_{\partial B(x_0,t)} \left(u(x)- (\nabla u(x)\cdot x)\right)^2 d\sigma dt,
\end{equation}
where $C> 0$ depends only on $n$, $\kappa$, $\alpha$, the norms 
$\|q_{\pm}^2\|_{L^\infty(B(x_0,R))}$ and $\|q_{\pm}^2\|_{C^{0,\alpha}(B(x_0,R))}$, 
and the Lipschitz norm of $u$ in $B(x_0,R)$. 
\end{proposition}

\begin{proof}
Again we may assume that $x_0 = 0$.
We write the right-hand side of \eqref{a5.2} as $A-B$ and compute 
 \begin{eqnarray}\label{eqn:firsttermmanip}
A &=& \int_s^\rho t^{-3}\int_{\partial B(0,1)} t\int_0^t (\nabla u(r\xi) \cdot \xi)^2 dr d\sigma(\xi)dt
\nonumber \\
&=& \int_s^\rho t^{-2} \int_0^t \int_{\partial B(0,1)}(\nabla u(r\xi)\cdot \xi)^2 d\sigma(\xi)drdt
\\ 
&=& \int_s^\rho t^{-2} \int_0^t \frac{1}{r^{n-1}}
\int_{\partial B(0,r)}(\nabla u\cdot \nu)^2 d\sigma drdt.
\nonumber 
\end{eqnarray}
Since $u(x_0) = 0$,  
\begin{eqnarray}\label{eqn:secondtermmanip}
B &=&
\int_s^\rho t^{-3}\int_{\partial B(0,1)} 
\left(\int_0^t \nabla u(r\xi)\cdot \xi dr \right)^2 d\sigma(\xi)dt 
 \nonumber\\
 &=& \int_s^\rho \frac{1}{t^3}\int_{\partial B(0,1)} (u(r\xi) - u(0))^2d\sigma(\xi) dt
 = \int_s^\rho 
 \frac{1}{t^{n+2}} \int_{\partial B(0,t)} u(x)^2 d\sigma(x) dt.
\end{eqnarray}
Set $F(t) = W(u, x_0, t) - \widetilde{W}(u,x_0,t)$ for a moment.
Thus by \eqref{eqn:mono2} and \eqref{eqn:mono1}
\begin{equation}\label{a5.9}
F(t) 
= -\frac{1}{t^{n+1}} \int_{\partial B(0,t)} u^2 d\sigma
+\frac{1}{t}\int_0^t \frac{1}{r^{n-1}} \int_{\partial B(0,r)} (\nabla u \cdot \nu)^2 d\sigma dr
\end{equation}
and now Theorem \ref{thm:monotonicityone} yields
\begin{equation}\label{eqn:differenceintildes}
\widetilde{W}(u, x_0, \rho) - \widetilde{W}(u, x_0, s) + C \rho^\alpha \geq A-B
= \int_s^\rho \frac{1}{t} F(t) dt 
\end{equation}
by \eqref{eqn:firsttermmanip}, \eqref{eqn:secondtermmanip}, and \eqref{a5.9}.
Hence
\begin{eqnarray}\label{eqn:differenceinregular}
W(u, x_0,\rho) - W(u, x_0, s) &=& F(\rho) - F(s) + \widetilde{W}(u, x_0, \rho) - \widetilde{W}(u, x_0, s)
\nonumber\\
 &\geq& F(\rho) - F(s) + \int_s^\rho \frac{1}{t} F(t) dt - C \rho^\alpha.
\end{eqnarray}
We shall see soon that $F$ has a derivative almost everywhere, and is the integral of $F'$.
That is, $F(\rho) - F(s) = \int_s^\rho F'(t) dt$, and hence
\begin{equation}\label{a5.12}
W(u, x_0,\rho) - W(u, x_0, s) \geq - C \rho^\alpha + \int_s^\rho \left(F'(t) + \frac{F(t)}{t}\right)  dt
\end{equation}
Next we compute $F'(t)$; notice first that by \eqref{a5.9}
\begin{equation}\label{a5.13}
F(t) = - \frac{1}{t^2}\int_{\partial B(0,1)} u(t\xi)^2 d\sigma(\xi)
+ \frac{1}{t} \int_0^t \int_{\partial B(0,1)} (\nabla u(r\xi)\cdot \xi)^2 d\sigma(\xi) dr.
\end{equation}
Write $F(t) = - t^{-2} G(t) + t^{-1} H(t)$, with  
\begin{equation}\label{a5.14}
G(t) = \int_{\partial B(0,1)} u(t\xi)^2 d\sigma(\xi) \ \text{ and } \ 
H(t) = \int_0^t \int_{\partial B(0,1)} (\nabla u(r\xi)\cdot \xi)^2 d\sigma(\xi) dr.
\end{equation}
Then 
\begin{equation}\label{a5.15}
G'(t) = 2\int_{\partial B(0,1)} u(t\xi)(\nabla u(t\xi) \cdot \xi) d\sigma(\xi)
\end{equation}
and 
\begin{equation}\label{a5.16}
 H'(t) = \int_{\partial B(0,1)} (\nabla u(t\xi)\cdot \xi)^2 d\sigma(\xi),
\end{equation}
so 
\begin{eqnarray}\label{a5.17}
t^{-1}F(t) + F'(t) &=&  t^{-1}F(t) + 2t^{-3} G(t) - t^{-2} G'(t) - t^{-2} H(t) + t^{-1} H'(t)
\nonumber\\
 &=& t^{-3} G(t) - t^{-2} G'(t) + t^{-1} H'(t)
 \nonumber\\
&=& \frac{1}{t}\int_{\partial B(0,1)} 
\left\{\left(\frac{u(t\xi)}{t}\right)^2 - \frac{2}{t}u(t\xi)(\nabla u(t\xi) \cdot \xi)
+ (\nabla u(t\xi)\cdot \xi)^2 \right\} d\sigma(\xi)
\\
&=&  \frac{1}{t} 
\int_{\partial B(0,1)} \left(\frac{u(t\xi)}{t} - (\nabla u(t\xi)\cdot \xi)\right)^2d\sigma(\xi) \geq 0.
\nonumber
\end{eqnarray}
We promised to return to the absolute continuity of $F$. Notice that both $G$ and $H$ are 
the indefinite integrals of their derivative, essentially by Fubini. Then multiplying them by 
$t^{-2}$ or $t^{-1}$ does not change this (away from $t=0$). This is rather standard and easy;
for instance write $G$ as the integral of $G'$, multiply by $t^{-2}$, and perform a soft integration by part
using Fubini. Thus \eqref{a5.12} holds, and by \eqref{a5.17} we get that
\begin{eqnarray}\label{a5.18}
W(u, x_0,\rho) - W(u, x_0, s) &\geq& - C \rho^\alpha 
+ \int_s^\rho \frac{1}{t}\int_{\partial B(0,1)} \left(\frac{u(t\xi)}{t} 
- (\nabla u(t\xi)\cdot \xi)\right)^2d\sigma(\xi) dt
\nonumber \\
&=& -C \rho^\alpha + 
\int_s^\rho \frac{1}{t^{n+2}} \int_{\partial B(x_0,t)} \left(u(x)- (\nabla u(x)\cdot x)\right)^2 
d\sigma(x) dt,
\end{eqnarray}
as announced in \eqref{a5.6}. The proposition follows.
\end{proof}

Before we examine the consequences of the monotonicity formula, let us make a quick observation 
concerning the case when $W(u, x_0, \cdot)$ is constant.

\begin{lemma}\label{rem:wconstantimplieshomogenous}
Suppose $q_+$ and $q_-$ are constant on $\Omega$, and let $u$ be a minimizer for 
$J$ or $J^+$ on $\Omega$. Suppose that $0 \in \Omega$, $u(0) = 0$, and
$0 < s < \rho < \dist(x_0,\d \Omega)$. Then $W(u, 0,\rho) - W(u, 0, s) = 0$ 
if and only if $u$ is homogeneous of degree $1$ in $B(0, \rho)\sm B(0, s)$. 

Furthermore, if $q_+$ and $q_-$ are constant on $\R^n$, 
$u$ is a minimizer for $J$ or $J^+$ in $\R^n$, and $u$ is homogeneous of degree $1$, 
then for $r > 0$
\begin{equation}\label{eqn:wononehomog}
W(u, 0, r) = W(u,0, 1) =  q_+^2 |B(0,1) \cap \{u > 0\}| + q_-^2|B(0,1) \cap \{u < 0\}|.
\end{equation}
\end{lemma}

\begin{proof}
If $u$ is a minimizer and the $q_{\pm}$ are constant, then by \eqref{a5.6}
$$ 
W(u, 0,\rho) - W(u, 0, s) \geq \int_s^\rho \frac{1}{t^{n+2}} \int_{\partial B_t(0)} \left(u(x)- (\nabla u(x)\cdot x)\right)^2d\sigma dt.
$$ 
If in addition $W(u, 0,\rho) - W(u, 0, s) = 0$, then $u(x) = \nabla u(x)\cdot x$ for almost 
every $x \in B(0, \rho)\sm B(0, s)$. The first part follows by integrating along rays.
It is well known (see for instance, Theorem 4.5.2 in \cite{AC} or Theorem 2.2 in \cite{ACF}) that if $u$ is a minimizer, then $u\Delta u = 0$ as a distribution. 
Therefore, an integration by parts implies that 
$$\int_{B(0,1)} |\nabla u|^2\ dx  = \int_{\partial B(0,1)} u^2\ d\sigma$$ 
and \eqref{eqn:wononehomog} follows. 
\end{proof}

\section{Consequences of the Weiss Monotonicity formula} 
\label{Consequences}

Throughout this section we assume that for some choice of $c_0, \alpha > 0$,
\begin{equation}\label{a6.1}
q_\pm\in L^\infty(\Omega)\cap C^\alpha(\Omega) \ \text{ and }\ 
q_\pm\ge c_0>0,
\end{equation}
but rather rapidly we shall concentrate on almost-minimizers for $J^+$, and thus work with $q_+$ alone,
and use the monotonicity formula of the previous section to detect points where the free boundary 
is infinitesimally flat.  We shall call these points ``regular" and denote the corresponding set 
by  $\mathcal R$ (see Definition \ref{defn:flatpoints} below). 
A key component of this analysis will be the identification of the blowup limits of almost-minimizers. 

We start with a few definitions. 
Set $U = \big\{x\in \Omega \, ; \, u(x) > 0 \big\}$ and $\Gamma^+(u) = \Omega \cap \d U$
as usual. For $x_0 \in \Gamma^+(u)$ and $r > 0$, define 
\begin{equation}\label{eqn:blowups}
u_{r,x_0}(x) = r^{-1} u(rx+ x_0)
\end{equation}
If $\{ r_j \}$ is a sequence that tend to $0$,  
we may also write $u_{j, x_0} = r_j^{-1} u(r_j x + x_0)$.  
Furthermore, when no confusion is possible, we may even drop the dependence of $u_j$ on $x_0$. 
We shall use the quantity $W(u,x_0,r)$ defined in \eqref{eqn:mono2} and associated to the 
monotonicity formula of Proposition \ref{prop:monotonicitytwo}, i.e., 
\begin{equation}\label{eqn:mono2a}
W(u, x_0, r) = \frac{1}{r^n} \int_{B(x_0,r)} \Big\{ |\nabla u|^2 
+ q^2_+(x_0)\1_{\{u > 0\}} + q^2_-(x_0)\1_{\{u < 0\}} \Big\}
- \frac{1}{r^{n+1}} \int_{\partial B(x_0,r)} u^2 d\sigma.
\end{equation} 
With our assumption \eqref{a6.1}, an application of the almost monotonicity 
Proposition~\ref{prop:monotonicitytwo} to a decreasing sequence of radii yields the existence of the limit
\begin{equation}\label{a6.4}
W(u, x_0, 0) = \lim_{r \to 0} W(u, x_0, r).
\end{equation}
Also, we immediately deduce from \eqref{eqn:mono2a} 
and the change of variables formula that
\begin{equation}\label{wunderblowups}
W(u, x_0, t r) = W_{x_0}(u_{r, x_0}, t),
\end{equation}
where
\begin{equation}\label{a6.6}
W_{x_0}(v, t) = \frac{1}{t^n} \int_{B(0,t)}
\Big\{ |\nabla v|^2 
+ q^2_+(x_0)\1_{\{v > 0\}} + q^2_-(x_0)\1_{\{v < 0\}} \Big\}
- \frac{1}{t^{n+1}} \int_{\partial B(0,t)} v^2 d\sigma
\end{equation}
is the analogue of $W$ at the origin, but with constant functions $q_\pm \equiv q^2_\pm(x_0)$.

We wish to take limits of the functions $u_{j, x_0} = r_j^{-1} u(r_jx + x_0)$; 
the existence of sufficiently many of blow-up limits is be given by the following lemma.

\begin{lemma}\label{blowupsexistandarehomogenous1}
Let $u$ be an almost-minimizer for $J$ or $J^+$, and assume that \eqref{a6.1} holds
for some $c_0 > 0$. For each $x_0 \in \Gamma^+(u)$ and every sequence $\{ r_j \}$ of positive numbers
such that $\lim_{j \to +\infty} r_j = 0$, we can find a subsequence $\{ r_{j_k} \}$, 
such that the $u_{r_{j_k},x_0}$ converge to a limit $u_\infty$, uniformly on
compact subsets of $\R^n$.
\end{lemma}

This is easy, because the $u_{r_{j},x_0}$ are uniformly Lipschitz in each ball; see
the remark above Theorem 9.2 in \cite{DT}. 
We shall call a \underbar{blow-up limit} of $u$ at $x_0$ any limit $u_\infty$ of a sequence 
$\{ u_{r_{j},x_0} \}$ that converges (as above). The following lemma
gives a little more information on the convergence and the blow-up limits.

\begin{lemma}\label{blowupsexistandarehomogenous2}
Let $x_0 \in \Gamma^+(u)$ and  $\{ r_j \}$ be as in Lemma \ref{blowupsexistandarehomogenous1},
and assume that the $u_{r_{j},x_0}$ converge (uniformly on compact subsets of $\R^n$) to a limit
$u_\infty$. Then $u_\infty$ is a global minimizer for $J^\infty$ for $J^{\infty, +}$, 
as defined by \eqref{eqn2.1} and \eqref{eqn2.2a} with the constants $\lambda_\pm = q_\pm(x_0)$.
In addition, 
\begin{equation}\label{a6.7}
\nabla u_\infty \ \text{ is the limit in $L^2_{\mathrm{loc}}(\R^n)$ of the } \nabla u_{r_{j},x_0},
\end{equation}
\begin{equation}\label{a6.8}
W_{x_0}(u_\infty, r) = W(u, x_0, 0) := \lim_{\rho \to 0} W(u, x_0, \rho) \ \text{ for } r > 0,
\end{equation}
and $u_\infty$ is (positively) homogeneous of degree $1$, i.e., $u_\infty(\lambda x) = \lambda u_\infty(x)$
for $x\in \R^n$ and $\lambda > 0$.
\end{lemma}

\begin{proof}
The fact that $u_\infty$ is a global minimizer and the convergence of $\nabla u_{r_{j},x_0}$
in $L^2_{\mathrm{loc}}(\R^n)$ are a part of Theorem 9.2 in \cite{DT}, which itself
is a direct application of Theorem 9.1 in \cite{DT}, applied to $u_{r_{j},x_0}$ which is
almost minimal with the functions $q_{j,\pm}(z) = q_\pm(x_0 + r_j z)$.
Now in Theorem 9.1 in \cite{DT}, (9.14) says that for each ball $B(x,r)$ and each choice of sign $\pm$, 
\begin{equation}\label{a6.9}
\int_{B(x,r)}  q_\pm(x_0) \1_{\{\pm u_{\infty} > 0 \}}(z) dz = \lim_{j\rightarrow \infty}
\int_{B(x,r)} q_{j,\pm}(z) \1_{\{ \pm u_{j,x_0} > 0 \}} (z) dz.
\end{equation}
Since by \eqref{a6.1} the $q_{j,\pm}$ converge uniformly in $B(x,r)$ to $q_\pm(x_0)$, and also $q_\pm \geq c_0 > 0$,
we may drop the $q$-functions and get that
\begin{equation}\label{a6.10}
\int_{B(x,r)}  \1_{\{\pm u_{\infty} > 0 \}}(z) dz = \lim_{j\rightarrow \infty}
\int_{B(x,r)} \1_{\{ \pm u_{j,x_0} > 0 \}} (z) dz
\end{equation}
(and in fact the proof of (9.14) in \cite{DT} essentially goes through this).
We may now use this and \eqref{a6.7} to take a limit in \eqref{a6.6} and get that
\begin{equation}\label{a6.11}
W_{x_0}(u_\infty, r) = \lim_{j\rightarrow \infty} W_{x_0}(u_{j,x_0}, r)
= \lim_{j\rightarrow \infty} W(u, x_0, r_j r) = W(u, x_0, 0),
\end{equation}
by \eqref{wunderblowups} and \eqref{a6.4}. Thus \eqref{a6.8} holds, and
$W_{x_0}(u_\infty, \cdot)$ is constant. Then by Lemma \ref{rem:wconstantimplieshomogenous},
$u_\infty$ is $1$-homogeneous, and Lemma \ref{blowupsexistandarehomogenous2} follows.
\end{proof}

For the rest of this section we keep the assumption \eqref{a6.1} but restrict to the case when $u$ 
is an almost-minimizer for $J^+$. Hence we drop $q_-$ and the definition of $W$ is a little simpler.

\begin{defn}\label{defn:flatpoints}
Set $\Gamma^+(u) = \Omega \cap \d U = \Omega \cap \partial \{u > 0\}$ as above, and
denote by $\omega_n$ the volume of the unit ball in $\mathbb R^n$.
The points of the set 
\begin{equation}\label{a6.12}
\mathcal R = \big\{ x_0 \in \Gamma^+(u) \, ; W(u,x_0,0) = q_+^2(x_0)\frac{\omega_n}{2} \, \big\}
\end{equation}
will be called \underbar{regular points} of $\Gamma^+(u)$ (for the one-phase problem).
\end{defn}

The next proposition will give a characterization of these points $x_0$ in terms of the blow-up limits of 
$u$ at $x_0$.
Notice that by \eqref{a6.8}, $W(u,x_0,0)$ is the constant value of the Weiss functional 
$W_{x_0}(u_\infty)$ for every blow-up limit $u_\infty$ of $u$ at $x_0$. In addition, since $u_\infty$
is homogeneous, \eqref{eqn:wononehomog} says that
\begin{equation}\label{a6.13}
W(u,x_0,0) = W_{x_0}(u_\infty,1) = q_+^2(x_0) |B(0,1) \cap \{u_\infty > 0\}|.
\end{equation}

As we shall see soon, $q_+^2(x_0)\frac{\omega_n}{2}$ is the smallest possible value
of $W(u,x_0,0)$, and is attained only when $u_\infty$ is a half-plane solution.
We say that $v$ is a \underbar{half-plane solution} (associated to $q_+(x_0)$)
when there  is a unit vector $\nu \in \mathbb S^{n-1}$ such that 
\begin{equation}\label{a6.14}
v(x) = q_+(x_0) \langle x, \nu \rangle_+  := q_+(x_0) \max(0, \langle x, \nu \rangle) 
\ \ \text{ for } x\in \R^n.
\end{equation} 
The name of solution and the choice of the ``slope" $q_+(x_0)$ are correct, 
because it is proved in \cite{AC}, Theorem 2.5, that 
$v(x) = a \langle x, \nu \rangle_+$  
is a a global minimizer of the functional $J^{\infty,+}$ associated to the constant 
coefficient $\lambda_+ = q_+(x_0)$ if and only if $a = q_+(x_0)$.

Analyzing the eigenvalues of the spherical Laplacian gives us several equivalent definitions of regular points
for the one phase problem.

\begin{proposition}\label{flatpointsregularpoints}
Assume \eqref{a6.1}, and let $u$ be an almost-minimizer for $J^+$. Then
\begin{equation}\label{a6.15}
W(u,x_0,0) \geq q_+^2(x_0)\frac{\omega_n}{2}
\ \text{ for every } x_0 \in \Gamma^+(u).
\end{equation}
In addition, for $x_0 \in \Gamma^+(u)$, the following are equivalent:
\begin{enumerate}[(1)]
\item $x_0 \in \mathcal R$;
\item Every blow-up limit of $u$ at $x_0$ is a half-plane solution;
\item Some blow-up limit of $u$ at $x_0$ is a half-plane solution.
\end{enumerate}
\end{proposition}

\begin{proof}
Let $u$ and $x_0 \in \Gamma^+(u)$ be given, and let $u_\infty$ be a blow-up
limit of $u$ at $x_0$, associated as above to a sequence $\{ r_j \}$.
Then $u_\infty$ is homogeneous of degree $1$ and harmonic on 
$U_\infty = \big\{ x\in \R^n \, ; \, u_\infty(x) > 0\big\}$. Let $g$ denote the restriction of
$u_\infty$ to the unit sphere $\mathbb S^{n-1}$; then 
\begin{equation}\label{a6.16}
\Delta_{\mathbb S^{n-1}}g(\theta) + (n-1)g(\theta) =0,
\ \text{ for } \theta\in \{g >0\} \cap \mathbb S^{n-1},
\end{equation}
where $\Delta_{\mathbb S^{n-1}}$ is the Laplace-Beltrami operator on the sphere. 
In other words, $g$ is an eigenfunction for $-\Delta_{\mathbb S^{n-1}}$ on $\{g >0\}$,
with the eigenvalue $n-1$.

For every open subset $\Sigma \subset \mathbb S^{n-1}$  
denote by $\lambda(\Sigma)$ the smallest eigenvalue of $-\Delta_{\mathbb S^{n-1}}$ 
on $\Sigma$ and by $V(\Sigma)$ its $(n-1)$-volume.
Sperner \cite{Sp} showed that $\lambda(\Sigma) \geq \lambda(S_{V(\Sigma)})$,
where $S_{V}$ denotes the spherical cap with the $(n-1)$-volume $V$. 
Later, Beckner, Kenig and Pipher \cite{BKP} (see also \cite{CK}, Remark 2.4.4 and Theorem 2.4.5) 
showed that this inequality is strict unless $\Sigma$ is a spherical cap. 

Finally, since $\lambda(S_{V})$ can also be expressed in terms of the optimal constant for 
a Poincar\'e inequality on $S_{V}$, it is clear that $\lambda(S_{V})$ is a decreasing function of $V$,
and a quick computation shows that for the half sphere, $\lambda(S_{\alpha_{n-1}/2}) = (n-1)$, where 
$\alpha_{n-1}$ is the $(n-1)$-volume of $\mathbb S^{n-1}$.
It follows that, if $V(\Sigma) \leq \alpha_{n-1}/2$, then 
\begin{equation}\label{a6.17}
\lambda(\Sigma) \geq \lambda(S_{V(\Sigma)}) \geq \lambda(S_{\alpha_{n-1}/2}) = n-1,
\end{equation}
with equality if and only if $\Sigma$ is a hemisphere. 

Return to $u_\infty$ and $g$. Since $(n-1)$ is an eigenvalue of $\Delta_{\mathbb S^{n-1}}$
on $\Sigma = \{ g > 0 \}$, \eqref{a6.17} says that $V(\Sigma) \geq \alpha_{n-1}/2$,
and that $\Sigma$ is a half sphere if $V(\Sigma) = \alpha_{n-1}/2$.
Since $W(u,x_0,0) =  q_+^2(x_0) |B(0,1) \cap \{u_\infty > 0\}|$ by\eqref{a6.13}
and $u_\infty$ is  the homogeneous extension of $g$,
we get that $W(u,x_0,0) \geq  q_+^2(x_0) \frac{\omega_n}{2}$,
and $\{ u_\infty > 0 \}$ is a half space if $W(u,x_0,0) =  q_+^2(x_0) \frac{\omega_n}{2}$.

In particular, \eqref{a6.15} holds, and we are ready to prove the equivalence of our three conditions.
First assume that $x_0 \in \mathcal R$. Then for any blow-up limit, $\{ u_\infty > 0 \}$ is a half space,
$g$ is a solution of \eqref{a6.16} for a half sphere, and it is known that in this case $g$ is affine and 
$u_\infty$ is a multiple of a half-plane solution. Since $u_\infty$ is a global minimizer, it is actually equal
to a half-plane solution. Thus (1) implies (2), which obviously implies (3). 
Finally, if some blow-up limit of $u$ at $x_0$ is a half-plane solution, then 
$W(u,x_0,0) = q_+^2(x_0) \frac{\omega_n}{2}$ by \eqref{a6.13}, hence $x_0 \in \mathcal R$.
The proposition follows.

\end{proof}

Recall from Theorem \ref{t4.2} that under the current assumptions, $\Gamma^+(u)$ is locally 
Ahlfors-regular and uniformly rectifiable, and the proof also gives a local version of Condition $B$. Thus
$\H^{n-1}$-almost every $x_0 \in \Gamma^+(u)$ (in fact, every point $x_0 \in \Gamma^+(u)$ 
where $\Gamma^+(u)$ has a tangent plane) lies in the reduced boundary $\partial^*\{u > 0\}$.
The next corollary shows that these points lie in $\mathcal R$. 

\begin{corollary}\label{cor:regularpoints}
Assume \eqref{a6.1} and let $u$ be an almost-minimizer for  
$J^+$ in $\Omega$. Then the reduced boundary $\Omega \cap \partial^*\{u > 0\}$ is contained
in $\mathcal R$. 
\end{corollary}

\begin{proof}
Let $x_0 \in \Omega \cap \partial^*\{u > 0\}$ be given, and let $u_\infty$ be a blow-up
limit of $u$ at $x_0$, associated as above to the sequence $\{ r_j \}$.
Set $u_j = u_{r_j, x_0}$; thus the $u_j$ tend to $u_\infty$ as in 
Lemma \ref{blowupsexistandarehomogenous2}.

By definition of $\partial^*\{u > 0\}$, the functions $\1_{\{u_j > 0\}}$
converge in $L^1_{\mathrm{loc}}(\R^n)$ to the characteristic function of a half plane.
Then \eqref{a6.10} (applied with $x=0$, $r=1$, and the sign $+$) yields 
$\int_{B(0,1)}  \1_{\{u_{\infty} > 0 \}}(z) dz = \frac{\omega_n}{2}$
and hence, by \eqref{a6.13}, $W(u,x_0,0) = q_+^2(x_0)\frac{\omega_n}{2}$ and $x_0 \in \mathcal R$.

\end{proof}

\begin{corollary}\label{low-dim}
Let $u$ be an almost-minimizer for $J^+$ in $\Omega\subset \R^n$,
and assume that \eqref{a6.1} holds and $2 \leq n \leq 4$. Then $\mathcal R = \Gamma^+(u)$.
\end{corollary}

\begin{proof}
Results of \cite{AC} when $n=2$ (Corollary 6.7), the Theorem in \cite{CJK} when $n=3$  
and \cite{JS} when $n=4$ (Theorem 1.1) guarantee that every 
one-homogeneous global minimizer for $J^{\infty, +}$ is a 
one-plane solution. If $u_\infty$ is any blow-up limit of $u$ at $x_0 \in \Gamma^+(u)$,
Lemma \ref{blowupsexistandarehomogenous2} says that $u_\infty$ is such a homogeneous global minimizer,
and hence is a one-plane solution. The corollary now follows from Proposition \ref{flatpointsregularpoints}.

\end{proof}

\ms
We know from \eqref{a6.15} that $W(u,x_0,0)$ is smallest at regular points. 
We are interested in quantitative versions of this, that will often be obtained with limiting arguments.
First, there is a gap between $q_+^2(x_0)\frac{\omega_n}{2}$ and the next authorized value.

\begin{lemma}\label{6gap}
There is a positive constant $\varepsilon(n) > 0$ such that if $v$ is a global minimizer for $J^+$,
as in Definition \ref{d2.2} with the constant $\lambda_+ > 0$, which is also homogeneous of degree $1$ and
such that
\begin{equation}\label{a6.gap}
\lambda_+ \big| B(0,1) \cap \{ v > 0 \}\big| \leq (1+\varepsilon(n)) \lambda_+^2\frac{\omega_n}{2},
\end{equation}
then $v$ is a half-plane solution.
\end{lemma}

Of course we replace $q_+(x_0)$ by $\lambda_+$ in the definition of a half-plane solution for $J^+$.

\begin{proof}
This result is not trivial at all, but it will be a rather simple consequence of Theorem~8.1 in \cite{AC}.
It is easy to see that $v$ is a global minimizer for $J^+$ with the constant $\lambda_+ > 0$
if and only if $v/\lambda_+$ a global minimizer for $J^+$ with the constant $1$. Thus we may
restrict to $\lambda_+=1$.

Assume, in order to obtain a contradiction, that for every $k \geq 0$ there exists a one-homogenous 
global minimizer $v_k$ for $J^+$ with $\lambda_+=1$, such that 
\begin{equation}\label{a6gap1}
\big| B(0,1) \cap \{ v_k > 0 \}\big| \leq (1+2^{-k}) \frac{\omega_n}{2}
\end{equation}
but which is not a half-plane solution. 
By Theorem 5.3 in \cite{ACF}, the functions $v_k$ are uniformly Lipschitz on $B(0,1)$
(or equivalently, since they are homogeneous, on any ball $B(0,R)$), and $v_k(0) = 0$ for all $k$, 
so we may extract a subsequence that converges uniformly on compact subsets of $\R^n$ to some limit
$v$. Then we can apply Theorem 9.1 in \cite{DT}, in the simpler situation where all the functions $q_+$
are identically equal to $1$. We get that $v$ is also a global minimizer for $J^+$,
and that (after extraction) the $\nabla v_k$ converge to $\nabla v$ in $L^2_{\mathrm{loc}}(\R^n)$.
We may also use (9.14) in \cite{DT}, as we did for \eqref{a6.9} and \eqref{a6.10},
to get that 
\begin{equation}\label{a6gap1a}
\big| B(0,1) \cap \{ v > 0 \}\big| = \int_{B(0,1)} \1_{\{ v > 0 \}}(z) dz
= \lim_{k\rightarrow \infty} \big| B(0,1) \cap \{ v_k > 0 \}\big| \leq  \frac{\omega_n}{2}.
\end{equation}
Then by the proof of \eqref{a6.15}, $v$ is a half-plane solution. That is, there is a unit vector $\nu$
such that $v(x) = \langle x, \nu \rangle_+$ for $x\in \R^n$. Without loss of generality,
we may assume that $\nu$ is the last coordinate vector and $v(x) = (x_n)_+$.

At this point we want to use the proximity to $v$ to show that for $k$ large, the free boundary
$\Gamma^+(v_k)$ is smooth at the origin, and this is where we apply Theorem~8.1 in \cite{AC}.
Thus we need to check, with the notation of \cite{AC}, that $v_k \in F(\sigma,1,\infty)$ in $B(0,1)$, say. 
Here the size of the ball does not matter, because $v_k$ is a minimizer (and is homogeneous anyway), 
and $\sigma$ is a small constant that comes from the theorem. 

Returning to Definition 7.1 in \cite{AC}, we see that in order to prove that
$v_k \in F(\sigma,\sigma_-,\infty)$ (in $B(0,1)$ and in the direction $\nu$),
we need to prove that $v_k$ is a weak solution (with $Q = 1$ here),  $0 \in \Gamma^+(v_k)$, 
\begin{equation}\label{a6gap2}
v_k(x) = 0 \ \text{ for $x\in B(0,1)$ such that } x_n \leq -\sigma
\end{equation}
(compared with \cite{AC}, we look in the other direction and $x_n$ is replaced with $-x_n$),  
\begin{equation}\label{a6gap3}
v_k(x) \geq x_n-\sigma_- \ \text{ for $x\in B(0,1)$ such that } x_n \geq \sigma_- \, ,
\end{equation}
and also $v_k$ is Lipschitz and bounded in $B(0,1)$.
It would not be hard to prove \eqref{a6gap3} with any $\sigma_- > 0$, because $\{v_k \}$
converges to $v$ uniformly in $B(0,1)$, but here $\sigma_- = 1$ and we do not even need to
do this. We know most of the other properties, and are only left with \eqref{a6gap2} to check.

So we let $x\in B(0,1)$ be such that $x_n \leq -\sigma$, assume that $v_k(x) > 0$, and 
prove that this leads to a contradiction if $k$ is large enough. Recall that $v_k$ is 
Lipschitz in $B(0,2)$, with a Lipschitz bound that depends only on $n$, 
so Theorem 10.2 in \cite{DT} (about the nondegeneracy of $v_k$ near the free boundary)
says that there is a constant $\tau > 0$, that depends only on $n$, such that 
\begin{equation}\label{a6gap4}
v_k(z) \geq \tau \dist(z,\Gamma^+(v_k)) \ \text{ for } z\in B(0,3/2) \cap \{ v_k > 0 \}.
\end{equation}
In particular, since $v_k(x) = |v_k(x)- v(x)| \leq ||v_k- v||_{L^\infty(B(0,2))}$
which tends to $0$, we see that if $k$ is large enough, we can find $y\in \Gamma^+(v_k)$ such that
$|x-y| < \sigma/2$. Then by the NTA property, we can find a corkscrew point
$z\in B(y,\sigma/2) \cap \{ v_k > 0 \}$ such that $\dist(z,\Gamma^+(v_k)) \geq C^{-1} \sigma$.
See Theorem \ref{t2.3} and the first item of Definition \ref{d2.3}. 
Then $v_k(z) \geq C^{-1} \sigma \tau$ by \eqref{a6gap4}. 
But $|z-x| < \sigma$ and $x_n \leq - \sigma$, so $z_n \leq 0$ and 
$v(z) = 0$. Our last estimate contradicts the fact that $||v_k- v||_{L^\infty(B(0,2))}$ tends to $0$,
and this completes our proof of \eqref{a6gap2}.

So we may apply Theorem~8.1 in \cite{AC}. We get that for $k$ large, $\Gamma^+(v_k)$
is smooth at the origin. Since $v_k$ is homogeneous, $\Gamma^+(v_k)$ is a hyperplane,
and hence $\big| B(0,1) \cap \{ v_k > 0 \}\big| =  \frac{\omega_n}{2}$.
This forces $v_k$ to be a half-plane solution, as in the proof of \eqref{a6.15}.
This contradiction with the definition of $v_k$ completes our proof of Lemma \ref{6gap}.

\end{proof}

Because of Lemma \ref{6gap}, we can also say that
\begin{equation}\label{a6.24}
\mathcal R = \big\{ x_0 \in \Gamma^+(u) \, ; 
W(u,x_0,0) \leq  (1+\varepsilon(n)) \, q_+^2(x_0)\frac{\omega_n}{2} \, \big\}.
\end{equation}
Indeed, one inclusion is obvious, and for the other one let $x_0 \in \Gamma^+(u)$ be such that 
$W(u,x_0,0) \leq  (1+\varepsilon(n)) \, q_+^2(x_0)\frac{\omega_n}{2}$ and let 
$u_\infty$ be any blow-up limit of $u$ at $x_0$.
By Lemma \ref{blowupsexistandarehomogenous2}, $u_\infty$ is a homogeneous global
minimizer with $\lambda_+ = q_+^2(x_0)$, and since
\begin{equation}\label{a6.25}
\big| B(0,1) \cap \{ u_\infty > 0 \}\big| = 
W_{x_0}(u_\infty,r) = W(u,x_0,0) \leq (1+\varepsilon(n)) \, q_+^2(x_0)\frac{\omega_n}{2}
\end{equation}
by \eqref{a6.13} and \eqref{a6.8}, Lemma \ref{6gap} says that $u_\infty$ is a 
half-plane solution, and hence $x_0 \in \mathcal R$ by Proposition \ref{flatpointsregularpoints}.
Here is a simple consequence of \eqref{a6.24}. 

\ms
\begin{corollary}\label{flatpointsareopen}
$\mathcal R$ is open in $\Gamma^+(u)$. 
\end{corollary}

\begin{proof}
Notice that since each $W(u,x,r)$ is a continuous function of $x$, and by \eqref{a5.6},
for each $x_0 \in \Gamma^+(u)$ there exist constants $C > 0$ and $\alpha > 0$ such that 
for $x \in \Gamma^+(u)$ near $x_0$, the sequence $\{ W(u,x,2^{-k})+C2^{-k\alpha} \}$
is decreasing. Then the almost monotone limit $W(u,x,0)$ is upper semi-continuous. That is, 
$\big\{ x\in \Gamma^+(u) \, ; \, W(u,x,0) < \lambda \big\}$ is open.

If $x_0 \in \mathcal R$, then $W(u,x_0,0) = q_+^2(x_0)\frac{\omega_n}{2}$ and, by semicontinuity, 
$W(u,x,0) < (1+\varepsilon(n)) q_+^2(x)\frac{\omega_n}{2}$ for $x\in \Gamma^+(u)$
close enough to $x_0$, as needed.
\end{proof}

The next proposition is another quantitative version of Proposition \ref{flatpointsregularpoints}.

\begin{proposition}\label{Missmallestatflatpoints}
Assume \eqref{a6.1}, let $u$ be an almost-minimizer for 
$J^+$ in $\Omega$, and let $K \subset \subset \Omega$ be compact. For every 
$\sigma > 0$ there exist $\varepsilon_\sigma > 0$ and 
$\rho_\sigma > 0$ (which may depend on $K$, $q_+$, and $u$) 
such that if $x_0 \in K \cap \Gamma^+(u)$ and $\rho \in (0,\rho_\sigma)$ are such that
\begin{equation}\label{a6.26}
W(u, x_0, 2\rho) \leq (1+ \varepsilon_\sigma) \, q_+(x_0)^2 \, \frac{\omega_n}{2},
\end{equation}
then $x_0 \in \mathcal R$ and we can find 
$\nu_\rho \in \mathbb S^{n-1}$ such that 
\begin{equation}\label{eqn:flatonbothsides1}
|u(x+x_0) - q_+(x_0)\left\langle x, \nu_\rho\right\rangle_+| \leq \sigma \rho 
\ \ \text{ for } x\in B(0,\rho) 
\end{equation}
and 
\begin{equation}\label{eqn:flatonbothsides2}
u(x+x_0) = 0
\ \ \text{ for $x\in B(0,\rho)$ such that } \left\langle x, \nu_\rho\right\rangle \leq -\sigma\rho.
\end{equation}
\end{proposition}

\ms
As per usual, we shall not try to see that $\varepsilon_\sigma > 0$ and $\rho_\sigma > 0$ depends only on $n$,
$\dist(K,\d U)$, $q_+$, and the almost minimality constants for $u$, but this would not be very hard.
We added the conclusion that $x_0 \in \mathcal R$ to comfort the reader, but what really matters is the
uniform approximation in \eqref{eqn:flatonbothsides1} and â\eqref{eqn:flatonbothsides2}.
In fact, if $x_0$ and $\rho \in (0,\rho_\sigma)$ are as in the statement, and if $\rho_\varepsilon$ is chosen small enough, then by 
Proposition \ref{prop:monotonicitytwo} (the almost monotonicity of $W(u,x_0,\cdot)$) we also have that 
$W(u, x_0, t) \leq (1+ 2\varepsilon_\sigma) \, q_+(x_0)^2 \, \frac{\omega_n}{2}$
for $0 \leq t \leq 2\rho$. We shall take $\varepsilon_\sigma < \varepsilon(n)/2$, so
$x_0 \in \mathcal R$ by \eqref{a6.24}. But also, at the price of making $\varepsilon_\sigma$ twice
smaller, we see that the approximation conclusion holds for $0 < \rho' < \rho$, although with possibly
different directions $\nu_{\rho'}$.

\begin{proof}
Let $\sigma > 0$ be given and assume, in order to obtain a contradiction, that there are points 
$x_i \in K\cap \Gamma^+(u)$ and scales $\{\rho_i\}_{i=1}^\infty$, 
with $\rho_i \downarrow 0$, such that 
\begin{equation}\label{a6.29}
W(u, x_j, 2\rho) \leq (1+2^{-j}) \, q_+(x_j)^2 \, \frac{\omega_n}{2},
\end{equation}
but the conclusion fails. Since we proved above that $x_j \in \mathcal R$, this means
that we cannot find $\nu \in \mathbb S^{n-1}$ such that
\eqref{eqn:flatonbothsides1} and \eqref{eqn:flatonbothsides2} hold (with $x_0 = x_j$ and $\rho = \rho_j$).
Set $u_i = u_{\rho_i,x_i}\,$, i.e., $u_i(x) = \rho_i^{-1} u(x_i + \rho_i x)$.
We may replace $\{ u_i \}$ by a subsequence for which $x_i$ tends to a limit 
$x_0 \in K\cap \Gamma^+(u)$. Also, $u$ is Lipschitz near $K$, and since the $x_i$
stay in $K$ and the $\rho_i$ tend to $0$, it is easy to extract a new subsequence, 
which we shall still denote by $\{ u_i \}$, which converges uniformly on compact subsets of $\R^n$
to a limit $u_\infty$.

We claim that we may now proceed as in Lemma \ref{blowupsexistandarehomogenous2} to control $u_\infty$. 
There is a small difference with the situation of Lemma \ref{blowupsexistandarehomogenous2}, 
because here $x_j$ is not fixed and so we cannot apply Theorem 9.2 in \cite{DT} directly.
Instead we apply Theorem 9.1 in \cite{DT} to the sequence $\{ u_j \}$
(just as Theorem 9.2 was deduced from Theorem 9.1 in \cite{DT}). The corresponding
weights $x \to q_+(x_i+\rho_j x)$ converge to $q_+(x_0)$ uniformly on compact sets of $\R^n$,
because $q_+$ is H\"older-continuous and $x_i \to x_0$, and the $u_j$ are locally Lipschitz
with estimates that do not depend on $j$.
This is enough to apply Theorem 9.1 in \cite{DT}. We get that 
$u_\infty$ is a global minimizer for $J^{\infty,+}$, the functional of Section \ref{global} associated
to the constant weight $\lambda_+ = q_+(x_0)$, and also that $\nabla u_\infty$ is the limit of 
$\nabla u_j$ in $L^2_{\mathrm{loc}}(\R^n)$. In addition (9.14) in \cite{DT} implies, as in \eqref{a6.10}, 
that for $r>0$,
\begin{equation}\label{a6.30}
\int_{B(0,r)}  \1_{\{ u_{\infty} > 0 \}} = \lim_{j\rightarrow \infty}
\int_{B(0,r)} \1_{\{ u_j > 0 \}}.
\end{equation}
We multiply by $q_+^2(x_0)$ and add energy integrals that converge and get that
\begin{equation}\label{a6.31}
W_{x_0}(u_\infty, r) = \lim_{j\rightarrow \infty} W_{x_0}(u_j, r).
\end{equation}
But
\begin{equation}\label{a6.32}
\begin{aligned}
W_{x_0}(u_j, r) - W(u,x_j,\rho_j r)
&= W_{x_0}(u_j, r) - W_{x_j}(u_j, r) 
\\
&= r^{-n} [q_+^2(x_0)-q_+^2(x_j)] \big| B(0,r) \cap \{ u_j > 0 \} \big|
\end{aligned}
\end{equation}
by \eqref{wunderblowups} and the definition \eqref{a6.6}. Since the right-hand side 
tends to $0$ because $q_+(x_i+\rho_j x)$ converges to $q_+(x_0)$ uniformly on $B(0,r)$,
we see that 
\begin{equation}\label{a6.33}
W_{x_0}(u_\infty, r) = \lim_{j\rightarrow \infty} W(u,x_j,\rho_j r).
\end{equation}
We use this with $r=2$ and deduce from \eqref{a6.29} that 
\begin{equation}\label{a6.34}
W_{x_0}(u_\infty, 2) \leq  q_+(x_0)^2 \, \frac{\omega_n}{2}
\end{equation}
because $q_+(x_j)$ tends to $q_+(x_0)$.
Since $u_\infty$ is a global minimizer, $W_{x_0}(u_\infty, r)$ is a nondecreasing function of $r$
and $W_{x_0}(u_\infty, r) \leq {q_+(x_0)^2}\frac{\omega_n}{2}$ for $0< r \leq 2$.

By Proposition \ref{flatpointsregularpoints}, applied to $u_\infty$ instead of $u$,
$W_{x_0}(u_\infty, 0) \geq {q_+(x_0)^2}\frac{\omega_n}{2}$, hence in fact 
$W_{x_0}(u_\infty, r) = {q_+(x_0)^2}\frac{\omega_n}{2}$ for $0< r \leq 2$.
By the proof of Proposition \ref{flatpointsregularpoints}, $u_\infty$ is homogeneous of degree $1$
on $B(0,2)$, and then (by the eigenvalue argument) coincides with a half-plane solution on that ball.

Thus we proved that the $u_j$ converge uniformly on $B(0,2)$ 
to a half-plane solution, which we write $v(x) = q_+(x_0)\left\langle x, \nu_\rho\right\rangle_+$
for some unit vector $\nu$ (see \eqref{a6.14}). We just need to show that for this $\nu$,
\eqref{eqn:flatonbothsides1} and \eqref{eqn:flatonbothsides2} hold for $j$ large 
(with $\nu_\rho = \nu$ and $x_0$ replaced by $x_j$), and this will prove the proposition by contradiction.

Now \eqref{eqn:flatonbothsides1} holds precisely because $\{ u_j \}$ converges to $v$
uniformly and $q_+(x_j)$ tends to $q_+(x_0)$, so we may concentrate on \eqref{eqn:flatonbothsides2}.
The proof will be quite similar to what we did for \eqref{a6gap2}, but we give the argument because 
the reader may worry that we used extra properties of global minimizers.

It is enough to let $x\in B(0,\rho_j)$ be such that $\left\langle x, \nu\right\rangle \leq -\sigma\rho_j$,
suppose that $u(x+x_j) > 0$, and get a contradiction.
Set $y = \rho_j^{-1} x$; thus $y\in B(0,1)$, $\left\langle y, \nu\right\rangle \leq -\sigma$, and 
$u_j(y) > 0$. Recall  that $u$ is Lipschitz in a neighborhood of $K$, and hence the $u_j$ are 
Lipschitz in $B(0,2)$, with a Lipschitz bound $M$ that does not depend on $j$. By the nondegeneracy
of $(u_{j})_+$ (see Theorem~10.2 in \cite{DT}), there is a constant $\tau > 0$, that depends only on 
$M$, $n$, $||q_+||_\infty$, and $c_0$, such that 
\begin{equation}\label{a6.22}
u_j(z) \geq \tau \dist(z,\Gamma^+(u_j)) \ \text{ for } z\in B(0,3/2) \cap \{ u_j > 0 \}.
\end{equation}
In particular, since $u_j(y) = |u_j(y)- v(y)| \leq ||u_j- v||_{L^\infty(B(0,2))}$
which tends to $0$, we see that if $j$ is large enough, we can find $w\in \Gamma^+(u_j)$ such that
$|y-w| < \sigma/2$. Then by the NTA property, we can find a corkscrew point
$z\in B(z,\sigma/2) \cap \{ u_j > 0 \}$ such that $\dist(z,\Gamma^+(u_j)) \geq C^{-1} \sigma$.
See Theorem \ref{t2.3} and the first item of Definition \ref{d2.3}. Then $u_j(z) \geq C^{-1} \sigma \tau$
by \eqref{a6.22}. But $|z-y| < \sigma$, so $v(y) = 0$ and the last estimate contradicts
the fact that $||u_j- v||_{L^\infty(B(0,2))}$ tends to $0$.
This contradiction completes our proof of \eqref{eqn:flatonbothsides2} and 
Proposition \ref{Missmallestatflatpoints}.
\end{proof}

{\it A priori}, the blow-up limit $u_\infty$ may vary with the sequence $\rho_j \downarrow 0$
that we chose to define it. However, if we are given extra geometric information about the point 
$x_0 \in \Gamma^+(u)$, 
then we can prove that there is a unique blow-up limit. 
We start with the existence of a tangent exterior ball.

\begin{corollary}\label{zerosetball}
Assume \eqref{a6.1} and let $u$ be an almost-minimizer  for $J^+$ in $\Omega$. 
Assume that $x_0 \in \Gamma^+(u)$ 
is such that there exists an open ball $B$, with $B \subset \{u = 0\}$ and 
$x_0 \in \d B$. 
Then $x_0 \in \mathcal R$ and we can find 
$\nu \in \mathbb S^{n-1}$ such that for every $\sigma > 0$, there exists 
$\rho_{\sigma, x_0} > 0$ such that  
\begin{equation}\label{eq: uniqueblowupzeroball} 
\left|u(x) - q_+(x_0)\left\langle x-x_0,\nu\right\rangle_+\right| < \sigma r 
\ \ \text{ for $r < r_{\sigma, x_0}$ and } x\in B(x_0,r).
\end{equation}
\end{corollary}

\begin{proof}
Let $u_\infty$ be any blow-up limit of $u$ at $x_0$, and let $\{ r_j \}$ be the associated sequence, 
so that $r_j$ tends to $0$ and the $u_j(x) = \frac{u(r_j x + x_0)}{r_j}$ converge to $u_\infty$ 
uniformly on compact sets. Set $D_j = r_j^{-1}(B-x_0)$; by assumption $u_j = 0$ on $D_j$.
and since the $D_j$ converge to a half space $H$, we get that $u_\infty = 0$ on $H$.

By \eqref{a6.13}, $W(u,x_0,0) = q_+^2(x_0) |B(0,1) \cap \{u_\infty > 0\}| \leq 
q_+^2(x_0) |B(0,1) \sm H| \leq q_+^2(x_0) \, \frac{\omega_n}{2}$, so
Proposition \ref{flatpointsregularpoints} says that $x_0 \in \mathcal R$ and $u_\infty$ is a half-plane solution. 

Since $u_\infty = 0$ on $H$, there is no choice and $u_\infty(x) = q_+(x_0) \langle x, \nu \rangle_+$,
where $\nu$ is the unit vector that points directly away from the center of $B$ seen from $x_0$.
Thus all the blow-up limits of $u$ at $x_0$ are the same $u_\infty$, associated to $\nu$.
This implies (by the existence of convergent subsequences) that the functions 
$u_{r,x_0}$ of \eqref{eqn:blowups} actually converge to this $u_\infty$, uniformly on compact sets, and 
\eqref{eq: uniqueblowupzeroball} follows at once.
\end{proof}

Here is a variant of the previous corollary, but for points of the reduced boundary $\partial^* \{u > 0\}$.

\begin{corollary}\label{zerosetreduced}
Assume \eqref{a6.1} and let $u$ be an almost-minimizer for $J^+$ in $\Omega$. 
Assume that $x_0 \in \Omega \cap \partial^* \{u > 0\}$, 
and let $\nu = \nu(x_0)$ denote the associated unit normal, pointing in the direction of $\{u > 0\}$.
Then $x_0 \in \mathcal R$ and for every $\sigma > 0$ there exists $r_{\sigma, x_0} > 0$ such that 
\begin{equation}\label{eq:uniqueblowupzeroball} 
\left|u(x) - q_+(x_0)\left\langle x-x_0,\nu\right\rangle_+\right| < \sigma r 
\ \ \text{ for $r < r_{\sigma, x_0}$ and } x\in B(x_0,r).
\end{equation}
\end{corollary}

We already said in Corollary \ref{cor:regularpoints} that $x_0 \in \mathcal R$, but we will prove it again.
When we restrict \eqref{eq:uniqueblowupzeroball} to $x = x_0 + t \nu$, $t > 0$, we get the existence of
a normal derivative
\begin{equation}\label{normalderivative1} 
\frac{\partial^+ u}{\partial \nu}(x_0) := \lim_{t \to 0_+} t^{-1} u(x_0 + t \nu) = q_+(x_0)\nu.
\end{equation}
When we stay in $U = \big\{ x\in \Omega \, ; \, u(x) > 0 \big\}$, \eqref{eq:uniqueblowupzeroball}
gives an expansion
\begin{equation}\label{Ihavetotrythisoncetoseehowitfeelstogivelongnamestothingsbutifyouseethisyoumayreplacebygradient}
u(x) = \left\langle\nabla^+ u(x_0),x-x_0\right\rangle_+ + o(|x-x_0|),
\end{equation}
(where by the Landau convention, $o(|x-x_0|)/|x-x_0|$ tends to $0$ when $x$ tends to $x_0$,
and we may also have dropped the positive part) with 
\begin{equation}\label{normalderivative}
\nabla^+ u(x_0)= \frac{\partial^+ u}{\partial \nu}(x_0)\nu =q_+(x_0)\nu.
\end{equation}

\begin{proof}
Let $u_\infty$ be a blow-up limit of $u$ at $x_0$, associated as above to a sequence $\{ r_j \}$.
Set $u_j(x) = \frac{u(r_j x + x_0)}{r_j}$ as above. By definition of $\partial^* \{u > 0\}$,
the functions $\1_{\{ u_j = 0 \}}$ converge in $L^1_{loc}(\R^n)$ to $\1_H$, where 
$H$ is the half space pointing in the direction opposite to $\nu$. If $u_\infty(x) > 0$ for some
interior point $x$ of $H$, then by the uniform convergence of $u_j$ to $u_\infty$ there is 
a small ball $B$ centered at $x$ such that for $j$ large, $u_j(y) > 0$ for $y\in B$. This 
contradicts the local $L^1$ convergence, so $u_\infty(x) = 0$ on $H$, and we may conclude as in
Corollary \ref{zerosetball}.
\end{proof}

We now use Corollary \ref{zerosetreduced} to prove the existence of a normal derivative and gradient, 
at points of the reduced boundary, of the function $h_{ x_0,r}$ that was defined near \eqref{a3.4}.

\begin{corollary}\label{cor6.1A} 
Let  $\Omega\subset \R^n$ be a bounded, connected open set, 
and let $q_+ \in L^\infty(\Omega)$ be H\"older-continuous and such that $q_+\ge c_0>0$.  
For each $r_0>0$, we can find a radius $\rho_4>0$, that depends only on 
$n$, $c_0$, $\|q_+\|_{L^\infty}$, $\kappa$, $\alpha$ and $r_0$,  
and a constant $\beta \in(0,\alpha/16n)$,  
that depends only on $n$ and $\alpha$, with the following properties.

Let $u$ be an almost-minimizer for  $J^+$ in $\Omega$ (with the constants $\alpha$ and $\kappa$).  
If $0 < r <\rho_4$, $x_0\in \Gamma^+(u) = \Omega \cap \partial\{u>0\}$, 
$B( x_0, 6r_0)\subset \Omega$, and $z\in \partial^\ast\{u>0\} \cap 
B(x_0, 2r^{1+\alpha/17n})$,  
then the function $h_{ x_0,r}$ defined near \eqref{a3.4} satisfies
\begin{equation}\label{eqn6.1A} 
(1-5r^\beta) q_+(z)\le \frac{\partial^+ h_{x_0,r}}{\partial \nu} (z)
\le (1+5r^\beta)q_+(z)
\end{equation}
and
\begin{equation}\label{eqn6.1B}
\nabla^+ h_{x_0,r}(z)= \frac{\partial^+ h_{x_0,r}}{\partial \nu}(z)\nu (z),
\end{equation}
where the existence of $\frac{\partial^+ h_{x_0,r}}{\partial \nu} (z)$ and $\nabla^+ h_{x_0,r}(z)$,
as defined below Corollary \ref{zerosetreduced}, are part of the statement.
\end{corollary}

The proof starts with Corollary \ref{zerosetreduced}, which gives similar results for $u$,
and deduce the result from estimates on $h_{x_0,r}/u$ that we proved in earlier sections.
But let us compare with the slightly different function $h_{z,r}$ first, for which we shall be 
able to use Theorem \ref{t3.1} more directly.

Our H\"older assumption on $q_+$ is used to prove the existence of $\frac{\partial^+ u}{\partial \nu}(z)$
and $\nabla^+ u(z)$, but we do not need quantitative estimates for this. The other assumptions come from
Section \ref{harmonic-functions} and are used to connect $h_{x_0,r}$ to $u$ and prove \eqref{eqn6.1A}.

\begin{proof} 
Let $r_0 > 0$ be given, and let $\beta \in (0,\alpha/16n)$ and $\rho_3$ be as in Theorem \ref{t3.1}
(the assumptions are satisfied). Suppose that $\rho_4 < \rho_3$ (other similar constraints will be added soon), and let $x_0$ and $z$ be as in the statement. Then Theorem~\ref{t3.1} says that
\begin{equation}\label{a6.43}
(1-5r^{\beta})u(x)\le h_{z,r}(x)\le (1+5r^{\beta})u(x)
\end{equation}
for $x\in U = \big\{ x\in \Omega \, ; \, u(x)> 0\big\}$ such that 
$|z-x_0|+ |x-x_0| < 5 r^{1+\alpha/17n}$. Since $|z-x_0| \leq 2r^{1+\alpha/17n}$,
this works for $|x-x_0| < 3 r^{1+\alpha/17n}$ and in particular for $|x-z| < r^{1+\alpha/17n}$.

Since $z \in \partial^\ast\{u>0\} \cap \Omega$, $\d U$ (or equivalently $\Gamma^+(u)$)
has an approximate tangent plane $P$ at $x$, and since $\d U$ is locally Ahlfors-regular, 
$P$ is actually a true tangent plane. Let us assume, without loss of generality, that we have coordinates in 
$\R^n$ such that $z=0$, $P$ is given by the equation $x_n = 0$, and $U$ lies above $\d U$ near $z$.
Let $\nu = e_n$ denote the unit normal at $z$, pointing in the direction of $e$.
We first want a control on $h_{z,r}$ on nontangential sectors, so we define, for $\tau \in (0,1)$, a sector
\begin{equation}\label{a6.44}
\Gamma_\tau = \big\{ \theta \in \SS^{n-1} \, ; \, \theta_n \geq \tau\big\}
\end{equation}
(where $\theta_n = \langle \theta, \nu \rangle $)
and two functions
\begin{equation}\label{a6.45}
a_-(r,t) = \inf_{\theta \in \Gamma_\tau}  
(t \theta_n)^{-1 } h_{z,r}(z+t\theta)
\ \text{ and } \ 
a_+(r,t) = \sup_{\theta \in \Gamma_\tau}  
(t \theta_n)^{-1 } h_{z,r}(z+t\theta).
\end{equation}
Also denote by $a_-^u(r,t)$ and $a_+^u(r,t)$ the analogues for $u$ of $a_-(r,t)$ and $a_+(r,t)$;
we want to compare the two and then use Corollary \ref{zerosetreduced} to compute their limits.
First observe that by taking infimums and supremums in the two halves of \eqref{a6.43},
\begin{equation}\label{a6.46}
(1-5r^{\beta}) a_-^u(r,t) \leq a_-(r,t)
\ \text{ and } \ 
a_+(r,t) \leq (1+5r^{\beta}) a_+^u(r,t)
\end{equation}
for $t < r^{1+\alpha/17n}$. Next we use the expansion of $u$ near the point $z$ that is given by
\eqref{Ihavetotrythisoncetoseehowitfeelstogivelongnamestothingsbutifyouseethisyoumayreplacebygradient}
and \eqref{normalderivative}. We get that for $\theta \in \Gamma_\tau$,
\begin{equation}\label{a6.47}
u(z+t\theta) = \langle\nabla^+ u(z), t\theta \rangle_+ + o(t)
= t \theta_n q_+(z) + o(t).
\end{equation}
This implies that 
\begin{equation} \label{a6.48}
\lim_{t \to 0^+} a_-^u(r,t) = \lim_{t \to 0^+} a_+^u(r,t) = q_+(z).
\end{equation}
Now set $a_-(r) = \liminf_{t \to 0^+} a_-(r,t)$ and $a_+(r) = \limsup_{t \to 0^+} a_+(r,t)$.
It is clear that $a_-(r) \leq a_+(r)$, but by \eqref{a6.46}
\begin{equation}\label{a6.49}
a_-(r) \geq (1-5r^{\beta}) q_+(z)
\ \text{ and } \ 
a_+(r) \leq (1+5r^{\beta}) q_+(z).
\end{equation}
This still leaves some uncertainty concerning the existence of limits for the $a_\pm(r,t)$, which we shall
resolve by replacing $r$ with smaller radii for which the error tends to $0$. 
For what we said so far, it was enough to assume that 
$B( x_0, 4r_0)\subset \Omega$, but we made sure to assume that $B( x_0, 6r_0)\subset \Omega$,
so that our argument is also directly valid (without thinking about the proof) with $x_0=z$.
Thus the estimates above are also valid for the functions $h_{z,s}$, $s \in (0,r)$. In particular, \eqref{a6.49}
says that
\begin{equation}\label{a6.50}
(1-5s^{\beta}) q_+(z) \leq a_-(s) \leq a_+(s) \leq (1+5s^{\beta}) q_+(z).
\end{equation}

We can relate $h_{z,s}$ and $h_{z,r}$ (say, on $U \cap B(z,r/2)$) because they are both positive 
harmonic functions that vanish at the boundary. In particular, \eqref{eqn3.34} 
(with $\rho = r$ and $x_0 = z$) says that for $0 < s < r$, 
we can define the limit
\begin{equation}\label{a6.51}
\ell_{s,r}(z) = \lim_{x \in U \, ; \,  x \to z}\frac{h_{z,s}(x)}{h_{z,r}(x)}
\end{equation}
(see \eqref{eqn3.34}) and in addition $1/2 \leq \ell_{s,r}(z) \leq 2$. It is then clear that
\begin{equation}\label{a6.52}
\ell_{s,r}(z)  a_-(r) = a_-(s) \ \text{ and} \ \ell_{s,r}(z)  a_+(r) = a_+(s)
\end{equation}
and since \eqref{a6.50} implies that $a_-(s)$ and $a_+(s)$ both tend to $q_+(z)$,
we see that $a_-(r)= a_+(r)$.

So we proved that $a_(r,t)$ and $a_+(r,t)$ have a common limit $a(r)$. 
We intend to check that we can take
\begin{equation}\label{a6.53}
\frac{\partial^+ h_{z,r}}{\partial \nu}(z) = a(r)
\ \text{ and }\ 
\nabla^+ h_{z,r} = a(r) \nu
\end{equation}
in the definitions \eqref{normalderivative1}-\eqref{normalderivative}, but first observe that
\begin{equation}\label{a6.54}
(1-5r^{\beta}) q_+(z) \leq a(r) \leq (1+5r^{\beta}) q_+(z)
\end{equation}
by \eqref{a6.49}. Now we return to the 
definition \eqref{a6.45} and find that for $x=z+t\theta$, with $\theta \in \Gamma_\tau$,
we have the expansion
\begin{equation}\label{a6.55}
h_{z,r}(z+t\theta) = t\theta_n a(r) + o(t).
\end{equation}
This implies that $\frac{\partial^+ h_{z,r}}{\partial \nu}(z) = a(r)$, as in \eqref{normalderivative1},
and the only difference with the definition of $\nabla^+$ is that we restrict to the sector 
$\R \Gamma_\tau$. Notice first that $a(r)$ does not depend on $\tau$, because it gives the derivative
in the normal direction; this will allow us let $\tau$ tend to $0$ and use the Lipschitz property
for the remaining region. That is, let $M$ be a bound for the Lipschitz norm of $u$ in 
$B(x_0,3 r^{1+\alpha/17n})$. Then let $\varepsilon > 0$ be given, and choose 
$\tau = M^{-1} \varepsilon$. Then for $x=z+t\theta \in U$ such that $\theta \notin \Gamma_\tau$, 
and if $t$ is small enough (depending on the good approximation of $\d U$ by its tangent plane),
\begin{equation}\label{a6.56}
|h_{z,r}(x) - t\theta_n a(r)| \leq |t\theta_n a(r)| + M \dist(x,\d U)
\leq t \tau a(r) + 2M\tau t \leq (a(r)+2) \varepsilon t.
\end{equation}
Since \eqref{a6.55} gives a good enough control when $\theta \in \Gamma_\tau$,
we get the full \eqref{a6.53}.
\ms
This gives the desired control on the function $h_{z,r}$, but our statement involved the slightly
different function $h_{x_0,r}$.
Notice that (if $\rho_4$ is chosen small enough, so that $r^{1+\alpha/17n} < r/10$),
$h_{z,r}$ and $h_{x_0,r}$ are both non-negative harmonic functions on $U \cap B(z, r/2)$,
that vanish on $\d U \cap B(z, r/2)$. By the local NTA property of $U$ and the comparison principle,

there exist constants $C \geq 1$ and $\eta \in (0,1)$ (that depend on $r_0$ and the usual constants
through the NTA constants) such that
\begin{equation}\label{eqn6.3A}
\left|\frac{h_{x_0,r}(x)}{h_{z,r}(x)} - \frac{h_{x_0,r}(y)}{h_{z,r}(y)}\right|
\le C \,\frac{h_{x_0,r}(x)}{h_{z,r}(x)} \left(\frac{|x-y|}{r}\right)^\eta
\end{equation}
for $x,y \in U \cap B(z, r/4)$. See the proof of \eqref{eqn3.16} for some additional detail.
Then, by the proof of \eqref{eqn3.34} (using the continuity of the ratio at the boundary),
there exists
\begin{equation}\label{eqn6.4A}
\ell(z)=\lim_{x\to z} \frac{h_{x_0,r}(x)}{h_{z,r}(x)}.
\end{equation}
At this point, for each $\tau \in (0,1)$ \eqref{a6.55} gives us a nice expansion for $h_{z,r}$
in the cone over $\Gamma_\tau$, and \eqref{eqn6.4A} implies that we have the same expansion
for $h_{x_0,r}$, with $a(r)$ replaced by $\ell(z) a(r)$. We can control the points that lie outside of the 
cone as we did for \eqref{a6.56}, and now the existence of $\frac{\partial^+ h_{x_0,r}}{\partial \nu} (z)$ 
and \eqref{eqn6.1B} follow from (the proof of) \eqref{a6.53}. 
Finally, for the inequalities in \eqref{eqn6.1A}, observe that \eqref{a6.43} also holds for $h_{x_0,r}$,
which gives a good control on $h_{x_0,r}/u$, and $\frac{\partial^+ h_{x_0,r}}{\partial \nu} (z) = q_+(z)$,
by \eqref{normalderivative1}. Proposition \ref{rn-derivative} follows.
\end{proof}

We end this section by showing that $h_{x_0,\rho}$ satisfies Definition 5.1 in \cite{AC}.

\begin{proposition}\label{rn-derivative}
The function $h_{x_0,r}$ of Corollary \ref{cor6.1A} satisfies
\begin{equation}\label{eqn6.10}
-\int_U \langle\nabla h_{x_0,r}, \nabla \zeta\rangle 
=\int_{\partial\{u>0\}} \zeta \; \frac{\partial^+ h_{x_0,r}}{\partial\nu}\; d\mathcal{H}^{n-1} 
\end{equation}
for all $\zeta\in C^1_c(B(x_0,r^{1+\alpha/17n}))$. 
\end{proposition}

\begin{proof}
Set $B = B(x_0, r^{1+\alpha/17n})$. 
By its definition near \eqref{a3.4}, $h_{x_0,r}$ is continuous on $B(x_0,r)$
and harmonic on $U \cap B(x_0,r)$; in addition, it satisfies the estimate \eqref{eqn3.11BB}
in $5B$, and Theorem~4.3 in \cite{AC} guarantees that $\lambda=\Delta h_{x_0,r}$ 
is an Ahlfors regular measure on $\d U \cap 3B$, say.
Let $k$ denote the Radon-Nikodym of $\lambda$ with respect to $\H^{n-1}$, thus for 
$\zeta\in C^1_c(B)$ we have
\begin{equation}\label{eqn6.11}
-\int \langle\nabla h_{x_0,r}, \nabla \zeta\rangle =\int_{\partial\{u>0\}} \zeta k\, d\mathcal{H}^{n-1}.
\end{equation}
Since $\nabla h_{x_0,r} = 0$ almost everywhere on $B \sm U$ (because $h_{x_0,r} = 0$ there),
the proposition will follow as soon as we prove that
\begin{equation}\label{a6.61}
k(z) = \frac{\partial^+ h_{x_0,r}}{\partial\nu}(z)
\ \text{ for $\H^{n-1}$-almost every $z\in \d U \cap B$.}
\end{equation}
Notice that $k\in L_{loc}^\infty(\H^{n-1}\res\partial U)$ near $B$, because
$\lambda$ is Ahlfors regular. The same arguments as those used in the proofs of 
Lemmata 3.1, 3.2 and 3.4 in \cite{KT1} show that the non-tangential limit $F$ of 
$\nabla h_{x_0,r}$ exists $\mathcal{H}^{n-1}$-a.e on $\partial U\cap B$ and,
\begin{equation}\label{eqn6.12}
F(z)= k(z)\nu(z) \ \  \text{for $\mathcal{H}^{n-1}$-almost every } z\in \partial U\cap B.
\end{equation}
Thus we just need to check that 
$\frac{\partial^+ h_{x_0,r}}{\partial\nu}(z) = \langle F(z) , \nu(z)\rangle$ a.e. on $\partial U\cap B$.
Recall that almost every $z \in \partial U\cap B$ lies in $\d^\ast U$, so Corollary \ref{cor6.1A}
applies to it, and gives the existence of the normal derivative $\frac{\partial^+ h_{x_0,r}}{\partial\nu}(z)$.
Here we use the definition \eqref{normalderivative1}, which mean that we have the expansion 
\begin{equation}\label{eqn6.13}
h_{x_0,r}(z+t\nu(z))=t \,\frac{\partial^+ h_{x_0,r}}{\partial\nu}(z) +o(t),
\end{equation}
valid for $t$ small, and where $\nu(z)$ is the same normal derivative that points towards $U$
as in \eqref{eqn6.12}, say. As before, the convention is that $t^{-1}o(t)$ tends to $0$.
The fact that $z+t\nu(z) \in U$ for $t$ small is easy here, since $\d U$
has a true tangent plane at $z$. We apply this to $2t$ and subtract to get that
\begin{equation}\label{a6.64}
h_{x_0,r}(z+2t\nu(z))-h_{x_0,r}(z+t\nu(z)) = t \,\frac{\partial^+ h_{x_0,r}}{\partial\nu}(z) +o(t)
\end{equation}
On the other hand, by the fundamental theorem of calculus (and for $t$ small),
\begin{equation}\label{a6.65}
h_{x_0,r}(z+2t\nu(z))-h_{x_0,r}(z+t\nu(z))
= t \langle\nabla h_{x_0,r}(z + \xi \nu(z)),\nu(z)\rangle
\end{equation}
for some $\xi \in [t,2t]$. Let $t$ tend to $0$. If $z$ is also such that the notangential limit at $z$
of $\nabla h_{x_0,r}$ is $F(z)$, then $\nabla h_{x_0,r}(z + \xi \nu(z))$ tends to $F(z)$
and the comparison of \eqref{a6.64} and \eqref{a6.65} yields $\frac{\partial^+ h_{x_0,r}}{\partial\nu}(z)
= \langle F(z),\nu(z)\rangle$, as needed. 
Proposition \ref{rn-derivative} follows.

\end{proof}

\section{Free boundary regularity for almost-minimizers}\label{bdryreg}

In this section we show that if $u$ is an almost-minimizer for $J^+$ in $\Omega\subset\R^{n}$ 
with $q_+$ H\"older continuous and bounded below, then the set 
$\mathcal R\subset \partial\{u>0\}$ (see Definition \ref{defn:flatpoints}) is locally 
a $C^{1,\beta}$ $(n-1)$-submanifold (see Theorem \ref{main-reg-theo}).
The definitions and arguments used in this section are reminiscent of those that appear in \cite{AC}. 
We discuss some of the technical arguments that concern harmonic functions (and specifically
weak solutions) in Section \ref{appendix}. 

In this whole section, we assume that $u$ is an almost-minimizer for $J^+$ in 
$\Omega\subset\R^{n}$, and that
\begin{equation} \label{a7.1}
q_+ \in C^{\alpha}(\Omega) \cap L^{\infty}(\Omega),
\ \text{ and there is a constant $c_0 > 0$ such that $q_+\ge c_0>0$ on } \Omega. 
\end{equation}
We set $U = \big\{ x\in \Omega \, ; \, u(x) > 0 \big\}$ and $\Gamma^+(u) = \Omega \cap \d U$
as usual. 

\begin{defn}\label{defn:flatpoint} 
Let $\sigma>0$. For $x_0 \in \Gamma^+(u)$ 
and $r_0>0$ with $B(x_0,r_0)\subset \Omega$ we say that 
\begin{equation}\label{f-notation}
u\in \mathcal{F}(\sigma;x_0,r_0)\hbox{ in the direction  } e_0\in\SS^{n-1}
\end{equation}
if for $x\in B(x_0,r_0)$, 
\begin{eqnarray}\label{f-definition}
\left\{ \begin{array}{ll}
u(x)=0 &\hbox{ if }  \langle x-x_0,e_0\rangle\le -\sigma r_0\\[2mm]
u(x)\ge q_+(x_0)[\langle x-x_0,e_0\rangle -\sigma r_0] &\hbox{ if }  \langle x-x_0,e_0\rangle\ge \sigma r_0.\end{array}\right.
\end{eqnarray}
\end{defn}

\ms

\begin{lemma}\label{reifenberg}
Let $u$ be an almost-minimizer for $J^+$ in $\Omega\subset\R^{n}$ and $\sigma>0$. 
If $u\in \mathcal{F}(\sigma;x_0,r_0)$  in the direction $e_0\in\SS^{n-1}$
and $L_0=x_0 + \langle e_0\rangle^\perp$, then
\begin{equation}\label{reif-flat}
\frac{1}{r_0}D[\partial\{u>0\}\cap B(x_0, r_0), L_0 \cap B(x_0, r_0)]\le C\sigma,
\end{equation}
where $D$ denotes the Hausdorff distance, and $C$ is a constant depending on $n$
\end{lemma}

\begin{proof}
Notice that if $\sigma\ge 2^{-n}$, then we have \eqref{reif-flat} with $C=2^n$.  
Thus let $\s<2^{-n}$.
Note that \eqref{f-definition} implies that $|\langle y_0-x_0, e_0\rangle|\le \sigma r_0$
for $y_0 \in \d U \cap B(x_0, r_0)$. 
For $y\in L_0\cap B(x_0, r_0\sqrt{1-4\s^2})$ observe that 
$u(y +2\s r_0e_0)>0$ and $u(y -2\s r_0e_0)=0$, thus since $u$ is continuous there is 
$y'=y +t r_0e_0\in \partial\{u>0\}\cap B(x_0, r_0\sqrt{1-4\s^2})$ with $t\in (-2\s, 2\s)$, 
thus $|y-y'|\le 2\s r_0$. For $z\in L_0\cap B(x_0, r_0)\backslash B(x_0, r_0\sqrt{1-4\s^2})$ 
there is $y\in L_0\cap B(x_0, r_0\sqrt{1-4\s^2})$ with $|z-y|\le \s r_0$ and using $y'$ as above 
we have $|z-y'|\le 3\s r_0$. 

\end{proof}

With the notation of Definitions \ref{defn:flatpoints} and \ref{defn:flatpoint}, 
Proposition \ref{Missmallestatflatpoints} implies that regular points are flat.

\begin{corollary}\label{flatp--f-flat}
Let $u$ be an almost-minimizer for $J^+$, assume \eqref{a7.1}, and let $\sigma > 0$ be given.
Then for every $x_0\in \mathcal R$ there exists $\rho_\sigma>0$
such that for $0<\rho\le \rho_\sigma$ there is $e_\rho\in\SS^{n-1}$ such that 
$u\in {\mathcal F}(\sigma; x_0,\rho)$ in the direction $e_\rho$. 
\end{corollary}

Note that by Corollary \ref{cor:regularpoints}, Corollary \ref{flatp--f-flat} applies to points in the reduced boundary $\partial^\ast U \cap \Omega$. 
Our first result uses Theorem \ref{t4.3} to study how the fact that 
$u\in \mathcal{F}(\sigma;x_0,r_0)$ in the direction $e$ translates into the behavior of 
the intermediate functions $h_{x_0,\rho}$.

\begin{lemma}\label{uflat-hflat}
Set $\gamma = \alpha/17n$ and $\tilde\gamma = 2 \gamma$, assume \eqref{a7.1},
and $u$ be an almost-minimizer for $J^+$ in $\Omega$. Then 
for $r_0>0$ there exist a radius $\rho_5>0$, 
depending only on $n$, $c_0$, $\|q_+\|_{L^\infty}$, 
$\|q_+\|_{C^\alpha}$, $\kappa$, $\alpha$, $\sigma$ and $r_0$, and a constant $\mu\in(0,1)$,
depending on $n$ and $\alpha$, such that if $0<\rho<\rho_5$,  
$x_0\in\partial\{u>0\}$, $B(x_0,4r_0)\subset\Omega$ and 
$u\in {\mathcal F}(\sigma;x_0,\rho^{1+\tilde\gamma})$ in the direction $e_0$ 
then the function $h_{x_0,\rho}$ defined near \eqref{a3.4}  
is such that for $x\in B(x_0,\rho^{1+\tilde\gamma})$,
\begin{eqnarray}\label{h-properties}
\left\{ \begin{array}{ll}
h_{x_0,\rho}(x)=0 &\hbox{ if }  \langle x-x_0,e_0\rangle\le -2\sigma \rho^{1+\tilde\gamma}\\[2mm]
h_{x_0,\rho}(x)\ge q_+(x_0)[\langle x-x_0,e_0\rangle -2\sigma \rho^{1+\tilde\gamma}] &\hbox{ if }  \langle x-x_0,e_0\rangle\ge 2\sigma \rho^{1+\tilde\gamma}\\[2mm]
|\nabla h_{x_0,\rho}(x)|\le  q_+(x_0)(1+ \rho^{\mu}).\end{array}\right.
\end{eqnarray}
Moreover for  $z\in\partial^\ast \{u>0\}\cap B(x_0,\rho^{1+\tilde\gamma})$
\begin{equation}\label{normal-derivative}
\frac{\partial^+ h_{ x_0,\rho}}{\partial \nu}(z)\ge q_+(x_0) (1-\rho^\mu).
\end{equation}
\end{lemma}

\begin{proof}
In addition to the large ball $B(x_0,\rho)$, we shall often use the 
smaller $B = B(x_0,\rho^{1+\gamma})$ and the even smaller ball 
$\wt B = B(x_0,\rho^{1+2\gamma})$. Recall that 
\begin{equation}\label{eqn7.1}
\left\{
  \begin{array}{ll}
    \Delta h_{ x_0,\rho}=0 & {\rm{in}} \ B(x_0,\rho)\cap U 
    \\
    h_{ x_0,\rho}=u & {\rm{in}}\ \Omega\backslash [B(x_0,\rho)\cap U] . 
  \end{array}
\right.
\end{equation}
Let us decide to pick $\rho_5 \leq \rho_3$, where $\rho_3$ comes from Theorem \ref{t3.1}. Then
$h_{x_0,\rho}$ satisfies \eqref{a3.44}, i.e. 
\begin{equation}\label{eqn7.2}
(1-5\rho^\beta)u(x)\le h_{x_0,\rho}(x)\le (1+5\rho^\beta)u(x)\hbox{   for   }
x\in 4B,
\end{equation} 
where $\beta \in (0,1)$ is as in Theorem \ref{t3.1}.
Moreover, if we also take $\rho_5$ smaller than $\rho_4$ in Corollary \ref{cor6.1A},
and if $z\in\partial^\ast U\cap B$, then by \eqref{eqn6.1A} 
\begin{equation}\label{eqn7.3}
(1-5\rho^\beta)q_+(z)\le \frac{\partial^+ h_{ x_0,\rho}}{\partial \nu}(z)\le (1+5\rho^\beta)q_+(z),
\end{equation} 
and 
\begin{eqnarray}\label{eqn7.3A}
 \frac{\partial^+  h_{ x_0,\rho}}{\partial \nu}(z)&\ge &(1-5\rho^\beta)q_+(x_0) 
  + (1-5\rho^\beta)(q_+(z)-q_+(x_0)) 
  \nonumber\\
 &\ge & q_+(x_0) (1-5\rho^\beta -c \rho^\alpha(1+\gamma))\\
 &\ge & q_+(x_0) (1-6\rho^\beta),
 \nonumber
\end{eqnarray} 
provided that we choose $\rho_5$ small enough, and because $\beta$ 
was chosen smaller than $\alpha$. We picked  $\tilde \gamma =2\gamma$ and assume that
 $u\in \mathcal{F}(\sigma;x_0,\rho^{1+\tilde\gamma})$ in the direction $e_0$. 
Then \eqref{f-definition} and \eqref{eqn7.2} yield for $x\in \wt B$
\begin{equation}\label{eqn7.4}
h_{x_0,\rho}(x)=0 \hbox{ if }  \langle x-x_0,e_0\rangle\le -\sigma \rho^{1+\tilde\gamma}.
\end{equation}
Moreover, provided that $\rho_4^\beta<1$ and $0<\rho<\rho_4$, then 
for $x\in \wt B$ such that $\langle x-x_0,e_0\rangle\ge \sigma \rho^{1+\tilde\gamma}$, 
\begin{eqnarray}\label{eqn7.5}
h_{x_0,\rho}(x)&\ge & (1-\rho^\beta)u(x)\nonumber\\
&\ge &(1-\rho^\beta)
q_+(x_0)[\langle x-x_0,e_0\rangle -\sigma  \rho^{1+\tilde\gamma}]\nonumber\\
&\ge & q_+(x_0)[\langle x-x_0,e_0\rangle -\sigma \rho^{1+\tilde\gamma}-\rho^{1+\beta+\tilde\gamma}] \nonumber\\
&\ge &q_+(x_0)[\langle x-x_0,e_0\rangle -2\sigma \rho^{1+\tilde\gamma}].
\end{eqnarray}
Since $h_{x_0,\rho}$ is harmonic in $B( x_0,\rho)\cap U$, 
so is $\nabla h_{x_0,\rho}$. 
By \eqref{eqn7.2} and Theorems 5.1 and 10.2 in \cite{DT}, there exists $C>0$, 
that depends on the usual constants, such that for 
$x\in B \cap U = B(x_0,\rho^{1+\gamma})\cap U$,  
\begin{equation}\label{eqn7.6}
C^{-1}\delta(x)\le h_{x_0,\rho}(x)\le C\delta(x)
\end{equation}
where $\delta(x)={\rm{dist}}(x,\partial U)$. 
Thus by standard PDE arguments (see \eqref{eqn3.7}), \eqref{eqn7.6} implies that 
$|\nabla h_{x_0,\rho}|$ it is bounded on $B\cap U$. 
Recall that $U$ is locally NTA in $\Omega$ (see Theorem~\ref{t2.3}). 
Let $\omega$ denote the harmonic measure of $B\cap U$. 
Theorem \ref{t4.1}, together with  
the fact that on a connected domain, harmonic measures with different poles are mutually 
absolutely continuous, ensures that for $x\in \wt B = B(x_0,\rho^{1+2\gamma})$, $\omega^x$ 
and $\mathcal H^{n-1}$ are mutually absolutely continuous. 
This fact plus \eqref{eqn7.3} yield, for  $x\in \wt B\cap U$, 
\begin{eqnarray}\label{eqn7.6A} 
|\nabla h_{x_0,\rho}(x)| &=& \Big| \int_{\partial (U \cap B)}
\nabla h_{x_0,\rho}(z)\, d\omega^x(z) \Big| 
\nonumber\\
&\le &\int_{\partial B\cap U} |\nabla h_{x_0,\rho}(z)|\, d\omega^x(z) +
\int_{\partial U \cap B} |\nabla h_{x_0,\rho}(z)|\, d\omega^x(z),
\end{eqnarray}
where in the second integral $\nabla h_{x_0,\rho}(z)$ denotes the nontangential limit
of $\nabla h_{x_0,\rho}$ at $z_0$, whose existence follows Lemmata 3.1, 3.2, and 3.4
in \cite{KT1}, and was already used in Proposition~\ref{rn-derivative} under the name of $F(z)$
(see \eqref{eqn6.12}). It follows from \eqref{eqn6.12}, \eqref{a6.61}, and
\eqref{eqn7.3} that 
\begin{equation} \label{a7.15}
|\nabla h_{x_0,\rho}(z)| = |F(z)| = k(z) = \frac{\partial^+ h_{x_0,\rho}}{\partial\nu}(z)
\leq (1+5\rho^\beta)q_+(z)
\end{equation}
for $\H^{n-1}$-almost every $z\in \d U \cap B$.

For the first integral we use the fact that $|\nabla h_{x_0,\rho}| \leq M$ for some $M \geq 0$
that does not depend on $\rho$, and altogether \eqref{eqn7.6A} yields
\begin{equation} \label{a7.16}
|\nabla h_{x_0,\rho}(x)| 
\leq M\omega^x(\partial B\cap U) 
+ (1+5\rho^\beta)\int_{\partial U\cap B}q_+(z)\, d\omega^x(z).
\end{equation}
By the assumption that $q_+\in C^{\alpha}$ the second term in \eqref{a7.16} is bounded by
\begin{eqnarray}\label{eqn7.7} 
(1+5\rho^\beta)q_+(x_0) + C\rho^{\alpha(1+\gamma)} 
\leq (1+ 10\rho^\beta + C'\rho^{\alpha(1+\gamma)})q_+(x_0),
\end{eqnarray}
where we have used the fact that $q_+\ge c_0>0$.
Since  
$\omega^x(\partial B\cap U)$ is a harmonic function on $B\cap U$  
which vanishes continuously on $\frac12 B \cap \partial U$ and that $U$ 
is locally NTA we have (see \cite{JK}) that for $x\in \wt B = B(x_0,\rho^{1+2\gamma})$
\begin{equation}\label{eqn7.8}
\omega^x(\partial B\cap U) 
\le C\left(\frac{|x-x_0|}{\rho^{1+\gamma}}\right)^\eta\le C\rho^{\gamma\eta}
\end{equation}
where $C$ and $\eta$ depend on the local NTA constants.
Combining \eqref{eqn7.6A}, \eqref{eqn7.7}, \eqref{eqn7.8} and using the fact that $q_+\ge c_0>0$ we obtain that for $x\in \wt B\cap U$ 
\begin{equation}\label{eqn7.9}
|\nabla h_{x_0,\rho}(x)|
\le (1+ 5 \rho^\beta + C'\rho^{\alpha(1+\gamma)}+ C\rho^{\eta\gamma})q_+(x_0).
\end{equation}
Letting $\mu=\frac{1}{2}\min\{\beta,\alpha, \eta\gamma\}$ then choosing $\rho_5$ 
such that $(5+C+C')\rho_4^\mu<1$, $6\rho^{\beta/2}<1$ and recalling that 
$\tilde \gamma = 2\gamma$, \eqref{eqn7.9} and \eqref{eqn7.3A} 
become
\begin{equation}\label{eqn7.10}
\sup_{\wt B}|\nabla h_{x_0,\rho}|
\le q_+(x_0)(1+\rho^\mu) \ \hbox{ and } \ 
\frac{\partial^+ h_{x_0,\rho}}{\partial \nu}(z)\ge q_+(x_0) (1-\rho^\mu).
\end{equation}
Note that \eqref{eqn7.4}, \eqref{eqn7.5} and \eqref{eqn7.10} yield \eqref{h-properties} 
and \eqref{normal-derivative}.

\end{proof}

\begin{lemma}\label{u-flat-improv}
Let $u$ be an almost-minimizer for $J^+$ in $\Omega \subset \R^n$ and assume \eqref{a7.1} holds. In addition, let $x_0 \in \Gamma^+(u)\equiv \partial U\cap \Omega$ and $r_0 > 0$ be such that $B(x_0, 4r_0) \subset \Omega$. Given $\theta \in (0,1)$ there exist $\sigma_{n, \theta} > 0$ and $\eta = \eta_{n, \theta} \in (0,1)$ so that if $\sigma \leq \sigma_{n,\theta}$, then we can choose an $r_1 > 0$ (which depends only on $n, c_0, \|q_+\|_{L^\infty}, \|q_+\|_{C^{\alpha}}, \kappa, \alpha, \sigma$ and $r_0$) such that for all $0 < r < r_1$, if $u \in \mathcal F(\sigma; x_0 ,r)$ in the direction $e_{x_0, r}$, then $u \in \mathcal F(\theta\sigma; x_0, \eta r)$ in some direction $e_{x_0, \eta r}$ where \begin{equation}\label{apb-eqn7B}
|e_{x_0,r}-e_{x_0,\eta r}|\le C\sigma.
\end{equation} (Here $C > 0$ depends only on $n, c_0, \|q_+\|_{L^\infty}, \|q_+\|_{C^{\alpha}}, \kappa, \alpha$ and $r_0$). 
\end{lemma}

\begin{proof}
Let $\theta\in(0,1)$ be given, and set $\theta' = \theta/3$. 
Let  $\s_{n,\theta'}>0$ and $\eta' = \eta_{n,\theta'}\in(0,1)$ 
be as in Corollary \ref{main-cor}. Let $\beta$ as in Theorem \ref{t3.1},  
$\tilde\gamma$ and $\mu$ as in Lemma \ref{uflat-hflat}. 
For $\s\le \frac12 \s_{n,\theta'}$ let $\rho_5$ be as in Lemma \ref{uflat-hflat}. 
Let $\rho_1 \leq \min\{\rho_5, (\theta' \s)^\frac{1}{\beta}, 
(\frac12 \s_{n,\theta'}\s^2)^\frac{1}{\mu}\}$, to be chosen later, and set
$r_1 = \rho_1^{1+\tilde\gamma}$.

For $0<r<r_1$ and $x_0\in\partial U$ such that $B(x_0,4r_0)\subset\Omega$, set
\begin{equation} \label{a7.22}
\rho = r^ \frac{1}{1+\tilde\gamma}, \ 
\tau=\rho^\mu = r^\frac{\mu}{1+\tilde\gamma},
\ \text{ and } v = h_{x_0,\rho} \, ;
\end{equation}
thus $\rho < \rho_1 \leq \rho_5$. All this is arranged so that
if $u\in {\mathcal F}(\sigma;x_0,r)$ in the direction $e_{x_0,r}$,
Lemma~\ref{uflat-hflat} says that 
$v\in F(2\s, 2\s;\tau)$ in $B(x_0,r)$ in the direction $-e_{x_0,r}$,
where the notation for $F$ will be given in Definition \ref{weak-ap-sol}.
Also, $\tau\s^{-2} \leq \rho_1^\mu \s^{-2} \leq \frac12 \s_{n,\theta'}$, and by our choice of constants
Corollary \ref{main-cor} guarantees that $v\in F(2\theta'\s, 2\theta'\s;\tau)$ in $B(x_0,\eta' r)$ 
in some direction $-e_{x_0,\eta' r}$ such that  
$|e_{x_0,r}-e_{x_0,\eta' r}|\le C\sigma$ (see \eqref{apb-eqn7A}). Thus  
for $x\in B(x_0,\eta' r)$
\begin{eqnarray}\label{eqn7.11}
\left\{ \begin{array}{ll}
v(x)=0 &\hbox{ if }  \langle x-x_0,e_{x_0,\eta' r}\rangle\le -2\theta'\sigma \eta r\\[2mm]
v(x)\ge q_+(x_0)[\langle x-x_0,e_{x_0,\eta' r}\rangle -2\theta'\sigma \eta' r] &\hbox{ if }  \langle x-x_0,e_{x_0,\eta' r}\rangle\ge 2\theta' \sigma \eta' r.\end{array}\right.
\end{eqnarray}
By the definition of $v$, \eqref{eqn7.11} ensures that 
\begin{equation} \label{a7.24}
u(x)=0 \ \text{ for $x\in B(x_0, \eta' r)$ such that }
\langle x-x_0,e_{x_0,\eta' r}\rangle\le -2\theta'\sigma \eta r.
\end{equation}

Next consider $x\in B(x_0,\eta' r)$ such that 
$\langle x-x_0,e_{x_0,\eta' r}\rangle > 2\theta' \sigma \eta' r$ 
(so that $u(x) > 0$ by \eqref{eqn7.11}). 
If we choose $\rho_1$ also smaller than $\rho_3$ from Theorem \ref{t3.1}, 
then this theorem applies to the pair $(x_0,\rho)$, and since
$x \in U \cap B(x_0,\eta' r) \subset  B(x_0, \rho^{1+\gamma})$ (because $\eta' < 1$),
\eqref{a3.44} yields
\begin{equation} \label{a7.25}
u(x) \geq (1+5\rho^\beta)^{-1} h_{x_0,\rho}(x) = (1+5\rho^\beta)^{-1} v(x),
\end{equation}
and hence by \eqref{eqn7.11} 
\begin{eqnarray} \label{eqn7.12}
u(x) &\ge& (1+5\rho^\beta)^{-1} v(x) 
\geq  (1+5\rho^\beta)^{-1} q_+(x_0) [\langle x-x_0,e_{x_0,\eta' r}\rangle -2\theta'\sigma \eta' r]
\nonumber\\
&\geq& q_+(x_0) 
\big[\langle x-x_0,e_{x_0,\eta' r}\rangle -2\theta'\sigma \eta' r - 5\rho^\beta |x-x_0| \big]
\nonumber\\
&\geq& q_+(x_0) 
\big[\langle x-x_0,e_{x_0,\eta' r}\rangle -3\theta'\sigma \eta' r \big]
\end{eqnarray}
because $(1+5\rho^\beta)^{-1} \geq 1-5\rho^\beta$ and $|x-x_0|\leq \eta' r$, and if 
$\rho_1$ is small enough (depending on $\theta'$, $\sigma$, and $\eta'$).

By \eqref{a7.24} and \eqref{eqn7.12}, $u\in {\mathcal F}(3\theta'\sigma;x_0,\eta' r)$ 
in the direction $e_{x_0,\eta' r}$. Choosing $\theta'=\frac{\theta}{3}$, $\eta_\theta=\eta'$ 
and recalling \eqref{apb-eqn7A} we conclude that $u\in {\mathcal F}(\theta\sigma;x_0,\eta r)$ 
and \eqref{apb-eqn7B} holds.
\end{proof}

\begin{theorem}\label{main-reg-theo}
Let $u$ be an almost-minimizer for $J^+$ in $\Omega\subset\R^{n}$, 
and assume that \eqref{a7.1} holds. 
There exists $\tilde \alpha\in (0,1)$ depending on $c_0$, $\alpha$ and $n$ such that 
$\mathcal R$ is (locally) a $C^{1,\tilde\alpha}$ $(n-1)$-submanifold.
\end{theorem}

\begin{proof}
Fix $\theta\in (0,1)$ and let $\s_{n,\theta}$ as in Lemma \ref{u-flat-improv}. 
Choose $\sigma'<\frac{\s_{n,\theta}}{10}$. 
Let $r\le\frac{1}{4}\min\{r_1, \rho_\s'\}$ where $r_1$ is an in Lemma \ref{u-flat-improv} and 
$\rho_\s$ is as in Corollary \ref{flatp--f-flat}. In particular  
$u\in\mathcal{F}(\s';x_0,4r)$ in the direction $e_{x_0, 4r}$ which by Lemma \ref{reifenberg}
yields $|\langle x_0-y_0, e_{x_0,4r}\rangle|\le 4\s' r$ for $y_0\in B(x_0, r)\cap\partial U$. 
Thus if $x\in B(y_0, r)$ and $\langle x-y_0,e_{x_0,4 r}\rangle\le -8\sigma' r$ then $u(x)=0$. 
Moreover if $\langle x-y_0,e_{x_0,4 r}\rangle\ge 8\sigma' r$ then 
$\langle x-x_0,e_{x_0,4 r}\rangle\ge 4\sigma' r$ and 
\begin{equation} \label{a7.27}
u(x) \geq q_+(x_0)[\langle x-x_0, e_{x_0,4r}\rangle - 4\sigma' r]
= q_+(y_0)[\langle x-x_0, e_{x_0,4r}\rangle - 4\sigma' r] + {\mathcal E}
\end{equation}
where 
\begin{eqnarray} \label{a7.28}
|{\mathcal E}| &=&  \big|(q_+(x_0)-q_+(y_0))[\langle x-x_0, e_{x_0,4r}\rangle - 4\sigma' r] \big|
\leq C |x_0-y_0|^{\alpha} [|x-x_0| + 4\sigma' r]
\nonumber\\
&\leq& C r^\alpha [2r + 4\sigma' r] \leq \sigma' r q_+(x_0)
\end{eqnarray}
if $r_1$ is chosen small enough (depending on the $\sigma'$, the H\"older constants for $q_+$,
and $c_0$ in particular). Thus by \eqref{a7.27}
\begin{equation} \label{eqn7.13}
u(x) \geq q_+(x_0)[\langle x-x_0, e_{x_0,4r}\rangle - 5\sigma' r].
\end{equation}
Thus if $u\in\mathcal{F}(\s';x_0,4r)$ then for all $y_0\in B(x_0, r)\cap\partial U$, 
\eqref{eqn7.13} ensures that $u\in\mathcal{F}(10\s';y_0,r)$ in the same direction. 
Letting $\s=10\s'<\sigma_{\theta,n}$ we conclude that for $x_0\in\mathcal{R}$ there exists 
$r\in (0, r_1)$ such that for $y_0\in B(x_0, r)\cap\partial\{u>0\}$,
$u\in\mathcal{F}(\s;y_0,r)$ in the direction $e_{y_0,r}=e_{x_0,4r}$. 
An iterative application of Lemma \ref{u-flat-improv} ensures that there exists $\eta$ so 
that for $m\in\NN$, $u\in\mathcal{F}(\theta^m\s;y_0,\eta^m r)$ 
in a direction $e_{y_0, \eta^m r}$ such that
\begin{equation}\label{eqn7.14}
|e_{y_0, \eta^m r}-e_{y_0, \eta^{m-1} r}|\le C \theta^{m-1}\s.
\end{equation}
Furthermore by Lemma \ref{reifenberg} 
\begin{equation}\label{eqn7.14A} 
D[\partial U 
\cap B(y_0, \eta^{m} r), L_{e_{y_0, \eta^{m} r}}  
\cap B(y_0, \eta^{m} r)]\le C \theta^{m}\s\eta^{m} r.
\end{equation}
Let $\tilde\alpha$ be such that $\theta=\eta^{\tilde\alpha}$; note that
for $s<r$ there is $m\in\NN$ such that $\eta^{m+1} r\le s<\eta^mr$ and \eqref{eqn7.14A} yields
\begin{eqnarray}\label{eqn7.14B}
\frac{1}{s} D[\partial\{u>0\}\cap B(y_0, s), L_{e_{z_0, \eta^{m} r}}\cap B(y_0, s)]
&\le& C\theta^{m}\s\frac{\eta^{m} r}{s} 
\nonumber\\
&\le& C\theta^{m}\s\eta^{-1} 
= C\s\eta^{-1}\theta^{-1}(\eta^{m-1})^{\tilde\alpha} 
\nonumber\\
&\le& C'\left(\frac{s}{r}\right)^{\tilde\alpha}= C'' s^{\tilde\alpha}.
\end{eqnarray}
Hence for each $x_0\in \mathcal R$ there exists $r>0$ such that the hypothesis of 
Proposition 9.1 in \cite{DKT} holds in $B(x_0, r)\cap\partial U$, 
which ensures that  $B(x_0, r)\cap\partial U$ 
is a $C^{1,\tilde\alpha}$ $(n-1)$-submanifold. 
Since $\mathcal R$ is an open subset of $\partial U$ 
by Corollary \ref{flatpointsareopen},
we also get that $\mathcal R$ is (locally) a $C^{1,\tilde\alpha}$ $(n-1)$-submanifold of $\R^n$.

\end{proof}

Combining Theorem \ref{main-reg-theo} and Corollaries \ref{cor:regularpoints} and \ref{low-dim} 
we get the following. 

\begin{corollary}\label{regularity-cor}
Let $u$ be an almost-minimizer for $J^+$ in $\Omega\subset\R^{n}$, and assume 
that \eqref{a7.1} holds.  
Then
\begin{equation}
\partial\{u>0\}=\mathcal R\cup\mathcal S,
\end{equation}
where
$\mathcal S$ is a closed set with $\mathcal H^{n-1}(\mathcal S)=0$ and $\mathcal R$ is a $C^{1,\tilde\alpha}$ $(n-1)$-submanifold for some $\tilde\alpha$ that depends only on 
$n$, $\alpha$, $||q_+||_\infty$, and $c_0$. Furthermore $\mathcal S=\emptyset$ when $n=2,3,4$. 
\end{corollary}

\section{Dimension of the Singular Set}
\label{dimensionofSS} 
In this section we establish bounds on the Hausdorff dimension of the singular set 
$\Gamma^+ \sm \mathcal R$ of the free boundary for almost minimizers to the one-phase problem. 

The arguments here follow very closely those of Sections 3 and 4 in Weiss \cite{W}, where analogous results 
for minimizers of $J^+$ are proven. Let $k^*$ be the smallest natural number such there exists a stable one-homogeneous  globally defined minimizer $u:\mathbb R^{k^*}\rightarrow \mathbb R$ which is not the half plane solution.  The work of Caffarelli-Jerison-Kenig \cite{CJK}, Jerison-Savin \cite{JS} and De Silva-Jerison \cite{DeJ}, implies that $4 < k^* \leq 7$ but the exact value is still an open question. 

The assumptions for this section are the same as for Section \ref{bdryreg}: $u$ is an almost 
minimizer for $J^+$ in $\Omega \subset \mathbb R^n$, and $q_+$ is H\"older continuous,
bounded, and bounded below. We still denote by $\mathcal R$ the set of regular points of 
$\Gamma^+(u) = \Omega \cap \partial \{u > 0\}$; see Definition \ref{defn:flatpoints}.
Here is the main result of this section.

\begin{theorem}\label{thm:dimsingularset}
Let $u$ be an almost-minimizer of $J^+$ in $\Omega \subset \mathbb R^n$,
assume that $q_+$ is H\"older continuous, bounded, and bounded below, and let $s > n-k^*$. 
Then $\mathcal H^s(\Gamma^+ \sm \mathcal R) = 0$. 
\end{theorem}

We now have an analogue of Theorem 4.1 in \cite{W}, which says that if $n \leq k^*$ then the singular set
$\Gamma^+(u) \sm \mathcal R$ consists of at most isolated points.

\begin{lemma}\label{lessthankstarisolatedpoints}
Let $u$ be an almost minimizer of $J^+$ in $\Omega \subset \mathbb R^n$ and assume $n \leq k^*$. 
Then $\Gamma^+(u) \sm \mathcal R$ is composed of isolated points.
\end{lemma}

\begin{proof}
Assume that there is a sequence of points $x_k \in \Gamma^+(u) \sm \mathcal R$ 
such that $x_k \rightarrow x_0 \in \Gamma^+(u)$. Set 
$\rho_k = |x_k - x_0|$ and define a blow-up sequence by 
$u_{k, x_0}(x) = \rho_k^{-1} u(\rho_kx+ x_0)$. 
Passing to a subsequence we may assume that $u_{k,x_0}$ converges to $u_0$
(see Lemma~\ref{blowupsexistandarehomogenous1})
and by Lemma~\ref{blowupsexistandarehomogenous2} $u_0$ is a
a homogeneous global minimizer, with $\lambda_+ = q_+(x_0)$ in \eqref{eqn2.4}. 
Further passing to a subsequence we may assume that 
$\frac{x_k-x_0}{\rho_k} \rightarrow y_0 \in \d B(0,1)$. 
Suppose that $\partial \{u_0 > 0\}$ is non-singular away from the origin (and in particular at $y_0$).
By Proposition \ref{flatpointsregularpoints},  this also means that $y_0 \in \mathcal R$
(with respect to $u_0$), and the definition \eqref{a6.12} of $\mathcal R$ and \eqref{a6.4} yield that
for each $\varepsilon > 0$ we can find $r_0 > 0$ such that 
$$
W(u_0, y_0, r) - q_+^2(x_0) \,\frac{\omega_n}{2} < \varepsilon/4
\ \text{ for } r < r_0.
$$
By the proof of Lemma~\ref{blowupsexistandarehomogenous2} (slightly modified because
now we take a function $W$ centered at a different point), we get that 
$$
W\left(u_k, \frac{x_k-x_0}{\rho_k}, r\right) - q_+^2(x_0) \,\frac{\omega_n}{2} < \varepsilon/2,
$$
where in the definition \eqref{eqn:mono2a} of $W(u_k, \frac{x_k-x_0}{\rho_k}, r)$ 
we use the constants $q_+(x_k)$ instead of $q_+(x_0)$, but this does not matter because $q_+$
is H\"older continuous and $x_k$ tends to $x_0$.

Then by almost-monotonicity (Proposition \ref{prop:monotonicitytwo}),
\begin{equation}\label{mxkismall}
\begin{aligned} 
W(u, x_k, 0) - q_+^2(x_k) \, \frac{\omega_n}{2} 
&\leq W(u, x_k,r \rho_k) + C(r\rho_k)^{\alpha} - q_+^2(x_k) \, \frac{\omega_n}{2}\\
& = W(u_k, \frac{x_k -x_0}{\rho_k}, r) - q_+^2(x_0) \, \frac{\omega_n}{2} 
+ C(r\rho_k)^\alpha  + C |x_0-x_k|^\alpha 
\\
&< \varepsilon/2 + C(r\rho_k)^\alpha + C |x_0-x_k|^\alpha  < \varepsilon,
\end{aligned}
\end{equation}
for $k$ large enough. But this implies, by Proposition \ref{Missmallestatflatpoints}, that 
$x_k \in \mathcal R$, a contradiction. 

Thus we can find $\wt x \in \partial \{u_0 > 0\}$, $\wt x \neq 0$,  
such that $\{u_0 > 0\}$ is not flat at $\wt{x}$. Consider $u_{00}$, any 
blowup limit of $u_0$ at $\wt{x}$. By Lemma 3.1 in \cite{W}, $u_{00}$ is constant in the direction 
of $\wt{x}$ and the whole line $t\wt{x}$ consists of singular points. 
Lemma 3.2 in \cite{W} tells us that $\widehat{u}$, the pushforward of $u$ under the projection 
map $\mathbb R^n \mapsto \wt{x}^\perp$, is a global minimizer with a singularity at $0$. 
However, $\dim\wt{x}^\perp < k^*$, which contradicts the definition of $k^*$. 
Ergo, our sequence $\{x_k \}$ in $\Gamma^+ \sm \mathcal R$ 
could not have an accumulation point in $\Gamma^+$. 
\end{proof}

The following is a version of Lemma 4.2 in \cite{W}. 

\begin{lemma}\label{covertheblowup}
Let $u$ be an almost minimizer of $J^+$, $x_0 \in \Gamma^+(u)$, and let 
$$
u_0 = \lim_{k\rightarrow \infty} \frac{u(x_0+ \rho_kx)}{\rho_k q_+(x_0)} 
= :  \lim_{k\rightarrow \infty} u_k(x)
$$
be any (normalized) blow-up limit of $u$ at $x_0$.
Call $\Sigma_k$ the singular part of $\Gamma^+(u_k)$ and $\Sigma_0$ 
the singular part of $\Gamma^+(u_0)$. 
Then for every compact set $K \subset \R^n$ and open set $U \subset \R^n$ such that 
$K \cap \Sigma_0 \subset U$, there is a $k_0 < \infty$ such that 
$\Sigma_k \cap K \subset U$ for $k\geq k_0$. 
\end{lemma}

\begin{proof}
Recall that the singular set of $u$ is $\Gamma^+(u) \sm \mathcal R$, and similarly for $u_k$ and $u_0$.
Assume, in order to obtain a contradiction, that there are $y_k \in (\Sigma_k \cap K)\setminus U$, which, passing to a subsequence, 
we may assume converge a limit $y_0$. Notice that $y_0 \in \Gamma^+(u_0) \cap K\setminus U$
because this set is closed. 

By assumption $y_0$ is a flat point of $\partial \{u_0 > 0\}$, so there exists $r_0 > 0$
such that if $r < r_0$ then 
$$
W(u_0, y_0, r) - \frac{\omega_n}{2} < \varepsilon/4,
$$ 
where $\varepsilon > 0$ is as in Proposition \ref{Missmallestatflatpoints}. 
A limiting argument gives us that for $k$ large enough (which depends on $r$), 
$W(u_k, y_k, r) - \frac{\omega_n}{2} < \varepsilon/2$. 
By almost-monotonicity (Proposition \ref{prop:monotonicitytwo}) this implies that 
$W(u_k, y_k, 0) - \frac{\omega_n}{2} < \varepsilon/2 + Cr^\alpha$. If $r$ is small enough, 
so that $\varepsilon/2 + Cr^\alpha< \varepsilon$, Proposition \ref{Missmallestatflatpoints} implies that 
$y_k$ is a flat point of $u_k$ for large enough $k$. This is the desired contradiction.  

\end{proof}

The proof of Theorem \ref{thm:dimsingularset} will now follow 
exactly as in  \cite{W}. Let us simply recall (without proofs) the sequence of results that gives 
our dimension estimate. 

The following result follows from Lemma \ref{covertheblowup} and a covering argument. 

\begin{lemma}\label{contentpassestoblowups}
Keeping the notation from Lemma \ref{covertheblowup}, for any $0\leq m < \infty$, the estimate $\mathcal H^m_\infty(\Sigma_0\cap K) \geq \limsup_{k\rightarrow \infty} \mathcal H^m_\infty(\Sigma_k\cap K)$ holds.
\end{lemma}

We can then immediately deduce the following. 

\begin{lemma}\label{measurepassestoblowups}
Again let $u$ be an almost-minimizer for $J^+$ in dimension $n$ and
suppose that $\mathcal H^m(\Sigma \cap D) > 0$ for some open set $D$ 
(where $\Sigma$ is the singular set of $\Gamma^+(u)$).
Then there exists $x_0 \in D$ and a blowup limit, $u_0$, of $u$ at the point $x_0$,
such that $\mathcal H^m(\Sigma_0\cap \overline{B(0,1)}) > 0$, where $\Sigma_0$  
is the singular set of $\partial \{u_0 > 0\}$. 
\end{lemma}

Finally Theorem \ref{thm:dimsingularset} follows.

\section{ A quantified version of the free boundary regularity theorem of Alt and Caffarelli}\label{appendix}

In this section we complete the proof of Theorem \ref{main-reg-theo} by showing \emph{flatness}  improvement estimates on 
weak solutions. A key feature is that at this point, we have transformed
our initial problem on almost-minimizers into a problem that only concerns harmonic functions, and more
specifically the weak minimizers defined below. Moreover, the proofs below follow the same
scheme as arguments of \cite{AC} and then \cite{KT2}. 
Because of this, we are able to go more rapidly over estimates that are very close to those of 
\cite{AC} and \cite{KT2}, and focus on those that are different.

To emphasize the similarities between the properties of $h_{x_0,\rho}$ obtained in 
Lemma \ref{uflat-hflat} and those described in Definition 7.1 in \cite{AC} or those studied in \cite{KT2},
we isolate some of the characteristics of $h_{x_0,\rho}$ for $x_0\in\partial\{h_{x_0,\rho}>0\}$ 
and $\rho>0$ as in Lemma \ref{uflat-hflat}. 
Set $v= h_{x_0,\rho}$, with $\rho = r^{\frac{1}{1+\tilde\gamma}}$, and set
$\tau=r^{\frac{\mu}{1+\tilde\gamma}}$, as we did in \eqref{a7.22}, but also replace 
$2\sigma$ by $\sigma$ and $e_0$ by $-e_0$ (that is, the ``positive" direction is where the zero set lies and the ``negative" direction is where the positivity set lies).
$v\in C(\overline{B(x_0,4r)})$, is harmonic on $\{v>0\}\cap B(x_0,4r)$, and for $x\in B(x_0,r)$
\begin{eqnarray}\label{a-h-properties}
\left\{ \begin{array}{ll}
v(x)=0 &\hbox{ if }  \langle x-x_0,e_0\rangle\ge \sigma r\\[2mm]
v(x)\ge -q_+(x_0)[\langle x-x_0,e_0\rangle +\sigma r] &\hbox{ if }  \langle x-x_0,e_0\rangle\le -\sigma r\\[2mm]
|\nabla v(x)|\le  q_+(x_0)(1+ \tau).\end{array}\right.
\end{eqnarray}
Moreover for $\mathcal H^{n-1}$-a.e $z\in\partial U 
\cap B(x_0,r)$
\begin{equation}\label{a-normal-derivative}
\frac{\partial^+ v}{\partial \nu}(z)\ge q_+(x_0)(1-\tau), 
\end{equation}
where, by Proposition \ref{rn-derivative},
$\Delta v=  \frac{\partial v}{\partial\nu}\, d\mathcal{H}^{n-1}\res\partial^\ast\{v>0\}$ 
in the sense that
\begin{equation}\label{eqn-a.1}
-\int \langle\nabla v, \nabla \zeta\rangle 
=\int_{\partial^\ast\{v>0\}} \zeta \, \frac{\partial^+ v}{\partial\nu}
\, d\mathcal{H}^{n-1} 
\hbox{  for all  }\zeta\in C^1_c(B(x_0,r)).
\end{equation}
In the present situation we do not need to worry about the regularity of $\partial\{v>0\}$,
because it is equal to $\partial U = \Gamma^+(u)$ near the support of $\zeta$, and we
could have integrated on $\partial \{v>0\}$ rather than the reduced boundary $\partial^\ast\{v>0\}$
because the difference has vanishing measure.

Furthermore Corollary \ref{cor3.2} ensures that there exist $0<c_{{min}}\le C_{{max}}<\infty$ such 
that for all $z\in \partial\{v>0\}\cap B(x_0,3r)$ and $0<s\le r$,
\begin{equation}\label{eqn-a.2}
c_{{min}}\le 
\frac{1}{s}\fint_{\partial B(z,s)} v\, d\mathcal{H}^{n-1}
\le C_{{max}}.
\end{equation}

By analogy with definitions 5.1 and 7.1 in \cite{AC} we define weak solutions and flat free boundary points.

\begin{defn}\label{weak-ap-sol}
A non-negative function $v$ is a weak solution in $B(x_0,4r)$ if 
\begin{enumerate}
\item $v\in C(\overline{B(x_0,4r)})$ is harmonic on $\{v>0\}\cap B(x_0,4r)$.
\item There exist $0<c_{{min}}\le C_{{max}}<\infty$ such that \eqref{eqn-a.2} holds
for all $z\in \partial \{v>0\}\cap B(x_0,3r)$ and $0<s\le r$. 
\item $\{ v > 0 \}$ is a set of finite perimeter in $B(x_0,\rho)$ for $0 < \rho < 4r$,
and \eqref{eqn-a.1} holds.
\end{enumerate}
\end{defn}

Here we added the condition on the finite perimeter so that we can easily integrate by parts
and talk about the reduced boundary. Similarly, we can take $\eqref{eqn-a.1}$ as a definition
of $\frac{\partial^+ v}{\partial\nu}$; we do not need to know that it can actually be computed
from $v$ as a derivative in the normal direction. But anyway, both things are true for our main
example $v = h_{x_0,\rho}$ above.

The weak solution $v$ comes with two constants $c_{{min}}$ and $C_{{max}}$,
which in the previous sections were estimated from properties of $q_+$, but observe here
that $q_+$ does not show up in the definition of a weak solution.

\begin{defn}\label{a-flat-point}
Let $\sigma_+, \sigma_-\in (0,1]$, $\tau\in (0,1/2)$. We say that 
\begin{equation}\label{a-flatness-def}
v\in F(\sigma_+,\sigma_-;\tau) \hbox{  in  } B(x_0,r)  \hbox{ in the direction } e_0 
\end{equation}
when
\begin{enumerate}
\item $v$ is a weak solution in $B(x_0,4r)$
\item $x_0\in \partial \{v>0\}$ and, for $x\in B(x_0,r)$,
\begin{eqnarray}\label{a-flatness-1}
\left\{ \begin{array}{ll}
v(x)=0 &\hbox{ if }  \langle x-x_0,e_0\rangle\ge \sigma_+ r\\[2mm]
v(x)\ge -q_+(x_0)[\langle x-x_0,e_0\rangle +\sigma_- r] 
&\hbox{ if }  \langle x-x_0,e_0\rangle\le -\sigma_- r.\end{array} \right.
\end{eqnarray}
\item 
\begin{equation} \label{a-flatness-2a}
\sup_{B(x_0,r)}|\nabla v(x)|\le q_+(x_0)(1+ \tau).
\end{equation}
and
\begin{equation}\label{a-flatness-2}
k(z)=\frac{\partial^+ v}{\partial \nu}(z)\ge q_+(x_0)(1-\tau)
\ \text{ for $\mathcal H^{n-1}$-a.e } z\in\partial^\ast \{v>0\}\cap B(x_0,r). 
\end{equation} 
\end{enumerate}
\end{defn} 

A few comments on this definition may help the reader 
get more familiar with the notion. The definition only depends on $q_+$
through the number $q_+(x_0)$, and incidentally this number could be estimated
from $v$, with a relative error of roughly $2\tau$, by comparing 
\eqref{a-flatness-2a} and \eqref{a-flatness-2}. So the variations of $q_+$
do not matter: we just use $q_+(x_0)$ as a normalization.

Some of our constraints (such as $x_0 \in \partial \{v>0\}$)
will concern generic points of $\partial \{v>0\}$, while others concern points
of the reduced boundary $\partial^\ast \{v>0\}$. We'll try to distinguish 
between the two, but when $v$ comes from an almost minimizer $u$ as in the sections above,
the two sets are almost the same because $\partial^\ast \{v>0\} \subset \partial \{v>0\}$
(as always) and
\begin{equation} \label{a8.9}
\H^{n-1}(B(x_0,3r) \cap \partial \{v>0\} \sm \partial^\ast \{v>0\}) = 0
\end{equation}
by the local uniform rectifiability properties of $\Gamma^+(u)$ that were proved above.
Possibly there is a simple argument that says that this stays true for any weak solution $v$, 
but we did not find it, so the reader that wants to feel safe could simply assume that \eqref{a8.9}
holds.

The main difference between Definition 7.1 in \cite{AC} and Definition \ref{a-flat-point} 
concerns the behavior of the derivative of $v$ at the boundary. 
A detailed analysis of the work in \cite{AC} reveals that condition \eqref{a-flatness-2}
(with the normal derivative for which we have \eqref{eqn-a.1})  
is enough to obtain some degree of improvement. Definition \ref{a-flat-point} can be 
understood as a perturbation of the case studied in \cite{KT2}, where  
the authors considered the case when $\tau=0$. Given the extent to which the arguments 
presented below are related to those in \cite{AC} and \cite{KT2} we only state the main results and 
describe in detail the proofs in which the condition concerning the behavior of the of 
the derivative of $v$ at the boundary plays a role.

\ms
The following preliminary technical lemma is closely related to Lemma 4.10
in \cite{AC}(see also Lemma 0.3 in \cite{KT2}).

\begin{lemma}\label{apb-lem0}
Let $v$ is a weak solution in $B(x_0,4r)$. Suppose that \eqref{a-flatness-2a} and
\eqref{a-flatness-2} hold. 
Let $z\in\partial\{v>0\}\cap B(x_0,r)$ 
and assume that there exists a ball
$B\subset\{v=0\}$ so that $z\in\partial B$. Then
\begin{equation}\label{apb-eqn501}
\liminf_{\mathop{x\to z}\limits_{\scriptstyle{x\in\{v>0\}}}} \frac{v(x)}{d(x,B)}
\ge q_+(x_0)(1-\tau).
\end{equation}
\end{lemma}

\begin{proof} Without loss of generality assume that $q_+(x_0)=1$.
Let $l=\mathop{\liminf}\limits_{\mathop{x\to z}\limits_{\scriptstyle{x\in\{v>0\}}}}
\frac{v(x)}{d(x,B)}$. 
Choose a sequence $\{y_k\}_{k\ge 1}$ in $\{v>0\}$ that tends to $z$ and such that
$\frac{v(y_k)}{d(y_k,B)}$ tends to $l$.
Set $d_k=d(y_k,B)$ and choose $x_k\in\p B$ so that $|y_k-x_k|=d_k$.
Set $v_k(x)=d_k^{-1} v(d_kx+x_k)$ for $x\in B(0,2/d_k)$ 
and $z_k= d_k^{-1}(y_k-x_k)$. 
Without loss of generality we may assume that $z_k\to e$ as $k\to\infty$, with $|e|=1$, and 
that $v_k$ converges to some limit $v_\infty$ in a suitable sense.
We shall not get into details here, because the argument is the same as in \cite{KT2},
but one gets that $v_\infty(e)=l$ (using the uniform convergence of the $v_k$)
and $v_\infty(y)=l \langle y,e\rangle_+$ 
for $y\in B(0,1)$ (this time, using a detailed analysis of the blow-up speed of $\{v>0\}$ 
as well as the maximum principle).

Set $h_k(x) = \frac{\d^+ v}{\d \nu}(d_kx+x_k)$; 
then for $\zeta\in C^\infty_c(B(0,1))$, $\zeta\ge 0$,
\begin{equation}\label{apb-eqn3.35}
\int_{\partial\{v_k>0\}}\zeta h_k d\H^{n-1}= - \int_{\R^n}\nabla v_k\cdot\nabla\zeta
\mathop{\longrightarrow}\limits_{k\to\infty} - \int_{\R^n}\nabla
v_\infty\cdot\nabla\zeta = \int_{\{\langle y,e\rangle=0\}} l\zeta d\H^{n-1}
\end{equation}
by \eqref{eqn-a.1}, because the $\nabla v_k$ happen to converge weakly to $\nabla v_\infty$,
because $v_\infty(y)=l \langle y,e\rangle_+$, and by the reverse integration by parts.
Thus
\begin{equation}\label{apb-eqn3.36}
\lim_{k\to\infty}\int_{\partial\{v_k>0\}}\zeta h_kd\H^{n-1} 
= \int_{\{\langle y, e\rangle=0\}}l\zeta d\H^{n-1}.
\end{equation}

On the other hand since $\zeta \geq 0$ and by the divergence theorem
(recall that $\{ v > 0 \}$ is locally a set of finite perimeter), 
\begin{equation}\label{apb-eqn3.36A}
\int_{\partial\{v_k>0\}} \zeta \, d\H^{n-1} 
\ge \int_{\partial\{v_k>0\}} \zeta \langle e, \nu_k \rangle  \, d\H^{n-1} 
= \int_{\{v_k>0\}}\div(\zeta e).
\end{equation}
Since
\begin{equation}\label{apb-eqn3.36B}
\int_{\{v_k>0\}}\div(\zeta e) \mathop{\longrightarrow}\limits_{k\to\infty}
\int_{\{v_\infty>0\}} \div(\zeta e) = \int_{\partial\{v_\infty>0\}}\zeta
d\H^{n-1} = \int_{\langle y,e\rangle=0}\zeta d\H^{n-1},
\end{equation}
then by \eqref{apb-eqn3.36A} and \eqref{apb-eqn3.36B} 
\begin{equation}\label{apb-eqn505}
\lim_{k\to\infty} \int_{\partial\{v_k>0\}}\zeta d\H^{n-1}\ge \int_{\{\langle y,
e\rangle=0\}} \zeta d\H^{n-1}.
\end{equation}
Since by \eqref{a-flatness-2}  $\frac{\d^+ v}{\d \nu} \geq (1-\tau) q_+(x_0) = 1-\tau$  
for $\H^{n-1}$-a.e. point of $\partial^\ast \{v>0\} \cap B(x_0,r)$, 
\begin{equation} \label{apb-eqn506}
\lim_{k\to\infty} \int_{\p\{v_k>0\}}h_k\zeta d\H^{n-1}
\ge  (1-\tau)\limsup_{k\to\infty}
\int_{\p\{v_k>0\}}\zeta d\H^{n-1}
\end{equation}
and hence, by \eqref{apb-eqn3.36} and \eqref{apb-eqn505},
\begin{equation} \label{apb-eqn506b}
l\int_{\{\langle y,e\rangle=0\}}\zeta d\H^{n-1} 
\ge  (1-\tau)\int_{\{\langle y,e\rangle=0\}}\zeta d\H^{n-1}
\end{equation}
for any $\zeta\in C^\infty_c(B(1,0))$ such that $\zeta\ge 0$. Therefore
(\ref{apb-eqn506}) yields
\begin{equation}\label{apb-eqn507}
l\ge 1-\tau,
\end{equation}
which is the same as \eqref{apb-eqn501}

\end{proof}

\ms
The next two lemmata will play an important role in the proof. 
They are quite close to Lemmata 7.2 and 7.9 in \cite{AC} or Lemmata 0.4 and 0.5 in \cite{KT2},
but nonetheless we shall sketch their proof for the reader's convenience.

\begin{lemma}\label{apb-lem1}
Suppose $v$ is a weak solution in $B(x_0,4r)$. 
There exists $\sigma_n > 0$ (that depends only on $n$), such that if 
$0 < \sigma \leq \sigma_n$, $0<\tau \le \sigma$, $e\in\SS^n$, and
$v\in F(\sigma, 1, \tau)$ in $B(x_0, r)$ in the direction $e$,
then $v\in F(2\sigma, C\sigma;\tau)$ in $B(x_0, \frac{r}{2})$ in the direction $e$.
\end{lemma}

\ms
\begin{lemma}\label{apb-lem2}
Suppose $v$ is a weak solution in $B(x_0,4r)$.
Given $\theta\in(0,1)$ there exist $\sigma_{n,\theta}>0$ and
$\eta = \eta_{n,\theta} \in(0,1)$ so that if $\sigma\le \sigma_{n,\theta}$, 
$\tau\sigma^{-2}\le \sigma_{n,\theta}$, $x_0 \in\partial\{v>0\}$, and $v\in F(\sigma,\sigma;\tau)$ 
in $B(x_0,r)$ in the direction $e_{x_0,r}$, then $v\in F(\theta\sigma,1;\tau)$ in $B(x_0; \eta r)$
in some direction $e_{x_0,\eta r}$ such that
\begin{equation}\label{apb-eqn7}
|e_{x_0,r}-e_{x_0,\eta r}|\le C\sigma.
\end{equation}
\end{lemma}

In both lemmata the constants $\sigma_n$ and $C$ depend only on $n$,
$c_{min}$, $C_{max}$, $|| q_+ ||_\infty$, and $c_0 > 0$ such that $q_+ \geq c_0$. 
Probably there are strong relations between these constants, but we decided not to investigate.
In Lemma \ref{apb-lem2}, $\sigma_{n,\theta}$ and $\eta_{n,\theta}$ depend on these constants, 
plus $\theta$.

Here is a consequence of Lemma \ref{apb-lem1} and Lemma \ref{apb-lem2}, which we shall establish 
before we discuss the proof of the lemmata. 

\begin{corollary}\label{main-cor}
Given $\theta\in (0,1)$  there exist $\sigma_{n,\theta}>0$ and
$\eta = \eta_{n,\theta}\in(0,1)$ so that if $\sigma\le \sigma_{n,\theta}$, 
$\tau\sigma^{-2}\le \sigma_{n,\theta}$, $x_0\in\partial\{v>0\}$, and
$v\in F(\sigma,\sigma;\tau)$ in $B(x_0,r)$ in the direction $e_{x_0,r}$, 
then $v\in F(\theta\sigma,\theta\sigma;\tau)$ in $B(x_0; \eta r)$ in some 
direction $e_{x_0,\eta r}$ such that 
\begin{equation}\label{apb-eqn7A}
|e_{x_0,r}-e_{x_0,\eta r}|\le C\sigma.
\end{equation}
\end{corollary}

\begin{proof}[Proof of Corollary \ref{main-cor}.]  
Apply Lemma \ref{apb-lem2} to $\theta'=\theta/C$, where $C\ge 2$ is an in 
Lemma~\ref{apb-lem1}. 
Then there exist $\sigma_{n,\theta'}>0$ and $\eta' = \eta_{n,\theta'}\in(0,1)$ so that if
$\sigma\le \sigma_{n,\theta'}$, $\tau\sigma^{-2}\le \sigma_{n,\theta'}$ and
$v\in F(\sigma,\sigma;\tau)$ in $B(x_0,r)$ in the direction $e_{x_0,r}$ then by Lemma \ref{apb-lem2} $v\in F(\theta'\sigma,1;\tau)$ in $B(x_0; \eta' r)$.
By Lemma~\ref{apb-lem1}, $v \in F(2\theta'\sigma, C\theta'\sigma;\tau)$ in 
$B(x_0, \frac{\eta' r}{2})$. Letting $\eta=\frac{\eta'}{2}$ we have
$v\in F(\theta\sigma,\theta\sigma;\tau)$ in $B(x_0; \eta r)$, and \eqref{apb-eqn7A} holds.  
\end{proof}

\ms
As mentioned earlier the proofs of Lemmas \ref{apb-lem1} and \ref{apb-lem2}
are very similar to those presented in \cite{AC} (Section 7) (see also \cite{KT2}),
so we will insist on differences and sometimes skip details.  
\medskip

\begin{proof}[Proof of Lemma \ref{apb-lem1}.] 
Without loss of generality we
may assume that $x_0=0\in\partial\{v>0\}$, $q_+(x_0)=1$, $r=1$ and $e=e_{n}$. By hypothesis
$v\in F(\sigma, 1;\tau)$ in $B_1=B(0,1)$ in the direction $e_{n}$,
so $\mathop{\sup}\limits_{B_1}|\nabla v|\le 1 +\tau$, and $k(q)\ge 1-\tau$ for $\H^{n-1}$ a.e.
$q\in\partial^\ast\{v>0\}$; this 
in particular implies that for $\varphi\in C^\infty_c(\R^n)$ such that $\varphi\ge 0$,
\begin{equation}\label{apb-eqn26}
- \int \nabla v \cdot \nabla \varphi 
\ge (1-\tau)\int_{\partial^\ast\{v>0\}} \varphi d\H^{n-1}. 
\end{equation}
Let $\eta(y)=\exp\big(\frac{-9|y|^2}{1-9|y|^2}\big)$ for
$|y|<\frac{1}{3}$ and $\eta(y)=0$ otherwise. Choose $s_0>0$ to be the
maximum $s$ so that
\begin{equation}\label{apb-eqn27}
B_1\cap\{v>0\}\subset D=\{x\in B_1:x_{n}<2\sigma-s\eta(\overline x)\},
\end{equation}
where $x=(\overline x, x_{n})$, with $\overline x\in\R^{n-1}\times\{0\}$.
Since $0 = x_0 \in\partial\{v>0\}$ and $\eta(0) = 1$,
then $0 \leq 2\sigma-s_0$ and $s_0 \leq 2\sigma$.
Since $\sigma \leq \sigma_n$ that can be chosen as small as we want,
both $\sigma$ and $s_0$ are very small.

By the maximality of $s_0$, we can find $z\in\p D\cap \partial\{v>0\}\cap B_1$. Furthermore, 
 $z_n \leq \sigma$ (because $v\in F(\sigma,1;\tau)$ in $B_1$),
which implies that that $\eta(\overline z) \neq 0$ and hence, $z\in B(0,1/3)$.

Recall that $s_0 \leq 2\sigma \leq 2\sigma_n$, which we can take small; thus 
$\d D \cap B_1$ is quite smooth and almost horizontal, and we can find a ball
$B \subset D^c$, tangent to $\d D$ at $z$, and with a radius at least $C_n/\sigma_n$
(which is as large as we want).

Consider the function $V$ defined by 
\[
\left\{
\begin{array}{ll}
\Delta V=0 &\mbox{in }D \\
V=0    &\mbox{on }\p D\cap B_1 \\
V=(1+\tau)(2\s-x_{n}) &\mbox{on }\p D\bs B_1.
\end{array}\right.
\]

For the following computations, we refer to \cite{AC} or \cite{KT2}
for some of the details.
By the maximum principle $V>0$ in $D$ and 
\begin{equation}\label{apb-eqn28}
v\le V\mbox{ in }D,
\end{equation}
in fact $v\le V$ on $\p D$ (by \eqref{a-flatness-1} and since $v\in F(\s, 1;\tau)$ in $B_1$) 
and $v$ is subharmonic. From (\ref{apb-eqn28}) we deduce that 
\begin{equation}\label{apb-eqn29}
\limsup_{\mathop{x\to z}\limits_{\scriptstyle{x\in\{v>0\}}}} \frac{v(x)}{|x-z|}
\le \limsup_{\mathop{x\to z}\limits_{\scriptstyle{x\in\{v>0\}}}}
\frac{v(x)}{d(x,B)} \le \frac{\p V}{\p n}(z),
\end{equation}
where  $\frac{\p V}{\p n}=\langle \nabla V,\vecn\rangle$ and 
$\vecn$ denotes the outward unit normal vector to $\p D$. 

For $x\in D$ define $F(x)=(1+\tau)(2\s-x_{n})-V(x)$;
then $F$ is a harmonic function on $D$, $F(x) = (1+\tau)(2\s-x_{n})$
on $\d D \cap B_1$, and $F = 0$ on $\d D \sm B_1$.

Recall that $\d D \cap B_1 = \big\{ (\overline x,x_n) \in B_1 \, ; \, 
x_n = 2\sigma - s_0 \eta(\overline x) \big\}$. Thus if we set
$G(\overline x,x_n) = (1+\tau) s_0\eta(\overline x)$,
we see that $F(\overline x,x_n) = (1+\tau)(2\s-[2\sigma - s_0 \eta(\overline x)])
= G(\overline x,x_n)$ on $\d D \cap B_1$.
In fact, $F = G$ on the whole $\d D$, because on $\d D \sm B_1$,
$\eta(\overline x) = 0$ (as $|\overline x| \geq 1/3$). Thus $F-G$
vanishes on $\d D$, and on $D$ its Laplacian is
$\Delta(F-G) = - \Delta G = - (1+\tau) s_0 \Delta[\eta(\overline x)]$,
which is smooth. By \cite[Lemma 6.5]{GT} (with possibly a minor adaptation 
because $D$ has corners far from $z$), 
\begin{equation} \label{a8.25}
|\nabla (F-G)| \leq C (1+\tau) s_0 ||\Delta[\eta(\overline x)]||_\infty
+ C ||F-G||_\infty
\end{equation}
in, say, $B(z,10^{-1})$. Now $s_0 \leq 2 \sigma$, so 
$||\Delta[\eta(\overline x)]||_\infty +  ||G||_\infty\leq C\sigma$,
and $||F||_\infty \leq C \sigma$ too, by the maximum principle and
because $F = G$ on $\d D$. Finally $||\nabla G||_\infty \leq C s_0 \leq C \sigma$
too, and \eqref{a8.25} implies that $|\nabla F(x)| \leq C \sigma$ near $z$.
Therefore, since $V(x) = (1+\tau)(2\s-x_{n}) - F(x)$ and $\tau\le \sigma$,
\begin{equation}\label{apb-eqn30}
-\frac{\p V}{\p x_{n}}(z)=1+\tau+\frac{\p F}{\p x_{n}}(z)\le 1+C\s
\end{equation}
and, by (\ref{apb-eqn30}) 
\begin{eqnarray}\label{apb-eqn31}
-\frac{\p V}{\p n}(z) & = & -\langle\nabla V(z),\vecn\rangle = -\langle \nabla
V(z), \vecn-e_{n}\rangle-\frac{\p V}{\p x_{n}} \\
& \le & 1+C\s + |\nabla V(z)|\,|\vecn(z)-e_{n}| \nonumber \\
& \le & 1+C\s + (1+C\s)|\vecn(z)-e_{n}|. \nonumber
\end{eqnarray}
Recall that $\vecn(z)=\left(\frac{-sD\eta(\overline
z)}{\sqrt{1+s^2|D\eta(\overline z)|^2}},
\frac{1}{\sqrt{1+s^2|D\eta(\overline z)|^2}}\right)$, and so $|\vecn(z)-e_{n}|
\le C\s$. Combining (\ref{apb-eqn29}) and (\ref{apb-eqn31}) we obtain that
\begin{equation}\label{apb-eqn32}
l :=\limsup_{\mathop{x\to z}\limits_{\scriptstyle{x\in\{v>0\}}}}
\frac{v(x)}{d(x,B)}\le 1+C\s.
\end{equation}
Lemma \ref{apb-lem0} ensures that
\begin{equation}\label{apb-eqn33}
1-{\s}\le1-\tau\le l \le 1+C\s.
\end{equation}

Our goal now is to estimate $v$ from below by the linear function, $-x_{n}$, with an error 
on the order of $\s$. Let $\xi\in\p B\left(0,\frac{3}{4}\right)
\cap\left\{x_{n}<-\frac{1}{2}\right\}$. Consider the solution $\omega_\xi$ of 
\begin{equation}\label{apb-eqn34}
\left\{
\begin{array}{rcll}
\Delta\omega_\xi & = & 0 & \mbox{in }D\bs B\left(\xi,\frac{1}{16}\right) \\
\omega_\xi & = & 0 & \mbox{on }\p D \\
\omega_\xi & = & -x_{n} & \mbox{on }\p B\left(\xi,\frac{1}{16}\right).
\end{array}\right.
\end{equation}
The Hopf boundary point lemma ensures that
\begin{equation}\label{apb-eqn4.26}
-\frac{\p\omega_\xi}{\p n}(z)\ge c(n)>0. 
\end{equation}
Let $K > 0$ be large (to be chosen later) and assume that for every 
$x\in\overline B\left(\xi, \frac{1}{16}\right)$
\begin{equation}\label{apb-eqn4.27}
v(x)\le V(x)+ K\s x_{n}.
\end{equation}
The maximum principle would then imply that
\begin{equation}\label{apb-eqn4.28}
v(x)\le V(x)-K\s\omega_\xi(x)\quad\mbox{in}\quad D\bs
B\left(\xi,\frac{1}{16}\right).
\end{equation}
Thus combining (\ref{apb-eqn30}), (\ref{apb-eqn33}), (\ref{apb-eqn4.26}),
(\ref{apb-eqn4.27}) and (\ref{apb-eqn4.28}) we would conclude that
\begin{eqnarray} \label{apb-eqn4.29}
1-\sigma &\leq& l 
= \limsup_{\mathop{x\to z}\limits_{\scriptstyle{x\in\{v>0\}}}} \frac{v(x)}{d(x,B)}
\leq \limsup_{\mathop{x\to z}\limits_{\scriptstyle{x\in\{v>0\}}}} 
\frac{V(x)-K\s\omega_\xi(x)}{d(x,B)}
\nonumber\\
&\leq& \frac{\p V}{\p n}(z)-K\s\frac{\p\omega_\xi}{\p n}(z)
\le 1+C\s-c(n) K\s 
\end{eqnarray}
which is a contradiction for $K > \frac{C+1}{c(n)}$.

Thus we can find $x_\xi\in B\left(\xi,\frac{1}{16}\right)$ such that
\begin{equation}\label{apb-eqn4.30}
v(x_\xi)\ge V(x_\xi)+K\s (x_{\xi})_n
\end{equation}
for some large, fixed, $K$.

We want to show that $v > 0$ on $B\big(x_\xi,\frac{1}{8}\big) \subset B(0,1)$.
Let $x\in B\big(x_\xi,\frac{1}{8}\big)$ be given. 
By definitions and the maximum principle,
\begin{equation} \label{a8.34}
V(x)\ge -x_{n} \ \text{ for } x\in D.
\end{equation}

Then we can estimate, for $\s_n$ small enough, 
\begin{eqnarray}\label{apb-eqn4.31}
v(x) & \ge & v(x_\xi)- |x-x_\xi| \sup_{B\left(x_\xi,\frac{1}{8}\right)} |\nabla v|
\geq v(x_\xi)- \frac{1}{8} \, (1+\tau)
 \nonumber \\
& \ge & V(x_\xi)+K\s(x_\xi)_{n}- \frac{1}{8} \, (1+\tau)
\geq -(x_\xi)_{n}+K\s(x_\xi)_{n}- \frac{1}{8} \, (1+\sigma)
\nonumber \\
& \ge & \frac{7}{16}-\frac{13}{16}K\s- \frac{1}{8} \, (1+\sigma) >0,
\end{eqnarray}
where the inequalities follow by the mean value theorem, the definition of flatness, \eqref{apb-eqn4.30}, \eqref{a8.34}, $\tau \leq \sigma$ and $x_\xi \in B(\xi, 1/8)$ so $ -7/16 >(x_\xi)_n > -13/16$, respectively. 

Since $v(x)>0$ for $x\in\overline B\left(x_\xi, \frac{1}{8}\right)$, 
$v$ is harmonic on $B\left(x_\xi, \frac{1}{8}\right)$ and so is $V-v$. 
Moreover $V-v\ge 0$ on $B\left(x_\xi, \frac{1}{8}\right)\supset B\left(\xi,\frac{1}{16}\right)$
because these sets lie well inside $D$ and by \eqref{apb-eqn28}.
Therefore Harnack's inequality combined with (\ref{apb-eqn4.30}) yields
\begin{equation}\label{apb-eqn4.32}
(V-v)(\xi)  \le  C(n)(V-v)(x_\xi) 
 \le  -CK\s(x_\xi)_{n} 
 \le  C\s, \quad \forall \xi \in \partial B(0, 3/4)\cap \{\xi_n < -1/2\}
\end{equation}
and
\begin{equation}\label{apb-eqn4.33}
v(\xi)  \ge  V(\xi)-C\s \ge  -\xi_{n}-C\s,\quad \forall \xi \in \partial B(0,3/4)\cap \{\xi_n <-1/2\}.
\end{equation}

For $x\in D\cap B\big(0,\frac{1}{2}\big)$, let $\xi_x \in \d B\big(0, \frac34\big) \cap 
\left\{\xi_{n}<-\frac{1}{2}\right\}$ be such that $\overline\xi_x = \overline x$, and
write $x=\xi_x+te_{n}$; then
\begin{equation} \label{apb-eqn4.34}
v(x) =  v(\xi+te_{n})\ge v(\xi_x)-(1+\tau)t \geq -(\xi_{n}+t)-C\s
\end{equation}
by \eqref{a-flatness-2a} and \eqref{apb-eqn4.33}, and since $\tau\le \sigma$.
Since $v\in F(\s,1;\tau)$ in $B_1$ in the direction $e_{n}$, 
\eqref{apb-eqn4.34} ensures that $v\in F(2\s; C\s;\tau)$ in 
$B\big(0,\frac{1}{2}\big)$ in the direction $e_{n}$.

\end{proof}

\ms
\begin{proof}[Proof of Lemma \ref{apb-lem2}.]
We will proceed by contradiction, using a non homogeneous blow-up. 
This argument follows closely the argument in \cite{AC} and \cite{KT2};  
we only include the proofs which are somewhat different than those that already appear in the literature.

It is enough to prove the lemma for $x_0 = 0$ and $r=1$, with varying functions 
$q_+$, although with uniform bounds on $||q_+||_\infty$ and $c_0>0$ such that
$q_+ \geq c_0$. In addition, notice that when we multiply $v$ and $q_+$
by a same positive number, $\lambda v$ is still a weak solution, with $c_{min}$
and $C_{max}$ multiplied by $\lambda$, and $\lambda v \in F(\sigma_+,\sigma_-,\tau)$
implies that $\lambda v \in F(\sigma_+,\sigma_-,\tau)$, in the same
direction, but with $\lambda q_+$. Because of this we just need to prove the
lemma when $q_+(0) = 1$. Notice that $q_+$ only shows up in the statement through
$q_+(0)$, so after this remark (applied with $\lambda = q_+(0)^{-1}$, which
does not upset too much our uniform bounds for $c_{min}$
and $C_{max}$), we will be able to forget about $q_+$ altogether.

Assume that Lemma \ref{apb-lem2} does not hold. 
There exists a $\theta_0\in (0,1)$ such that for any $\eta>0$ (later we specify one), 
there exist non-negative decreasing sequences $\{\s_j\}_j$ and $\{\tau_j\}_j$, 
with $\sigma_j\to 0$ and $\sigma_j^{-2}\tau_j\to 0$, weak solutions $v_j$
in $B(0,4)$, and unit vectors $\nu_j$, so that
\begin{equation}\label{apb-eqn35}
v_j\in F(\s_j,\s_j;\tau_j)\mbox{ in }B(0,1)\mbox{ in the direction } \nu_j
\end{equation}
but we cannot find $\wt \nu_j$ such that \eqref{apb-eqn7} holds (with a constant
$C$ that will be chosen later, but that is independent of $j$) and 
\begin{equation}\label{apb-eqn36}
v_j \in F(\theta_0\s_j, 1;\tau_j)\mbox{ in }B(0,\eta)
\ \text{ in the direction } \wt \nu_j.  
\end{equation}

By rotation invariance of the lemma, we may assume that all the $\nu_j$ are equal
to the last coordinate unit vector $e_n$. 
Let us record some of our assumptions. First, $\Delta v_j=0$ in $\{v_j>0\} \cap B(0,4)$
and (by \eqref{eqn-a.1})
\begin{equation}\label{apb-eqn4.37}
-\int\nabla v_j\cdot \nabla\phi\, dx=\int_{\partial^\ast\{v_j>0\}}\phi k_jd\H^{n-1}
\end{equation}
for $\phi \in C^\infty_c(B(0,1))$ and where $k_j$ is our normal derivative for $v_j$
on $\d \{ v_j > 0 \}$. Also,  \eqref{apb-eqn35} holds with $\nu_j = e_n$ and 
$q_{j,+}(0)=1$ and in particular 
\begin{equation}\label{apb-eqn37}
\sup_{B(0,1)}|\nabla v_j|\le (1+\tau_j) 
\quad\mbox{and}\quad 
k_j\ge (1-\tau_j) 
\quad  \H^{n-1} \mbox{ a.e. in }\partial^\ast\{v_j>0\}.
\end{equation}
We also have that $0 \in \d \{ v > 0 \}$ and \eqref{a-flatness-1}, which says that 
for $x\in B(0,1)$, 
\begin{eqnarray}\label{a8.44}
\left\{ \begin{array}{ll}
v_j(x)=0 &\hbox{ if }  x_n \ge \sigma_j\\[2mm]
v_j(x)\ge - x_n - \sigma_j
&\hbox{ if }  x_n \le -\sigma_j.
\end{array} \right.
\end{eqnarray}
Recall also that $\s_j\to 0$ and $\tau_j\sigma_j^{-2}\to 0$ as $j\to\infty$,
and that we are assuming that \eqref{apb-eqn36} fails for $\wt \nu_j$ close to 
$e_n$ (as in \eqref{apb-eqn7}), and this is what we want to contradict for $j$ large.
The idea is to define sequences of scaled height functions (in the direction
$e_{n}$) corresponding to $\partial\{v_j>0\}$, prove that this sequence converges to a
subharmonic Lipschitz function, and use this information to prove \eqref{apb-eqn36}
for $j$ large. 

Set $B = B(0,1/2)$ and $B' = B\cap [\R^{n-1}\times\{0\}]$.
Define, for $y\in B'$,
\begin{equation}\label{apb-eqn38}
f^+_j(y)=\sup\{h:(y,\s_jh)\in\p\{v_j>0\}\}\le 1, 
\end{equation}
where the last inequality is by $v_j \in F(\sigma_j ,\sigma_j, \tau_j)$, and
\begin{equation}\label{apb-eqn39}
f^-_j(y)=\inf\{h; (y, \s_jh)\in\p\{v_j>0\}\}\ge -1,
\end{equation}
where again we are $\geq -1$ by the assumed flatness. 

This non-homogeneous (so called because the $e_n$ direction is weighted differently) blow-up is the key ingredient of the proof of Alt and Caffarelli's result. From now on the statement of the results, and a good part of the proofs, are  
almost identical to those appearing in \cite{AC} and \cite{KT2}; this will allow us to be a little
more sketchy at times.

\begin{lemma}[Non homogeneous blow up (Lemma 7.3 \cite{AC} or Lemma 0.6 \cite{KT2})]
\label{apb-lem3}
There exists a strictly increasing subsequence $\{ j_k \}$ such that for $y\in B'$,  
\begin{equation}\label{apb-eqn40}
f(y)=\limsup_{\mathop{k\to\infty}\limits_{\scriptstyle{z\to y}}}
f^+_{j_k}(z) = \liminf_{\mathop{k\to\infty}\limits_{\scriptstyle{z\to y}}}
f^-_{j_k}(z).
\end{equation}
\end{lemma}

See \cite{AC} or \cite{KT2} for the proof of this lemma and the next one. 
Also, from now on we assume, without
loss of generality, that we actually started with the subsequence, and write $f_j$
instead of $f_{j_k}$.
In what follows we establish that $f$ is a subharmonic Lipschitz function
bounded above by an affine function. From this we eventually deduce a contradiction 
with the definition of the $v_j$.

\begin{lemma}
{\bf(Corollary 7.4 \cite{AC} or Corollary 0.7) \cite{KT2}} \label{apb-cor1}
The function $f$ that appears in (\ref{apb-eqn40}) is a continuous function
in $B'$, $f(0)=0$; and  $f^+_{j}$ and $f^-_{j}$ 
converge uniformly to $f$ on compact sets of $B'$.
\end{lemma}

\begin{lemma}{\bf(Lemma 7.5 \cite{AC} or Lemma 0.8 \cite{KT2})}\label{apb-lem4}
The function $f$ introduced in Lemma \ref{apb-lem3} is subharmonic in $B'$.
\end{lemma}

\begin{proof}
We proceed by contradiction, i.e. assume that $f$ is not subharmonic in $B'$.
Then there exists $y_0\in B'$ and $\rho>0$ so that $B'(y_0,\rho)\subset B'$
and
\begin{equation}\label{apb-eqnb5.27}
f(y_0) > \fint_{\lower7pt\hbox{$\scriptstyle{\p B'(y_0, \rho)}$}}f(x)dx.
\end{equation}
Set $\delta = f(y_0) - \fint_{\p B'(y_0, \rho)}f(x)dx > 0$ and
pick $\varepsilon_0$ so that $\frac{\delta}{3} \leq \varepsilon_0 \leq \frac{2\delta}{3}$.
Then let $g$ be the solution to the Dirichlet problem
\begin{equation}\label{apb-eqn5.26}
\left\{\begin{array}{rcll}
\Delta g & = & 0 & \hbox{in }B'(y_0, \rho) \\
g & = & f+\epsilon_0 & \hbox{on }\partial B'(y_0,\rho).
\end{array}\right.
\end{equation}
Note that
\begin{equation}\label{apb-eqn5.27}
f<g\hbox{ on }\partial B'(y_0,\rho), 
\end{equation}
and
\begin{equation} \label{apb-eqn5.28}
g(y_0) = \fint_{\partial B'(y_0, \rho)} g = \varepsilon_0 + \fint_{\partial B'(y_0, \rho)} f 
< \delta + \fint_{\partial B'(y_0, \rho)} f = f(y_0)
\end{equation}
The main idea of the proof is to compare the $(n-1)$-dimensional Hausdorff
measure of $\partial\{v_{k_j}>0\}$ inside 
the cylinder $B'(y_0,\rho)\times(-1,1)$ to that of the graph of $\sigma_{k_j}g$ inside 
the same cylinder to obtain a contradiction with 
an estimate on the size of the area enclosed by these 2 surfaces. 
We introduce some new definitions.

Let $Z=B'(y_0,\rho)\times\R$ be the infinite cylinder. For $\phi$ defined
on $\R^{n-1}$ define
\begin{eqnarray}\label{apb-eqn5.30}
Z^+(\phi) & = & \{(y,h)\in Z: h>\phi(y)\} \\
Z^-(\phi) & = & \{(y,h)\in Z: h<\phi(y)\} \nonumber \\
Z^0(\phi) & = & \{(y,h)\in Z: h=\phi(y)\}. \nonumber
\end{eqnarray}
We left some room in the choice of $\varepsilon_0$ above, and this way we can assume
that 
\begin{equation}\label{apb-eqn5.31}
\H^{n-1}(Z^0(\sigma_j g)\cap \partial\{v_j >0\})=0.
\end{equation}
because the set of values of $\varepsilon_0$ for which this fails is at most countable.

Let us make three claims, then show how to combine them to get the desired contradiction,
then discuss their proofs.

\begin{Claim}\label{apb-claim1}
\begin{equation}\label{apb-eqn5.32}
\H^{n-1}(Z^+(\sigma_j g)\cap\partial\{v_j >0\})\le  \frac{1 +\tau_j}{1-\tau_j}\ \H^{n-1}
(Z^0(\sigma_j g)\cap\{v_j >0\}).
\end{equation}
\end{Claim}

\begin{Claim}\label{apb-claim2}
Let $E_j =\{v_j >0\}\cup Z^-(\sigma_j g)$. Then $E_j$ is a set of locally finite
perimeter and
\begin{equation}\label{apb-eqn5.33}
\H^{n-1}(Z\cap\partial^\ast E_j)
\le \H^{n-1}(\partial\{v_j >0\}\cap Z^+(\sigma_j g))+\H^{n-1}(\{v_j =0\}\cap Z^0(\sigma_j g)).
\end{equation}
Here $\partial^\ast E_j$ denotes the reduced boundary of $E_j$.
\end{Claim}

\begin{Claim}\label{apb-claim3}
\begin{equation}\label{apb-eqn5.34}
\H^{n-1}(Z\cap\partial^\ast E_j ) \ge \H^{n-1}(Z^0(\sigma_j g))+C\sigma^2_j \rho^{n-1}
\end{equation}
where $C>0$ is a constant independent of $j$. 
\end{Claim}

Before addressing the claims we can combine them to get the desired contradiction.
We use (\ref{apb-eqn5.34}), (\ref{apb-eqn5.33}), (\ref{apb-eqn5.32}), 
and the harmonicity of $g$ to prove that for $j$ large, 
since $\|\nabla g\|_{L^2(B)}$ is bounded and $\sigma_j\to 0$,
\begin{eqnarray}\label{apb-eqn5.35}
\H^{n-1}(Z^0(\sigma_j g))+C\sigma^2_j \rho^{n-1} 
& \le & \H^{n-1}(Z\cap\partial^\ast
E_j)\nonumber \\ 
& \le & \H^{n-1}(\partial\{v_j >0\}\cap Z^+(\sigma_j g)) 
+ \H^{n-1}(\{v_j =0\}\cap Z^0(\sigma_j g)) \nonumber \\ 
& \le &  \frac{1 +\tau_j}{1-\tau_j}\ \H^{n-1}(Z^0(\sigma_j g)\cap \{v_j >0\}) +
\H^{n-1}(\{v_j =0\}\cap Z^0(\sigma_j g))  
\nonumber\\
& \le & \frac{1 +\tau_j}{1-\tau_j}\  \H^{n-1}(Z^0(\s_j g))  
\le  (1+ 4\tau_j) \H^{n-1}(Z^0(\s_j g))  
\\
& \le &  \H^{n-1}(Z^0(\s_j g)) + C\tau_j \rho^{n-1}.
\nonumber
\end{eqnarray}
Note that \eqref{apb-eqn5.35} yields $1\le C\sigma_j ^{-2}\tau_j$, 
which is a contradiction since we are assuming $\sigma_j ^{-2}\tau_j \to 0$ and $j \to 0$.

Claim 2 is straightforward and anyway does not use normal derivatives. 
The proof of Claim~3 here is identical to the corresponding one in \cite{AC} or \cite{KT2}.
To verify Claim 1, notice that (\ref{apb-eqn37}) and then (\ref{apb-eqn4.37}) imply that for 
 $\phi\in C^\infty_c(B(0,1))$ such that $\phi\ge 0$,
\begin{equation}\label{apb-eqn66}
(1-\tau_j)  \int_{\partial^\ast\{v_j>0\}}\phi d\H^{n-1}
\le \int_{\partial^\ast\{v_j>0\}} \phi(x) k_j(x) d\H^{n-1}(x)
= - \int_{\{v_j>0\}}\langle\nabla v_j, \nabla\phi\rangle.
\end{equation}

Take an increasing sequence of mappings $\phi_k \in C^\infty_c(B(0,1))$
that converges to $\1_{Z^+(\s_jg) \cap B(0,1)}$; then 
\begin{equation} \label{a8.59}
\lim_{k \to +\infty}\int_{\partial^\ast\{v_j>0\}}\phi_k d\H^{n-1} 
= \int_{\partial^\ast\{v_j>0\}\cap Z^+(\s_jg)}d\H^{n-1}
= \H^{n-1}(\p^\ast\{v_j>0\} \cap Z^+(\s_jg)),
\end{equation}
for instance by Beppo-Levi and because $\partial^\ast\{v_j>0\} \cap Z$ 
does not get high enough to meet $Z \cap Z^+(\s_jg) \sm B(0,1)$.
Next $Z^+(\s_jg)$ is a an open set with finite perimeter, whose boundary is composed of a vertical
piece of $\d Z$, plus two roughly horizontal smooth pieces (a piece of $\d B(0,1)$
above and a piece of the graph of $\sigma_j g$ below).
Denote by $\nu$ the outward unit normal for this domain; we want to show that
\begin{equation} \label{a8.60}
\lim_{k \to +\infty} \int_{\{v_j>0\}}\langle\nabla v_j, \nabla\phi_k\rangle
= - \int_{\p Z^+(\s_j g) \cap \{v_j>0\}} \langle\nabla v_j, \nu\rangle d\H^{n-1}
\end{equation}
By \eqref{apb-eqn5.27}, $f < g$ on $\d B'(y_0,z)$, so for $j$
large $Z^+(\s_jg)$ lies strictly above $\overline{\{v_j >0\}}$ in a neighborhood
of $\d Z$. This neighborhood does not contribute to either side of \eqref{a8.60},
so we only consider the rest of $Z$, where all the contributions come from
a small region on and slightly above the graph of $\sigma_j g$.

In this region, $\nabla\phi_k(x) dx$ converges weakly to 
$- \nu \H^{n-1}_{\vert \d Z^+(\s_jg)}$, with no need to disturb sets 
of finite perimeters here because $Z^+(\s_jg)$ is the region above a smooth 
graph. But maybe $\nabla v_j$ varies a little bit too wildly for this weak convergence,
so we'll cut the region in two.

Recall  from \eqref{apb-eqn5.31} that 
$\H^{n-1}(Z^0(\sigma_j g)\cap \partial\{v_j >0\})=0$. 
Let $\varepsilon > 0$ be given; by regularity of the restriction of $\H^{n-1}$
to $Z^0(\sigma_j g)$ we can choose $\delta >0$ so that if 
$H_\delta$ denotes the $\delta$-neighborhood of $\partial\{v_j >0\}$,
$\H^{n-1}(Z^0(\sigma_j g)\cap H_{2\delta}) < \varepsilon$.
Recall from \eqref{apb-eqn37} that $\nabla v_j$ is bounded; then with a small covering
of $Z^0(\sigma_j g)\cap H_{\delta}$ by balls of radius $\delta/10$, we can see that
for $k$ large the contribution of $H_{\delta}$ to both sides of \eqref{a8.60}
are less than $C\varepsilon$. 

In the remaining region $\{ v_j > 0 \} \sm H_\delta$, $\nabla v_j$ is smooth
(because $v_j$ is harmonic in $\{ v_j > 0 \}$), we can use the weak convergence
of $\nabla\phi_k(x) dx$ to $- \nu \H^{n-1}_{\vert \d Z^+(\s_jg)}$ to construct a region
$X_\delta$, that contains $\{ v_j > 0 \} \sm H_\delta$, and where the analogue of
\eqref{a8.60} holds. Then \eqref{a8.60} itself follows by letting $\delta$
tend to $0$.

We may now return to \eqref{apb-eqn66}, take a limit, and we get that
\begin{eqnarray}\label{apb-eqn68}
(1-\tau_j) \H^{n-1}(\p^\ast\{v_j>0\} \cap Z^+(\s_jg))  
&\leq& - \int_{\p Z^+(\s_j g) \cap \{v_j>0\}} \langle\nabla v_j, \nu\rangle d\H^{n-1}
\nonumber \\
&\,& \hskip-4cm
\leq \int_{Z^0 (\s_jg) \cap \{ v_j>0 \}}|\nabla v_j| d\H^{n-1}
\leq (1 +\tau_j)\ \H^{n-1}(Z^0(\s_jg)\cap \{v_j>0\})
\end{eqnarray}
by \eqref{a8.59}, \eqref{a8.60}, the fact that $\{ v_j>0 \}$ does not meet
$\d Z \cap \d Z^+(\s_j g)$, and \eqref{apb-eqn37}.
This concludes the proof of \eqref{apb-eqn5.32}; \eqref{apb-eqn5.35}
and Lemma \ref{apb-lem4} follow.

\end{proof}

\ms
The proof of the fact that $f$ is Lipschitz will rely  on the following lemma, 
which claims that on average, the averages of $f$ converge to $f$ faster
than linearly. We denote these averages by
\begin{equation}\label{eqn5.62}
f_{y,r}=\fint_{\lower7pt\hbox{$\scriptstyle{\partial B'(y,r)}$}}fd\H^{n-1},
\ \text{ with } \d B'(y,r) = \d B(y,r) \cap [\R^{n-1}\times\{0\}].
\end{equation}

\begin{lemma}{\bf(Lemma 7.6 \cite{AC} or Lemma 0.9 \cite{KT2})}\label{apb-lem5}
There is a constant $C=C(n)>0$ such that for 
$y\in B'_{1/4} = B(0,\frac{1}{4})\cap\R^{n-1}\times\{0\}$
\begin{equation}\label{eqn5.61}
0 \le \int^{\frac{1}{8}}_0 (f_{y,r}-f(y)) \, \frac{dr}{r^2} \le C.
\end{equation}
\end{lemma}

\begin{proof}
Let $y\in B'_{1/4}$ be given; since $v_j\in F(\s_j, \s_j;\tau_j)$ in $B(0,1)$,
we also get that $v_j\in F(8\s_j,8\s_j;\tau_j)$ in 
$B\left(\overline y_j, \frac{1}{2}\right)$, where $\overline y = (y, \s_jf^+_j(y))$.
Notice that we choose the last coordonate of $\overline y$ so that
$\overline y_j$ lies in $\d \{ v_j > 0 \}$.
We shall prove the lemma in the special case when $y=0$, so that we can refer
to the fact that $v_j\in F(\s_j, \s_j;\tau_j)$ directly, but with the observation
above, the proof would also work for general points $y$ (with slightly worse constants).
We also have the additional advantage that since $f(0) = 0$, we do not have 
to subtract the limit $f(y)$ of the $f^+_j(y))$.
With this reduction it is enough to prove that
\begin{equation}\label{apb-eqn69}
0\le\int^{\frac{1}{8}}_0  
\frac{1}{r^2} \fint_{\lower7pt\hbox{$\scriptstyle{\p B'_r}$}} fd\H^{n-1}\le C,
\end{equation}
where $B'_r=B'(0,r)$ and $C$ only depends on $n$.
By Lemma \ref{apb-lem4}, $f$ is subharmonic in $B'$. 
Thus for $r\in \left(0, \frac{1}{2}\right)$, $f(0)\le \fint_{\p B'_r}fd\H^{n-1}$, 
which proves the first inequality in (\ref{apb-eqn69}).

Let $h > 0$ be small, and restrict to $j$ large, so that $2 \s_j < h$. 
Set $B = B\left(0; \frac{1}{4}\right)$, and let $G_h$ denote the Green function of $B \cap\{x_{n}<0\}$ with pole $-he_{n}$.
Using a reflection argument we know that $G_h$ can be extended to be a
smooth function on $B\bs\{\pm he_{n}\}$,
with $G_h(x, x_{n})=-G_h(x, -x_{n})$ for $x_{n}>0$.

For $j$ large let $G^j_h(x)=G_h(x+\s_je_{n})$, which is defined on
$B^j = B - \s_je_{n}$, minus the two poles $-\s_je_{n} \pm he_{n}$.
In the definition of $B^j$, we may always replace the radius $1/4$
with something slightly different (and the estimates would be the same).
So, avoiding an at most countable set of radii, we may assume that for $j$ large, 
\begin{equation} \label{a8.65}
\H^{n-1}(\p B^j \cap\p^*\{v_j>0\})=0.
\end{equation}
We claim that by Green's formula 
(applied on the domain $B^j \cap\{v_j>0\}$, minus a tiny ball 
centered at the pole $-(h+\s_j) e_{n}$),
\begin{eqnarray}\label{apb-eqn70}
-\int_{B^j \cap\{v_j>0\}} \langle\nabla v_j, \nabla G^j_h \rangle 
& = & \int_{\p^\ast [B^j \cap\{v_j>0\}]}
v_j\p_\nu G^j_h d\H^{n-1} - v_j(-(h+\s_j)e_{n})
\nonumber \\
&=&  \int_{\p B^j \cap \{v_j>0\}} v_j\p_\nu G^j_h d\H^{n-1}
- v_j(-(h+\s_j) e_{n}),
\end{eqnarray}
where $\p_\nu G^{j}_h=\langle\nabla G^{j}_h,\nu\rangle$, and $\nu$
denotes the inward pointing unit normal.  
For the first line, the overanxious reader may be worried about the joint regularity 
of the boundary and $v_j$, but the part of boundary where $\nabla v_j$ may be wild 
is near $B_j \cap \d^\ast \{v_j>0\}$, where $G_h^j$ is smooth and $v_j$ is Lipschitz
(by \eqref{apb-eqn37}); we may need a small limiting argument here, but an argument a little similar to the rapid justification of \eqref{a8.60}, where you integrate 
against a smoothed out version of $v_j$ and go to the limit, will do the job.
Notice that $G_h^j$ is smooth away from the pole, so does not create trouble, and
also the contribution of $B_j \cap \d^\ast \{v_j>0\}$ to the right-hand side of the
first line disappears, because $v_j$ (is Lipschitz and) vanishes on that part of the
boundary. A different Green-type computation yields
\begin{equation} \label{apb-eqn70aa}
-\int_{B^j \cap\{v_j>0\}} \langle\nabla v_j, \nabla G^j_h \rangle = 
\int_{\p\{v_j>0\}\cap B^j} G^j_h k_j d\H^{n-1};
\end{equation}
if $G_h^j$ were a smooth, compactly supported function in $B^j$, this would 
be \eqref{apb-eqn4.37}. Now $G_h^j$ has a singularity at the pole $-(h+\s_j) e_{n}$,
but where $\nabla G_h^j$ is locally integrable, and since $\nabla v_j$ is smooth
near the pole, a small approximation allows one to get rid of the singularity.
Similarly, $G_h^j$ vanishes nicely on $\d B^j$, and we can approximate it by
smooth compactly supported functions (because on $\d B^j \sm \p\{v_j>0\}$,
$\nabla v_j$ is smooth, and by \eqref{a8.65} the contribution near 
$\d B^j \cap \p\{v_j>0\}$ can be estimated as near \eqref{a8.60}).

Now \eqref{apb-eqn70} and \eqref{apb-eqn70aa} yield
\begin{equation}\label{apb-eqn70A}
\int_{\p B^j \cap\{v_j>0\}} v_j\p_\nu G^j_h-v_j(-(h+\s_j)e_{n})
- \int_{\p\{v_j>0\} \cap B^j} k_j G^j_h d\H^{n-1}=0
\end{equation}
A new application of Green's formula, as in the first line of \eqref{apb-eqn70}
but with $v_j$ replaced by $x_n$, yields
\begin{equation}\label{apb-eqn70a}
-\int_{B^j \cap\{v_j>0\}} \langle \nabla x_n , \nabla G^j_h \rangle 
= \int_{\p^\ast [B^j \cap\{v_j>0\}]} x_n\p_\nu G^j_h d\H^{n-1} + (h+\s_j)
\end{equation}
But $\langle \nabla x_n , \nabla G^j_h \rangle = \Div(x_n G^j_h)$, so by Green again
\begin{eqnarray}\label{apb-eqn70b}
-\int_{B^j \cap\{v_j>0\}} \langle \nabla x_n , \nabla G^j_h \rangle 
&=& \int_{\p^\ast [B^j \cap\{v_j>0\}]} G^j_h \langle e_n, \nu \rangle d\H^{n-1}
\nonumber\\
&=& \int_{B^j \cap \p^\ast \{v_j>0\}} G^j_h \langle e_n, \nu \rangle d\H^{n-1}
\end{eqnarray}
where $\nu$ denotes the inward pointing normal, and because $G^j_h$
vanishes on $\d B^j$.
We cut the boundary in \eqref{apb-eqn70a} into two pieces, compare with 
\eqref{apb-eqn70b}, and get that
\begin{equation}\label{apb-eqn72}
\int_{B^j\cap\p^\ast\{v_j>0\}}\langle G_h^j e_{n}-x_{n} \nabla G^j_h, \nu\rangle 
\, d\H^{n-1}
=  (\s_j+h) + \int_{\p B^j\cap\{v_j>0\}} x_{n}\p_\nu G^j_h\, d\H^{n-1}.
\end{equation}
Let us even write $\nu_j$ for $\nu$, to stress the dependence on $j$. Thus
\begin{equation}\label{apb-eqn72A}
\int_{B^j \cap\p^\ast\{v_j>0\} }x_{n} \p_{\nu_j} G^j_hd\H^{n-1} 
= \int_{B^j \cap\p^\ast\{v_j>0\}}G_h\langle e_{n},\nu_j\rangle d\H^{n-1} 
- (\s_j+h) - \int_{\p B^j \cap\{v_j>0\}}x_{n}\p_\nu G^j_h d\H^{n-1}.
\end{equation}
Dividing (\ref{apb-eqn70A}) by $1-\tau_j$ and subtracting it from 
(\ref{apb-eqn72A}) we obtain
\begin{eqnarray}\label{apb-eqn73}
\int_{B^j \cap \p^\ast\{v_j>0\}}x_{n}\p_{\nu_j}G^j_hd\H^{n-1}
& = & \int_{B^j\cap\p^\ast\{v_j>0\}}\left(\frac{1}{1-\tau_j} k_j
+\langle e_{n}, \nu_j\rangle\right)G^j_h d\H^{n-1}
 \nonumber\\
&\,& \hskip-3cm
- \int_{\p B^j \cap\{v_j>0\}}(x_{n}+\frac{v_j}{1-\tau_j})\p_\nu G^j_h d\H^{n-1}
+ \frac{1}{1-\tau_j}v_j(-(h+\s_j)e_{n}) - (\s_j+h).
\end{eqnarray}
We estimate each term separately. 
Recall that $h>2\s_j$ and $v_j\in F(\s_j, \s_j;\tau_j)$ in $B(0,1)$ in the direction
$e_{n}$. Then $G^j_h \le 0$ on $\p\{v_j>0\}\cap B^j_{\frac{1}{2}}$; this is the
reason why we lowered $B$ to get $B^j$

Moreover, since $k_j\ge 1-\tau_j \ $ $\H^{n-1}$ a.e. on $\p^\ast\{v_j>0\}$ 
(see (\ref{apb-eqn37})), then
$\frac{1}{1-\tau_j}k_j+\langle e_{n},\nu_j\rangle\ge 0$ and 
\begin{equation}\label{apb-eqn74}
\int_{B \cap\p^\ast\{v_j>0\}} \left(\frac{1}{1-\tau_j}k_j+\langle e_{n},
\nu_j\rangle\right)G^j_h  d\H^{n-1}\le 0. 
\end{equation}
Furthermore since $v_j\ge 0$ and $v_j(0)=0$, (\ref{apb-eqn37})
ensures that
\begin{eqnarray}\label{apb-eqn75}
|v_j(-(h+\s_j)e_{n})| & = & |v_j(-(h+\s_j)e_{n})-v_j(0)| \\ 
& \le & \sup_{\{v_j>0\}}|\nabla v_j|(h+\s_j)\le (1+\tau_j)(h+\s_j). \nonumber
\end{eqnarray}
Hence \eqref{apb-eqn75} yields
\begin{equation}\label{apb-eqn76}
\frac{1}{1-\tau_j}v_j(-(h+\s_j)e_{n})-(h+\s_j)\le \frac{2\tau_j}{1-\tau_j}(h+\s_j)
\end{equation}

Recall that $\{v_j>0\}\subset\{x_{n}<\s_j\}$. 
For $x_{n}\le\s_j$ and by (\ref{apb-eqn37}),
\begin{equation}\label{apb-eqn78}
v_j(x,x_{n}) \leq 
|v_j(x,x_{n})-v_j(x,\s_j)| \le (\s_j-x_{n})\sup |\nabla v_j| 
\le (1+\tau_j)
(\s_j-x_{n}),
\end{equation}
which yields for $x_{n}\in[0,\s_j]$
\begin{equation}\label{apb-eqn79}
0\le \frac{v_j(x, x_{n})}{1-\tau_j} +x_{n}
\le \frac{1+\tau_j}{1-\tau_j}(\s_j-x_{n}) +x_n
\le \frac{1+\tau_j}{1-\tau_j}\s_j - \frac{2\tau_j}{1-\tau_j}x_n 
\le \frac{1+\tau_j}{1-\tau_j}\s_j, 
\end{equation}
and for $x_{n}\in[-\s_j,0]$
\begin{equation}\label{apb-eqn80}
-\s_j\le \frac{v_j(x, x_{n})}{1-\tau_j}+x_{n}
\le \frac{1+\tau_j}{1-\tau_j}\s_j - \frac{2\tau_j}{1-\tau_j}x_n 
\le \frac{1+3\tau_j}{1-\tau_j}\s_j. 
 \end{equation}
Since $v_j\in F(\s_j,\s_j;\tau_j)$ in $B(0,1)$ in the direction $e_{n}$
\begin{equation}\label{apb-eqn81}
v_j(x,x_{n})\ge-x_n-\s_j\qquad\mbox{for}\qquad x_{n}\le-\s_j
\end{equation}
(and even for $x_n \leq 0$) and so 
\begin{equation}\label{apb-eqn81A}
 \frac{v_j(x, x_{n})}{1-\tau_j}+x_{n}\ge -\frac{\s_j+x_n}{1-\tau_j} +x_n\ge -\frac{\s_j}{1-\tau_j}-\frac{\tau_j}{1-\tau_j}x_n\ge  -\frac{\s_j}{1-\tau_j}.
\end{equation}
We combine the fact that $\p_\nu G^j_h>0$ on $\d B^j$
(by the Hopf boundary lemma) and (\ref{apb-eqn79}), and (\ref{apb-eqn81A})
to estimate 
\begin{eqnarray}\label{apb-eqn82}
\int_{\p B^j\cap\{v_j>0\}}(x_{n}+\frac{v_j}{1-\tau_j})\p_\nu G^j_h d\H^{n-1}
&=&
\int_{\p B^j \cap\{v_j>0\}\cap\{x_{n}\le 0\}}
(x_{n}+\frac{v_j}{1-\tau_j})\p_\nu G^j_h d\H^{n-1}
\nonumber\\
&\,&+ \int_{\p B^j \cap\{v_j>0\}\cap \{0<x_{n}\le\s_j\}} 
(x_{n}+\frac{v_j}{1-\tau_j})\p_\nu G^j_h d\H^{n-1}
\nonumber \\
&\ge &-\frac{\s_j}{1-\tau_j} \int_{\p B^j\cap\{v_j>0\}\cap\{x_{n}\le 0\}}
\p_\nu G^j_h. 
\end{eqnarray}
Combining (\ref{apb-eqn73}), \eqref{apb-eqn74}, \eqref{apb-eqn76} and (\ref{apb-eqn82}) we obtain
\begin{eqnarray}\label{apb-eqn83}
\limsup_{j\to\infty}\frac{1}{\s_j}
\int_{B^j \cap\p\{v_j>0\}} x_{n}\p_{\nu_j}G^j_hd\H^{n-1} && \\
&&\kern-1in \le   \limsup_{j\to\infty}\frac{1}{1-\tau_j}
\int_{\p B^j \cap\{v_j>0\}\cap\{x_{n}\le 0\}}\p_\nu G^j_h d\H^{n-1} 
\nonumber \\
&&\kern-1in \le \int_{\p B \cap\{x_{n}\le 0\}}\p_\nu G_h d\H^{n-1}\le Ch.
\nonumber
\end{eqnarray}
The last inequality was obtained by applying the comparison principle for
non-negative harmonic function in the domain
$D=B_{\frac{1}{2}}\cap\{x_{n}\le 0\}$ to the harmonic measure of $D$ and
the function $s(x, x_{n})=-x_{n}$ at the point $-he_{n}$ (see
\cite[Lemma 4.10]{JK}).

The rest of the proof is exactly as the one presented in \cite{KT2}, so we just
describe the scheme. 

 Notice that since $v_j\in F(\s_j,\s_j;\tau_j)$ then
$\chi_{\{v_j>0\}}\mathop{\longrightarrow}\limits_{j\to\infty}
\chi_{\{x_{n}\le 0\}}$ in $L^1(B(0,1))$ and
$\partial\{v_j>0\}\to\{x_{n}=0\}$ in the Hausdorff distance sense
uniformly on compact subsets. Moreover since $f^+_j$ and $f^-_j$ converge
uniformly to $f$ on compact sets and $\nabla G^j_h$ converges to $\nabla
G_h$ smoothly away from $\pm he_{n}$ we have that
\begin{equation}\label{apb-eqn84}
\sup_{(x, x_{n})\in \p^\ast\{v_j>0\}\cap B^j}
\left|\frac{x_{n}}{\s_j} \nabla G^j_h (x,x_{n}) - f(x)\nabla G_h(x,0)
\right| \mathop{\longrightarrow}\limits_{j\to\infty} 0.
\end{equation}
Thus combining (\ref{apb-eqn83}) and (\ref{apb-eqn84}) we obtain that
\begin{equation}\label{apb-eqn85}
\frac{1}{h} \int_{B'}f(x)\nabla_{-e_{n}}G_h(x,0)dx\le C.
\end{equation}
Note that $\nabla_{-e_{n}}G_h \big|_{x_{n}=0} = -\frac{\p G_h}{\p
x_{n}}\big|_{x_{n}=0}$ is radially symmetric on $B'$. 
Let $g_h(r)=g_h(|x|)=-\frac{\p G_h}{\p x_{n}}(x,0)$ for $x=r\theta$ and 
$\theta\in \SS^{n-1}$. With this notation (\ref{apb-eqn85}) becomes
\begin{eqnarray}\label{apb-eqn86}
\frac{1}{h}\int_{B'} f(x) g_h(|x|)dx & = & \frac{1}{h}
\int^{\frac{1}{2}}_0 r^{n-1}g_h(r)\int_{\SS^{n-1}} f(r\theta)d\theta dr \\
& = & \frac{\s_{n-1}}{h} \int^{\frac{1}{2}}_0 r^{n-1} g_h(r) \fint_{\lower7pt\hbox{$\scriptstyle{\p B'_r}$}}
fd\H^{n-1}dr\le C.\nonumber
\end{eqnarray}
Comparing $g_h(r)$ with the Poisson kernel of $\R^n$ with pole at $-he_{n}$,
$P_h(r)$ (see \cite[Lemma~4.3]{KT3}), and using once more the comparison
principle for non-negative harmonic functions on $B^-$
(\cite[Lemma 4.10]{JK}) we obtain
\begin{equation}\label{apb-eqn87}
\frac{g_h(r)}{P_h(r)} = \lim_{x\to (r\theta,0)} \frac{G_h(x)}{G^\infty_h(x)}\ge
C_n \frac{G_h(A_h)}{G^\infty_h(A_h)}, 
\end{equation}
here $G^\infty_h$ denotes the Green's function of $\R^n$ with pole at
$-he_{n}$; and $A_h=-\frac{h}{64} e_{n}$. Since $G^\infty_h(A_h)\le
\frac{C_n}{h^{n-1}}$ and $G_h(A_h)\ge \frac{C_n}{h^{n-1}}$, (\ref{apb-eqn87})
yields
\begin{equation}\label{apb-eqn88}\
g_h(r)\ge \frac{C_n h}{(r^2+h^2)^{\frac{(n+1)}{2}}}.
\end{equation}
Combining (\ref{apb-eqn86}) and (\ref{apb-eqn88}) we obtain
\begin{equation}\label{apb-eqn89}
\int^{\frac{1}{2}}_0
\frac{r^{n-1}}{(r^2+h^2)^{\frac{n}{2}}}\left(\fint_{\lower7pt\hbox{$\scriptstyle{\p B'_r}$}} f(x)dx\right)
dr\le C,
\end{equation}
here $C$ only depends on $n$.
Letting $h$ tend to $0$ we conclude that (\ref{apb-eqn69}) holds.
\end{proof}

\begin{lemma}{\bf(Lemma 7.7 \cite{AC} or Lemma 0.10 \cite{KT2})}\label{apb-lem6}
The function $f$ introduced in Lemma \ref{apb-lem3} is Lipschitz on
$\overline B'_{\frac{1}{16}}$
with a Lipschitz constant that only depends on $n$.
\end{lemma}

\begin{lemma}{\bf(Lemma 7.8 \cite{AC} or Lemma 0.11 \cite{KT2})}\label{apb-lem7}
Let $f$ be the function introduced in Lemma \ref{apb-lem3}. There exists a
large constant $C=C(n)>0$ such that for any given $\theta\in(0,1)$ there
exist $\eta=\eta(\theta)>0$ and $l\in\R^n\times\{0\}$ with $|l|\le C$ so
that
\begin{equation}\label{apb-eqn97}
f(y)\le \langle l, y\rangle + \frac{\theta}{2}\eta\quad\mbox{for}\quad y\in
B'_\eta.
\end{equation}
\end{lemma}

{\bf Contradiction in the proof of Lemma \ref{apb-lem2}}: 
Recall that by assuming that the statement in Lemma \ref{apb-lem2} was false, 
we were able to construct sequences of functions $\{v_j\}$ and $\{k_j\}$ 
satisfying (\ref{apb-eqn35})- \eqref{a8.44}.
Using the functions $\{v_j\}$ we constructed 
sequences of functions $\{f^+_j\}$ and $\{f^-_j\}$ defined in $B'$ (see 
(\ref{apb-eqn38}) and (\ref{apb-eqn39})). The function $f$ introduced in
Lemma~\ref{apb-lem3}, and defined in $B'$ is a limit of subsequences of
$\{f^+_j\}$ and $\{f^-_j\}$ (which we relabeled). In Corollary \ref{apb-cor1}, and Lemmas \ref{apb-lem4}, \ref{apb-lem5},
\ref{apb-lem6} and \ref{apb-lem7} we studied the properties of $f$. We now
combine all this information about $f$ to produce a contradiction. By
Corollary \ref{apb-cor1}, $f^+_j\mathop{\longrightarrow}\limits_{j\to\infty}
f$ uniformly on compact subsets of $B'$. Therefore Lemma \ref{apb-lem7}
yields that for every $\theta\in(0,1)$ there exists $\eta>0$ so that for $j$
large enough
\begin{equation}\label{apb-eqn104}
f^+_j(y)\le \langle l,y\rangle+\theta\eta\quad\mbox{for}\quad y\in B'_\eta.
\end{equation}
This is how we define $\eta = \eta(\theta)$, independently of the sequence itself,
as promised.
Hence by the definition (\ref{apb-eqn38}) 
\begin{equation}\label{apb-eqn105}
v_j(x)=0\mbox{ for }x=(x, x_{n})\in B(0,\eta)\mbox{ with }x_{n}>\s_j\langle l,x\rangle + \theta\eta\s_j.
\end{equation}
Let $\wt\nu=(1+\s^2_j|l|^2)^{-\frac{1}{2}}(-\s_jl, 1)$,
and notice that $\wt\nu$ satisfies \eqref{apb-eqn7}; in addition,
(\ref{apb-eqn105}) implies that
\begin{equation}\label{apb-eqn106}
v_j(x)=0\mbox{ for }x\in B(0,\eta)\mbox{ and } 
\langle x, \wt\nu\rangle 
\ge \frac{\theta\eta\s_j}{(1+\s^2_j|l|^2)^{\frac{1}{2}}}
\ge 2\theta\eta\s_j,
\end{equation}
for $j$ large enough. Note that (\ref{apb-eqn106}) says that
$v_j\in F(2\theta\s_j,1;\tau_j)$; this contradicts 
our contradiction assumption that \eqref{apb-eqn36} fails for
all $\wt\nu$ that satisfies \eqref{apb-eqn7};
Lemma \ref{apb-lem2} follows.
\end{proof}

\end{document}